%% file: main.tex
\documentclass[a4paper,12pt,reqno]{amsart}

\usepackage[utf8]{inputenc}
\usepackage[top=25truemm,bottom=25truemm,left=20truemm,right=20truemm]{geometry}

\input{packages}
\input{newcommands}

\title{Relativized universal algebra via partial Horn logic}
\author{Yuto Kawase}
\address{Research Institute for Mathematical Sciences, Kyoto University, Kyoto 606-8502, Japan}
\email{ykawase@kurims.kyoto-u.ac.jp}
\date{\today}
\thanks{This is a revised version of the author's master's thesis submitted to Research Institute for Mathematical Sciences on January 15, 2024. The author wishes to thank his supervisor, Masahito Hasegawa, for his support, and Hayato Nasu for valuable corrections. He also thanks Ji\v{r}\'{\i} Rosick\'{y} and an anonymous referee of an earlier version of this paper for helpful comments, especially regarding relevant literature.}

\begin{document}
\begin{abstract}
    Algebraic theories, sometimes called equational theories, are syntactic notions given by finitary operations and equations, such as monoids, groups, and rings.
    There is a well-known category-theoretic treatment of them that algebraic theories are equivalent to finitary monads on $\Set$.
    In this paper, using partial Horn theories, we syntactically generalize such an equivalence to arbitrary locally presentable categories from $\Set$;
    the corresponding algebraic concepts relative to locally presentable categories are called relative algebraic theories.
    Finally, we give a framework for universal algebra relative to locally presentable categories by generalizing Birkhoff's variety theorem.
\end{abstract}

\maketitle

\tableofcontents

\section{Introduction}
Algebraic theories, also called equational theories, are algebraic concepts given by several finitary operations and equations, which include monoids, groups, rings, etc.
The first unified investigation of algebraic theories was taken by Birkhoff \cite{birkhoff1935structure}, which marked the beginning of \emph{universal algebra}.
There is an elegant category-theoretic treatment of algebraic theory due to Linton, who implicitly showed in his paper \cite{linton1969outline} that algebraic theories correspond to finitary monads on $\Set$, the category of sets.
More precisely, we have an equivalence of categories
\begin{equation}\label{eq:equiv_th_and_mndf}
    \Th \simeq \Mndf(\Set),
\end{equation}
where $\Th$ and $\Mndf(\Set)$ denote the categories of algebraic theories and finitary monads on $\Set$, respectively.
It should be emphasized that the equivalence \cref{eq:equiv_th_and_mndf} preserves the concept of models; that is, it commutes (up to natural isomorphism) with the functors sending each algebraic theory to the category of its algebras and sending each finitary monad to its Eilenberg--Moore category.

Analogy brings us to the relativization of algebraic concepts:
we should regard a finitary monad on a general category $\A$ as an algebraic theory relative to $\A$.
Indeed, it is well-known that finitary monads on $\Set^S$, the functor category from a discrete category $S$ to $\Set$, correspond to $S$-sorted algebraic theories.
Furthermore, the base category can be generalized to an arbitrary locally presentable category rather than $\Set^S$, and there are several general frameworks for algebraic theories based on such monads \cite{kelly1993adjunctions,berger2012monads,bourke2019monads}.
Some authors also described what the algebra for such extended algebraic theories is by giving a practical presentation for monads \cite{rosicky2021metric,lucyshyn2023diagram}.
However, these presentations are not so syntactic in the sense that neither logical terms nor logical formulas appear explicitly, while Birkhoff's classical universal algebra is built upon logical symbols.
On the other hand, in several special cases, there are more syntactic approaches.
One of them is taken in \cite{ford2021monads}:
by giving ``algebraic theories'' relative to a category of models of relational Horn theories in a syntactic way, the authors describe monads on that category.
Another one is taken in \cite{adamek2021finitary}:
by giving ``algebraic theories'' relative to $\Pos$, the category of posets, the authors syntactically describe finitary monads on $\Pos$.

We aim to give a framework for universal algebra relative to locally presentable categories by describing monads on an arbitrary locally presentable category by logical terms and formulas.
To do so, we will introduce \emph{relative algebraic theories}, which are algebraic concepts relative to locally presentable categories and will be a syntactic description of such monads.
In an ordinary algebraic theory, each operator $\omega$ has a natural number $n$ called an \emph{arity}.
Here, $\omega$ represents an ($n$-ary) operator taking $n$ elements $x_1,\dots,x_n$ as input and receiving one element $\omega(x_1,\dots,x_n)$ as output.
In contrast, in our relative algebraic theory, each operator $\omega$ has a logical formula $\phi(x_1,\dots,x_n)$ as its arity.
Then, $\omega$ represents an $n$-ary partial operator that returns an output only if the formula $\phi(x_1,\dots,x_n)$ is satisfied.
A toy example is the subtraction of natural numbers:
$x-y$ is defined only if $y\le x$ holds; hence, the subtraction ``$-$'' on $\bN$ is a binary partial operator with an arity $y\le x$.
Furthermore, we treat not only ordinary finitary operators but also infinitary operators.
However, we assume that the arity of operators is uniformly bounded by an infinite regular cardinal $\lambda$.
This restriction is intended to focus our interest on \emph{accessible monads}, i.e., monads preserving $\lambda$-filtered colimits for some $\lambda$.

Our relative algebraic theories are based on \emph{partial Horn theories} \cite{palmgren2007partial}, which are logical theories characterizing locally presentable categories; that is, a category is locally presentable if and only if it is equivalent to the category of models of some partial Horn theory.
More precisely, for a partial Horn theory $\bS$, we will define \emph{$\bS$-relative $\lambda$-ary algebraic theories} and show that they are equivalent to $\lambda$-ary monads on $\PMod\bS$, i.e., monads on the category of models of $\bS$ that preserve $\lambda$-filtered colimits.
Consequently, we have an equivalence of categories
\begin{equation}\label{eq:equiv_rat_and_mnd_introduction}
    \Thl^\bS \simeq \Mndl(\PMod\bS),
\end{equation}
where $\Thl^\bS$ and $\Mndl(\PMod\bS)$ denote the categories of $\bS$-relative $\lambda$-ary algebraic theories and $\lambda$-ary monads on $\PMod\bS$.
The equivalence \cref{eq:equiv_rat_and_mnd_introduction} is one of our main results, which subsumes the classical theory--monad equivalence \cref{eq:equiv_th_and_mndf}.
Since every locally presentable category is written as $\PMod\bS$ for some partial Horn theory $\bS$, we have characterized monads on an arbitrary locally presentable category.
In addition, this allows us to call $\bS$-relative algebras $\A$-relative algebras to the extent that there is no confusion when an arbitrary locally presentable category $\A$ is given and $\A\simeq \PMod\bS$ holds.

Although Freyd's \emph{essentially algebraic theories} also characterize locally presentable categories, we use partial Horn theories rather than them since we want to treat relation symbols explicitly, such as the order $\le$.

Birkhoff's variety theorem is one of the fundamental results in ordinary algebraic theories, which characterizes equationally definable full subcategories via closure properties.
There is a many-sorted version (in our language, the $\Set^S$-relative version) of Birkhoff's theorem:
\begin{theorem}[\cite{adamek2012birkhoffs}]\label{thm:many_sorted_birkhoff}
    Let $(\Omega, E)$ be a (many-sorted) algebraic theory, where $\Omega,E$ is the sets of operators and equations.
    Then, a full subcategory $\E\subseteq\Alg(\Omega,E)$ of the category of algebras is definable by equations if and only if it is closed under:
    \begin{itemize}[nosep]
        \item products,
        \item subobjects,
        \item quotients,
        \item filtered colimits.
    \end{itemize}
\end{theorem}\noindent
Note that this theorem contains Birkhoff's original one as the single-sorted case.
In this paper, we generalize the above theorem to our relative algebraic theories (\cref{thm:birkhoff_for_rat}).
This generalization is another main result, which subsumes not only \cref{thm:many_sorted_birkhoff} but also Birkhoff's theorem for quasi-varieties \cite[3.22 Theorem]{adamek1994locally}.
In this generalization, we will replace subobjects with ``$\Sigma$-closed subobjects'' and quotients with ``$U$-retracts.''

The classical Birkhoff's theorem (\cref{thm:many_sorted_birkhoff}) does not depend on syntax because quotients can be characterized by regular epimorphisms, and hence all the conditions of \cref{thm:many_sorted_birkhoff} are purely category-theoretic.
In contrast, our theorem depends on syntax because the concept of ``$\Sigma$-closed subobjects'' depends on the choice of syntax for the base category $\A$ (\cref{eg:birkhoff_thm_depends_on_syntax}).
This is a notable feature of our generalization of Birkhoff's theorem, and illuminates that our theorem is partial Horn theoretic rather than category-theoretic.

Even though closure under filtered colimits is required in the many-sorted case (\cref{thm:many_sorted_birkhoff}) and our generalized one, it is not required in Birkhoff's original one \cite{birkhoff1935structure}.
In other words, closure under filtered colimits can be eliminated from the $\Set$-relative case, which is expressed by saying that \emph{filtered colimit elimination} holds.
It is known that filtered colimit elimination also holds in some other cases:
finite-sorted algebras ($\Set^n$-relative algebras) \cite{adamek2012birkhoffs}, ordered algebras ($\Pos$-relative algebras) \cite{bloom1976varieties}, metric algebras ($\gMet$-relative algebras) \cite{hino2016varieties}.

In this paper, we will give a sufficient condition on a locally presentable category $\A$ for which filtered colimit elimination holds in $\A$-relative algebras.
Our sufficient condition on $\A$ looks like a noetherian condition: $\A$ has no strictly ascending sequence (\cref{def:ACC_for_cat}), usually called the ascending chain condition (ACC).
We will show that if a locally presentable category $\A$ satisfies ACC, then filtered colimit elimination holds in Birkhoff's theorem for $\A$-relative algebras, which is our final main result.

\vspace{1em}
\paragraph{\textbf{Outline}}
In \cref{section:rat}, we introduce relative algebraic theories via partial Horn theories with several examples.
We also discuss morphisms of relative algebraic theories.

In \cref{section:birkhoff}, we introduce \emph{(dense, closed--mono)-factorization systems} on locally presentable categories.
Using such factorization systems, we show Birkhoff's variety theorem for partial Horn theories (\cref{thm:birkhoff_partialHorn}) and give Birkhoff's theorem for relative algebraic theories (\cref{thm:birkhoff_for_rat}) as its corollary.

In \cref{section:monad}, we discuss the relationship between our relative algebraic theories and accessible monads on locally presentable categories.
For a locally $\lambda$-presentable category $\A$, we obtain an equivalence between the category of $\A$-relative $\lambda$-ary algebraic theories and the categories of $\lambda$-ary monads on $\A$ (\cref{thm:equiv_between_monad_and_rat}).

\cref{section:elimination} is devoted to \emph{filtered colimit elimination}.
We first introduce pure quotients, which play an essential role in filtered colimit elimination.
We next introduce the ascending chain condition for a category, which will be a sufficient condition for eliminating closure under filtered colimits from Birkhoff's theorem relative to that category (\cref{thm:filcolim_elim}).

\vspace{1em}
\paragraph{\textbf{Summary}}
For convenience, we summarize the main contributions as follows:
\begin{itemize}
    \item
    Introducing \emph{relative algebraic theories} based on infinitary partial Horn theories, we syntactically describe accessible monads on locally presentable categories (\cref{thm:equiv_between_monad_and_rat}).
    \item
    We show Birkhoff's variety theorem for partial Horn theories (\cref{thm:birkhoff_partialHorn}), which can be applied to our relative algebras.
    \item
    We give a sufficient (and nearly necessary) condition for eliminating closure under filtered colimits from Birkhoff's variety theorem (\cref{thm:filcolim_elim}), which subsumes several cases: finite-sorted algebras, ordered algebras, metric algebras, etc.
\end{itemize}

\begin{remark}
    This paper is based on infinitary partial Horn logic, which is an infinitary extension of (finitary) partial Horn logic \cite{palmgren2007partial}.
    Such infinitary extension has already been considered by some authors \cite{parker2022covariant,tsukada2022linear}, but since there is no cohesive self-contained literature, so we summarize it in \cref{appendix:phl} for convenience.
\end{remark}

\begin{remark}
    Almost all of the content of this paper is based on the author's previous papers \cite{kawase2023birkhoffs,kawase2025filtered}.
    However, only the finitary case is considered in them.
    In this paper, we generalize them to the infinitary case.
\end{remark}

\section{Infinitary relative algebraic theories}\label{section:rat}
We now introduce \emph{relative algebraic theories} based on infinitary partial Horn theories.
Since partial Horn theories can express arbitrary locally presentable categories as their models, our relative algebraic theory is an algebraic concept relative to locally presentable categories.

\subsection{Relative algebras}
We fix an $S$-sorted $\lambda$-ary signature $\Sigma$ and a $\lambda$-ary partial Horn theory $\bS$ over $\Sigma$ throughout this subsection.

\begin{definition}\label{def:relative_alg_theory}\quad
    \begin{enumerate}
        \item
        An \emph{$\bS$-relative ($\lambda$-ary) signature} $\Omega$ is a set $\Omega$ such that for each element $\omega\in\Omega$, a Horn formula $\tup{x}.\phi$ over $\Sigma$ and a sort $s\in S$ are given.
        The Horn formula $\tup{x}.\phi$ is called an \emph{arity} of $\omega$ and written as $\ar(\omega)$.
        The sort $s$ is called a \emph{sort} of $\omega$ and written as $\outsort(\omega)$.
        \item\label{def:relative_alg_theory-2}
        Given an $\bS$-relative ($\lambda$-ary) signature $\Omega$,
        each $\omega\in\Omega$ can be regarded as a function symbol $\omega\colon \sqcap_{i<\alpha}s_i\to s$ if $\outsort(\omega)=s$ and $\ar(\omega)$ is in the context $(x_i\ofsort s_i)_{i<\alpha}$.
        Denote by $\Sigma+\Omega$ the $S$-sorted $\lambda$-ary signature obtained by adding to $\Sigma$ all $\omega\in\Omega$ in this way.
        A $\lambda$-ary Horn sequent $\phi\seq{\tup{x}}\psi$ over $\Sigma+\Omega$ is called an \emph{$\bS$-relative ($\lambda$-ary) judgment} if $\phi$ is over $\Sigma$, i.e., if $\phi$ contains no function symbol derived from $\Omega$.
        \item
        An \emph{$\bS$-relative ($\lambda$-ary) algebraic theory} is a pair $(\Omega,E)$ of an $\bS$-relative $\lambda$-ary signature $\Omega$ and a set $E$ of $\bS$-relative $\lambda$-ary judgments.\qedhere
    \end{enumerate}
\end{definition}

\begin{definition}
    Let $\Omega$ be an $\bS$-relative signature.
    An \emph{$\Omega$-algebra} $\bA$ consists of:
    \begin{itemize}
        \item
        a partial $\bS$-model $A$,
        \item
        for each $\omega\in\Omega$, a map $\intpn{\omega}{\bA}\colon \intpn{\ar(\omega)}{A}\to A_{\outsort(\omega)}$.\qedhere
    \end{itemize}
\end{definition}

An $\Omega$-algebra $\bA$ can be regarded as a partial $(\Sigma+\Omega)$-structure by considering $\intpn{\omega}{\bA}$ as a partial map $\prod_{i<\alpha} A_{s_i}\pto A_{\outsort(\omega)}$, where $\ar(\omega)$ is in the context $(x_i\ofsort s_i)_{i<\alpha}$.
Conversely, a partial $(\Sigma+\Omega)$-structure satisfying all sequents in $\bS$ and the bisequent $\omega(\tup{x})\defined\biseq{\tup{x}}\ar(\omega)$ for each $\omega\in\Omega$ can be regarded as an $\Omega$-algebra.

\begin{definition}
    Let $\Omega$ be an $\bS$-relative signature.
    We say an $\Omega$-algebra $\bA$ satisfies an $\bS$-relative judgment if it is valid in the partial $(\Sigma+\Omega)$-structure $\bA$.
\end{definition}

\begin{notation}\quad
    \begin{enumerate}
        \item
        Given an $\bS$-relative signature $\Omega$, we will denote by $\Alg\Omega$ the category of $\Omega$-algebras and $(\Sigma+\Omega)$-homomorphisms.
        \item
        Given an $\bS$-relative algebraic theory $(\Omega,E)$, we will denote by $\Alg(\Omega,E)$ the full subcategory of $\Alg\Omega$ consisting of all algebras satisfying all $\bS$-relative judgments in $E$.
        An $\Omega$-algebra belonging to $\Alg(\Omega,E)$ is called an \emph{$(\Omega,E)$-algebra}.\qedhere
    \end{enumerate}
\end{notation}

\begin{definition}\label{def:pht_for_rat}
    Let $(\Omega,E)$ be an $\bS$-relative algebraic theory.
    We define the $\lambda$-ary partial Horn theory $\pht{\Omega}{E}$ over $\Sigma+\Omega$ associated with $(\Omega,E)$ as follows:
    \begin{equation*}
        \pht{\Omega}{E}\coloneq\bS\cup\{\omega(\tup{x})\defined\biseq{\tup{x}}\ar(\omega)\}_{\omega\in\Omega}\cup E
    \end{equation*}
    Then, we have $\Alg(\Omega,E)\cong\PMod\pht{\Omega}{E}$.
    In particular, $\Alg(\Omega,E)$ is locally $\lambda$-presentable.
\end{definition}

\subsection{Examples of relative algebraic theories}
We present several examples of relative algebraic theories.
The examples introduced here are classified into the finitary ($\aleph_0$-ary) case and the $\aleph_1$-ary case.

\subsubsection{The finitary case}
\begin{example}[Small categories]\label{eg:rat_smallcat}
    The finitary partial Horn theory $\bS_\quiv$ for quivers is given by:
    \begin{gather*}
        S_\quiv\coloneq\{e,v\},\quad
        \Sigma_\quiv\coloneq\{ s, t\colon e\to v \},\quad
        \bS_\quiv\coloneq\{ \top\seq{f\ofsort e} s(f)\defined\wedge t(f)\defined \}.
    \end{gather*}
    We define an $\bS_\quiv$-relative finitary algebraic theory $(\Omega,E)$ as follows:
    \begin{center}
        $\Omega:$\qquad
        \begin{tabular}{ccc}
            \qquad\qquad & arity & sort \\
            \hline
            $\circ$ & $(g{,}f\ofsort e).s(g)= t(f)$ & $e$ \\
            $\id$ & $(x\ofsort v).\top$ & $e$
        \end{tabular}
    \end{center}
    \begin{equation*}
        E\coloneq\left\{
        \begin{gathered}
            \top\seq{x\ofsort v} s(\id(x))=x\wedge t(\id(x))=x,\\
             s(g)= t(f)\seq{g{,}f\ofsort e} s(g\circ f)= s(f)\wedge t(g\circ f)= t(g),\\
            \top\seq{f\ofsort e}f\circ\id( s(f))=f\wedge \id( t(f))\circ f=f,\\
             s(h)= t(g)\wedge s(g)= t(f)\seq{h{,}g{,}f\ofsort e}(h\circ g)\circ f=h\circ (g\circ f)
        \end{gathered}
        \right\}
    \end{equation*}
    Then, we have $\Alg(\Omega,E)\cong\Cat$.
\end{example}

\begin{example}[Monoid-graded rings]
    We first define the finitary partial Horn theory $\bS_\mgset$ for \emph{monoid-graded sets}.
    Let $S\coloneq\{m,s\}$, where $m$ represents ``a monoid'' and $s$ represents ``a base set.''
    The $S$-sorted finitary signature $\Sigma_\mgset$ consists of the following function symbols:
    \begin{gather*}
        d\colon s\to m,\quad
        \cdot\colon m\sqcap m\to m,\quad
        e\colon ()\to m.
    \end{gather*}
    The finitary partial Horn theory $\bS_\mgset$ consists of:
    \begin{gather*}
        \top\seq{x\ofsort s} d(x)\defined,\quad
        \top\seq{a{,}b\ofsort m} a\cdot b\defined,\quad
        \top\seq{} e\defined;\\
        \top\seq{a{,}b{,}c\ofsort m} (a\cdot b)\cdot c=a\cdot (b\cdot c),\quad
        \top\seq{a\ofsort m} e\cdot a=a\wedge a\cdot e=a.
    \end{gather*}
    A partial model for $\bS_\mgset$, called a \emph{monoid-graded set}, is simply a map from a set to a monoid.
    
    Now, we define an $\bS_\mgset$-relative finitary algebraic theory $(\Omega,E)$ by the following:
    \begin{center}
        $\Omega$:\qquad
        \begin{tabular}{ccc}
            & arity & sort \\
            \hline
            $1$ & $().\top$ & $s$\\
            $\otimes$ & $(x{,}y\ofsort s).\top$ & $s$\\
            $0$ & $a\ofsort m.\top$ & $s$\\
            $-$ & $x\ofsort s.\top$ & $s$\\
            $+$ & $(x{,}y\ofsort s).d(x)=d(y)$ & $s$
        \end{tabular}
    \end{center}
    \begin{equation*}
        E\coloneq\left\{
        \begin{gathered}
            \top\seq{}d(1)=e,\quad
            \top\seq{x{,}y\ofsort s}d(x\otimes y)=d(x)\cdot d(y);\\
            \top\seq{a\ofsort m}d(0(a))=a,\quad
            \top\seq{x\ofsort s}d(-x)=d(x);\\
            d(x)=d(y)\seq{x{,}y\ofsort s}d(x+y)=d(x);\\
            d(x)=d(y)\wedge d(y)=d(z)\seq{x{,}y{,}z\ofsort s}(x+y)+z=x+(y+z);\\
            d(x)=d(y)\seq{x{,}y\ofsort s} x+y=y+x;\\
            \top\seq{x\ofsort s}x+0(d(x))=x\wedge x+(-x)=0(d(x));\\
            \top\seq{x{,}y{,}z\ofsort s}(x\otimes y)\otimes z=x\otimes (y\otimes z),\quad
            \top\seq{x\ofsort s}1\otimes x=x \wedge x\otimes 1=x;\\
            d(x)=d(y)\seq{x{,}y{,}z\ofsort s}(x+y)\otimes z=(x\otimes z)+(y\otimes z) \wedge z\otimes (x+y)=(z\otimes x)+(z\otimes y)
        \end{gathered}
        \right\}
    \end{equation*}
    An $(\Omega,E)$-algebra is called a \emph{monoid-graded ring}.
\end{example}

\begin{example}[Uniquely difference-ordered semirings]
    Let $\bS_\pos$ be the partial Horn theory for posets as in \cref{eg:pht_for_posets}.
    We define an $\bS_\pos$-relative finitary algebraic theory $(\Omega,E)$ as follows:
    \begin{center}
        $\Omega$:\qquad
        \begin{tabular}{ccc}
            & arity & sort \\
            \hline
            $0$ & $().\top$ & $\sort$\\
            $1$ & $().\top$ & $\sort$\\
            $+$ & $(x,y).\top$ & $\sort$\\
            $\cdot$ & $(x,y).\top$ & $\sort$\\
            $\ominus$ & $(x,y).x\le y$ & $\sort$
        \end{tabular}
    \end{center}
    \begin{equation*}
        E\coloneq\left\{
        \begin{gathered}
            \top\seq{x,y,z}(x+y)+z=x+(y+z)\wedge (x\cdot y)\cdot z=x\cdot (y\cdot z);\\
            \top\seq{x,y}x+y=y+x;\\
            \top\seq{x}0+x=x\wedge 1\cdot x=x\wedge x\cdot 1=x\wedge 0\cdot x=0\wedge x\cdot 0=0;\\
            \top\seq{x,y,z}x\cdot (y+z)=x\cdot y+x\cdot z\wedge (x+y)\cdot z=x\cdot z+y\cdot z;\\
            x\le y\seq{x,y}x+(y\ominus x)=y;\\
            \top\seq{x,y}x\le x+y\wedge (x+y)\ominus x=y
        \end{gathered}
        \right\}
    \end{equation*}
    Then, an $(\Omega,E)$-algebra is precisely a \emph{uniquely difference-ordered semiring} in \cite{golan2003semiring}.
\end{example}

\begin{example}[Partial Boolean algebras]
    The partial Horn theory $\bS_\rsrel$ for reflexive symmetric relations is given by:
    \begin{gather*}
        S_\rsrel\coloneq\{ \sort \},\quad
        \Sigma_\rsrel\coloneq\{ \odot\colon \sort\sqcap\sort \},\quad
        \bS_\rsrel\coloneq\{ \top\seq{x} x\odot x,\quad x\odot y\seq{x,y}y\odot x \}.
    \end{gather*}
    We define an $\bS_\rsrel$-relative finitary algebraic theory $(\Omega,E)$ as follows:
    \begin{center}
        $\Omega$:\qquad
        \begin{tabular}{ccc}
            & arity & sort \\
            \hline
            $0$ & $().\top$ & $\sort$\\
            $1$ & $().\top$ & $\sort$\\
            $\neg$ & $x.\top$ & $\sort$\\
            $\vee$ & $(x,y).x\odot y$ & $\sort$\\
            $\wedge$ & $(x,y).x\odot y$ & $\sort$
        \end{tabular}
    \end{center}
    \begin{equation*}
        E\coloneq\left\{
        \begin{gathered}
            \top \seq{x} x\odot 0,~x\odot 1;\quad\quad
            x\odot y \seq{x,y} x\odot\neg y;\\
            x\odot y,~y\odot z,~z\odot x \seq{x,y,z} x\odot (y\vee z),~x\odot (y\wedge z);\\
            x\odot y,~y\odot z,~z\odot x \seq{x,y,z} (x\vee y)\vee z=x\vee (y\vee z),~(x\wedge y)\wedge z=x\wedge (y\wedge z);\\
            x\odot y \seq{x,y} x\vee y=y\vee x,~x\wedge y=y\wedge x;\\
            x\odot y \seq{x,y} (x\wedge y)\vee x=x,~x\wedge (y\vee x)=x;\\
            \top \seq{x} x\vee 0=x,~x\wedge 1=x,~x\vee\neg x=1,~x\wedge\neg x=0;\\
            x\odot y,~y\odot z,~z\odot x \seq{x,y,z} (x\wedge y)\vee z=(x\vee z)\wedge (x\vee z);\\
            x\odot y,~y\odot z,~z\odot x \seq{x,y,z} (x\vee y)\wedge z=(x\wedge z)\vee (y\wedge z)
        \end{gathered}
        \right\}
    \end{equation*}
    In the above, we use the symbol (,) instead of $\wedge$ to avoid confusion.
    An $(\Omega,E)$-algebra is a Boolean algebra-like algebra whose conjunction and disjunction are partial, which is called a \emph{partial Boolean algebra} in \cite{berg2012noncomm}.
    There, the reflexive symmetric relation $\odot$ is called \emph{commeasurability}.
\end{example}

\subsubsection{The infinitary case}
\begin{example}[$\omega$-cpos]
    Let $\bS_\pos$ be the partial Horn theory as in \cref{eg:pht_for_posets}.
    In what follows, we regard $\bS_\pos$ as an $\aleph_1$-ary partial Horn theory.
    We present an $\bS_\pos$-relative $\aleph_1$-ary algebraic theory $(\Omega,E)$ for \emph{$\omega$-cpos}.
    Let $\Omega\coloneq\{\sup\}$ with 
    \begin{equation*}
        \ar(\sup)\coloneq(x_n)_{n<\omega}.\bigwedge_{n<\omega}x_n\le x_{n+1},\quad \outsort(\sup)\coloneq\sort.
    \end{equation*}
    The set $E$ is defined by the following:
    \begin{equation*}
        E\coloneq\left\{
        \begin{gathered}
            \bigwedge_{n<\omega}x_n\le x_{n+1} \longseq{(x_n)_{n<\omega}} \bigwedge_{n<\omega} x_n\le \sup(\tup{x});\\
            \bigwedge_{n<\omega}x_n\le x_{n+1}\wedge\bigwedge_{n<\omega}x_n\le y \longseq{(x_n)_{n<\omega},y} \sup(\tup{x})\le y
        \end{gathered}
        \right\}
    \end{equation*}
    Then, an $(\Omega,E)$-algebra is precisely an \emph{$\omega$-cpo}, i.e., a poset where every $\omega$-chain has a supremum.
\end{example}

\begin{example}[Generalized complete metric spaces]
    Let $\bS_\met$ be the $\aleph_1$-ary partial Horn theory as in \cref{eg:pht_for_met}.
    We now present an $\bS_\met$-relative $\aleph_1$-ary algebraic theory $(\Omega,E)$.
    Let $\Omega\coloneq\{\mathrm{lim}^\mu\}_\mu$, where $\mu$ represents all weakly decreasing maps $\mu\colon\bN\to [0,\infty)$ converging to $0$.
    The arity of $\mathrm{lim}^\mu$ is
    \begin{equation*}
        (x_n)_{n<\omega}.\left(\bigwedge_{N\in\bN}\bigwedge_{m,n\ge N}R^{\mu (N)} (x_m,x_n)\right).
    \end{equation*}
    The set $E$ contains the following $\bS_\met$-relative judgment
    \begin{equation*}
        \bigwedge_{N\in\bN}\bigwedge_{m,n\ge N}R^{\mu (N)} (x_m,x_n) \longseq{\tup{x}} \bigwedge_{N\in\bN}R^{2\mu (N)} (x_N,\mathrm{lim}^\mu (\tup{x}))
    \end{equation*}
    for each $\mu$.
    We can consider the operator $\mathrm{lim}^\mu$ assigns to each Cauchy sequence having the modulus of convergence of type $\mu$, its limit.
    Thus, an $(\Omega,E)$-algebra is precisely a generalized metric space where every Cauchy sequence converges.
    This algebraic presentation of Cauchy completeness is also discussed in \cite[Example 4.8]{ford2021monads}.
\end{example}

\begin{example}[Banach spaces]
    Let $\bS_\pmet$ be the $\aleph_1$-ary partial Horn theory as in \cref{eg:pht_for_pmet}.
    We present an $\bS_\pmet$-relative $\aleph_1$-ary algebraic theory $(\Omega,E)$.
    The $\bS_\pmet$-relative $\aleph_1$-ary signature $\Omega$ consists of:
    \begin{itemize}
        \item
        for each pair of $\alpha\in\bC$ and $r,r'\in [0,\infty)$ such that $|\alpha|r\le r'$, a total operator $\alpha_r^{r'}$ with
        \begin{equation*}
            \ar(\alpha_r^{r'})\coloneq(x\ofsort r).\top,\quad \outsort(\alpha_r^{r'})\coloneq r';
        \end{equation*}
        \item
        for each pair of $r,r'\in [0,\infty)$, a total operator $\dot{+}$ with
        \begin{equation*}
            \ar(\dot{+})\coloneq(x\ofsort r,y\ofsort r').\top,\quad \outsort(\dot{+})\coloneq r+r';
        \end{equation*}
        \item
        for each pair of $r\in [0,\infty)$ and a weakly decreasing map $\mu\colon\bN\to [0,\infty)$ converging to $0$, a partial operator $\mathrm{lim}_r^\mu$ with
        \begin{equation*}
            \ar(\mathrm{lim}_r^\mu)\coloneq (x_k\ofsort r)_{k<\omega}.\left(\bigwedge_{N\in\bN}\bigwedge_{m,n\ge N}R_r^{\mu (N)} (x_m,x_n)\right),\quad
            \outsort(\mathrm{lim}_r^\mu)\coloneq r.
        \end{equation*}
    \end{itemize}
    The set $E$ consists of the following $\bS_\pmet$-relative $\aleph_1$-ary judgments:
    \begin{align}
        \label{eq:banach_1}
        \top &\seq{x\ofsort r} 1_r^{r'}(x)=\iota_r^{r'}(x);\\
        \top &\seq{x\ofsort r_0} \beta_{r_1}^{r_2}(\alpha_{r_0}^{r_1}(x))=(\beta\alpha)_{r_0}^{r_2}(x);\\
        \top &\longseq{x\ofsort r_0,y\ofsort s_0} \iota_{r_0}^{r_1}(x)\dot{+}\iota_{s_0}^{s_1}(y)=\iota_{r_0+s_0}^{r_1+s_1}(x\dot{+}y) \qquad (r_0\le r_1,s_0\le s_1);\\
        \top &\longseq{x{,}y{,}z\ofsort r}(x\dot{+}y)\dot{+}\iota_r^{2r}(z)=\iota_r^{2r}(x)\dot{+}(y\dot{+}z);\\
        \top &\seq{x{,}y\ofsort r}x\dot{+}y=y\dot{+}x;\\
        \top &\seq{x\ofsort r}\iota_0^r(\dot{0})\dot{+}x=\iota_r^{2r}(x);\\
        \top &\seq{x{,}y\ofsort r_0}\alpha_{2r_0}^{2r_1}(x\dot{+}y)=\alpha_{r_0}^{r_1}(x)\dot{+}\alpha_{r_0}^{r_1}(y) \qquad (|\alpha|r_0\le r_1);\\
        \label{eq:banach_2}
        \top &\seq{x\ofsort r_0}(\alpha+\beta)_{r_0}^{2r_1}(x)=\alpha_{r_0}^{r_1}(x)\dot{+}\beta_{r_0}^{r_1}(x) \qquad (|\alpha|r_0\le r_1, |\beta|r_0\le r_1);
    \end{align}
    \begin{equation}\label{eq:banach_3}
        \bigwedge_{N\in\bN}\bigwedge_{m,n\ge N}R_r^{\mu (N)} (x_m,x_n) \longseq{(x_k\ofsort r)_{k<\omega}} \bigwedge_{N\in\bN}R_r^{2\mu (N)} (x_N,\mathrm{lim}^\mu (\tup{x})).
    \end{equation}
    The axioms \cref{eq:banach_1} to \cref{eq:banach_2} make an $(\Omega,E)$-algebra to be a complex normed space, and the last axiom \cref{eq:banach_3} requires that every Cauchy sequence converges.
    Thus, there is an equivalence $\Alg (\Omega,E)\simeq\Ban$, where $\Ban$ is the category of (complex) Banach spaces and contractions.
    According to \cref{cor:equiv_between_monadicity_and_rat}, which will be proved later, it follows that $\Ban$ is $\aleph_1$-ary monadic over $\pMet$, the category of pointed metric spaces and contractions.
    A similar statement about monadicity appears in \cite[Theorem 3.2]{rosicky2022banach}.
\end{example}

\subsection{Morphisms between relative algebraic theories}
We now introduce morphisms of relative algebraic theories and define the category $\Thl^\bS$ of $\bS$-relative $\lambda$-ary algebraic theories, which will turn out to be equivalent to the category of $\lambda$-ary monads on $\PMod\bS$ later in \cref{section:monad}.
In this subsection, we fix an $S$-sorted $\lambda$-ary signature $\Sigma$ and a $\lambda$-ary partial Horn theory $\bS$ over $\Sigma$ again.

\begin{definition}\label{def:morphism_rat}
    Let $(\Omega,E)$ and $(\Omega',E')$ be $\bS$-relative $\lambda$-ary algebraic theories.
    \begin{enumerate}
        \item
        A \emph{($\lambda$-ary) theory morphism} $\rho\colon (\Omega,E)\to (\Omega',E')$ is an assignment to each operator $\omega\in\Omega$, a $\pht{\Omega'}{E'}$-term $\omega^\rho$ of sort $\outsort(\omega)$ generated by $\ar(\omega)$ satisfying that for every $(\phi\seq{\tup{x}}\psi)\in E$, $\pht{\Omega'}{E'}\vdash (\phi\seq{\tup{x}}\psi^\rho)$ holds.
        Here, $\psi^\rho$ is the $\rho$-translation, which is constructed by replacing all symbols that $\psi$ includes by $\rho$.
        \item
        Let $\rho,\sigma\colon (\Omega,E)\to (\Omega',E')$ be theory morphisms.
        We say that \emph{$\rho$ and $\sigma$ are equivalent} and write $\rho\sim\sigma$ if $\pht{\Omega'}{E'}\vdash (\phi\seq{\tup{x}}\omega^\rho=\omega^\sigma)$ holds for every $\omega\in\Omega$ with $\ar(\omega)=\tup{x}.\phi$.
        \item
        Let $\rho\colon (\Omega,E)\to (\Omega',E')$ be a theory morphism.
        Given an algebra $\bA\in\Alg(\Omega',E')$, an algebra $\bA^\rho\in\Alg(\Omega,E)$ is defined by $\intpn{\omega}{\bA^\rho}\coloneq\intpn{\omega^\rho}{\bA}$ for each $\omega\in\Omega$.
        Then, there exists a unique functor $\funcAlg\rho\colon \Alg(\Omega',E')\to\Alg(\Omega,E)$ such that $\bA\mapsto\bA^\rho$ and the following diagram
        \begin{equation*}
            \begin{tikzcd}
                \Alg(\Omega,E)\arrow[rd,"U"'] & & \Alg(\Omega',E')\arrow[ll,"\funcAlg\rho"']\arrow[ld,"U'"] \\
                & \PMod\bS &
            \end{tikzcd}
        \end{equation*}
        commutes.
        Here, $U$ and $U'$ are forgetful functors.\qedhere
    \end{enumerate}
\end{definition}

\begin{remark}
    A theory morphism between relative algebraic theories is a special case of theory morphisms between partial Horn theories defined in \cref{def:theory_morphism_pht}.
    Indeed, a theory morphism $\rho\colon (\Omega,E)\to (\Omega',E')$ is simply a theory morphism $\rho\colon \pht{\Omega}{E}\to\pht{\Omega'}{E'}$ which moves no sort and no symbol in $\Sigma$.
\end{remark}

\begin{remark}\label{rem:equivalence_of_theory_morphism}
    Let $\rho,\sigma\colon (\Omega,E)\to (\Omega',E')$ be theory morphisms.
    Then, $\funcAlg\rho=\funcAlg\sigma$ if and only if $\rho\sim\sigma$.
\end{remark}

\begin{definition}\quad
    \begin{enumerate}
        \item
        We now define the category $\Thl^\bS$ of $\bS$-relative $\lambda$-ary algebraic theories:
        $\Thl^\bS$ is the category whose objects are $\bS$-relative $\lambda$-ary algebraic theories and whose morphism from $(\Omega,E)$ to $(\Omega',E')$ is an equivalence class $[\rho]$ of a $\lambda$-ary theory morphism $\rho$.
        Given two morphisms $[\rho]\colon (\Omega,E)\to (\Omega',E')$ and $[\sigma]\colon (\Omega',E')\to (\Omega'',E'')$, their composite is an equivalence class of the theory morphism $\sigma\circ\rho$ which assigns to $\omega\in\Omega$, the $\pht{\Omega''}{E''}$-term $(\omega^\rho)^\sigma$.
        The identity morphism is an equivalence class of the identity theory morphism $\rho$ such that $\omega^\rho\coloneq\omega$.
        \item
        We can define the functor $\funcAlg\colon \Thl^\bS\to (\CAT/\PMod\bS)^\op$ by
        \[
            (\Omega,E)\mapsto\Alg(\Omega,E),\qquad
            [\rho]\mapsto\funcAlg\rho,
        \]
        where $\CAT$ is the category of (not necessarily small) categories and $\CAT/\PMod\bS$ is the slice category.
        This functor is well-defined by \cref{rem:equivalence_of_theory_morphism}.\qedhere
    \end{enumerate}
\end{definition}

\begin{theorem}\label{thm:alg_fully_faithful}
    The functor $\funcAlg\colon \Thl^\bS \to (\CAT/\PMod\bS)^\op$ is fully faithful.
\end{theorem}
\begin{proof}
    By \cref{rem:equivalence_of_theory_morphism}, the functor $\funcAlg$ is faithful.
    To prove fullness, take an arbitrary functor $K\colon \Alg(\Omega',E')\to\Alg(\Omega,E)$ which commutes with forgetful functors.
    Let $\omega\in\Omega$ with $\ar(\omega)=(x_i\ofsort s_i)_{i<\alpha}.\phi$ and $\outsort(\omega)=s$ and let $\bA\coloneq\repn{\tup{x}.\phi}_\pht{\Omega'}{E'}\in\Alg(\Omega',E')$ and $A\coloneq U'\bA$.
    Considering the interpretation $\intpn{\omega}{K\bA}\colon \intpn{\tup{x}.\phi}{A}\to A_s$, we get an equivalence class $\intpn{\omega}{K\bA}([\tup{x}]_\pht{\Omega'}{E'})\in A_s$ of a $\pht{\Omega'}{E'}$-term of sort $s$ generated by $\tup{x}.\phi$ from the construction of $\repn{\tup{x}.\phi}_\pht{\Omega'}{E'}$ in \cref{def:representing_model}.
    We now define $\omega^\rho$ as a representative of the class $\intpn{\omega}{K\bA}([\tup{x}]_\pht{\Omega'}{E'})\in A_s$.

    We now show that the interpretation maps $\intpn{\omega}{K\bB},\intpn{\omega^\rho}{\bB}\colon\intpn{\tup{x}.\phi}{B}\to B_s$ coincide for every $\bB\in\Alg(\Omega',E')$.
    Write $B\coloneq U'\bB$.
    Then, for every morphism $f\colon \bA\to\bB$ in $\Alg(\Omega',E')$, the following diagram is pairwise commutative:
    \begin{equation}\label{eq:pairwise_commutative_AB}
        \begin{tikzcd}
            \prod_{i<\alpha}A_{s_i}\arrow[d,"\prod_{i<\alpha}(U'f)_{s_i}"']\arrow[r,phantom,"\supseteq"] &[-10pt] \intpn{\tup{x}.\phi}{A}\arrow[d,dashed,"\intpn{\tup{x}.\phi}{U'f}"']\arrow[r,shift left=1,"\intpn{\omega}{K\bA}"]\arrow[r,shift right=1,"\intpn{\omega^\rho}{\bA}"'] & A_s\arrow[d,"(U'f)_s"] \\
            \prod_{i<\alpha}B_{s_i}\arrow[r,phantom,"\supseteq"] &[-10pt] \intpn{\tup{x}.\phi}{B}\arrow[r,shift left=1,"\intpn{\omega}{K\bB}"]\arrow[r,shift right=1,"\intpn{\omega^\rho}{\bB}"'] & B_s
        \end{tikzcd}\incat{\Set}
    \end{equation}
    By the definition of $\omega^\rho$, we have $\intpn{\omega}{K\bA}=\intpn{\omega^\rho}{\bA}$.
    By $\Alg(\Omega',E')(\bA,\bB)\cong\intpn{\tup{x}.\phi}{B}$, we see that every element of $\intpn{\tup{x}.\phi}{B}$ lies in the image of $\intpn{\tup{x}.\phi}{U'f}$ for some $f$.
    Thus, \cref{eq:pairwise_commutative_AB} implies $\intpn{\omega}{K\bB}=\intpn{\omega^\rho}{\bB}$.

    Consequently, we have $\bB^\rho=K\bB\in\Alg(\Omega,E)$ for every $\bB\in\Alg(\Omega',E')$; hence for every $(\phi\seq{\tup{x}}\psi)\in E$, $\pht{\Omega'}{E'}\vdash (\phi\seq{\tup{x}}\psi^\rho)$ is satisfied.
    Now, this yields a theory morphism $\rho\colon (\Omega,E)\to (\Omega',E')$ such that $K=\funcAlg\rho$.
\end{proof}

\section{Birkhoff's variety theorem}\label{section:birkhoff}
Our goal is to prove Birkhoff's variety theorem for partial Horn theories (\cref{thm:birkhoff_partialHorn}), which can be applied to our relative algebraic theories.

\subsection{Presentable proper factorization systems}
We recall the definition of an orthogonal factorization system, which plays an important role in subsequent subsections.

\begin{definition} Let $\C$ be a category.
    \begin{enumerate}
        \item
        Given morphisms $e$ and $m$ in $\C$, we write $e\perp m$ if for any commutative square $ve=mu$ there exists a unique diagonal filler making both triangles commute:
        \begin{equation*}
            \begin{tikzcd}
                \cdot\arrow[r,"u"]\arrow[d,"e"'] & \cdot\arrow[d,"m"] \\
                \cdot\arrow[r,"v"']\arrow[ur,dashed,"\exists !"] & \cdot
            \end{tikzcd}
        \end{equation*}
        \item
        Given a class of morphisms $\Lambda$, denote by $\lorth(\Lambda)$ and $\rorth(\Lambda)$ the classes
        \begin{align*}
            \lorth(\Lambda)&\coloneq\{ e \mid e\perp m\text{ for all }m\in\Lambda\} \\
            \rorth(\Lambda)&\coloneq\{ m \mid e\perp m\text{ for all }e\in\Lambda\}.
        \end{align*}
        \item
        An \emph{orthogonal factorization system} on $\C$ is a pair $(\bfE,\bfM)$ of classes of morphisms in $\C$ that satisfies the following conditions:
        \begin{itemize}
            \item $\bfE$ and $\bfM$ are closed under composition and contain all isomorphisms in $\C$;
            \item Every morphism $f$ in $\C$ has a factorization $f=me$ with $e\in\bfE$ and $m\in\bfM$;
            \item $\bfE\perp\bfM$ holds, i.e., for any $e\in\bfE$ and $m\in\bfM$, $e\perp m$ holds.
        \end{itemize}
        Given an orthogonal factorization system $(\bfE,\bfM)$, we have $\bfE=\lorth(\bfM)$ and $\bfM=\rorth(\bfE)$.
        These can be verified straightforwardly.
        \item
        A \emph{proper factorization system} is an orthogonal factorization system $(\bfE,\bfM)$ such that every morphism in $\bfE$ is an epimorphism and every morphism in $\bfM$ is a monomorphism.\qedhere
    \end{enumerate}
\end{definition}

We now recall the orthogonal version of the small object argument on locally presentable categories.
\begin{theorem}\label{thm:OFS_in_LPcategory}
    Let $\Lambda$ be a set of morphisms in a locally presentable category $\A$.
    Then, $(\,\lorth\rorth(\Lambda),\rorth(\Lambda)\,)$ is an orthogonal factorization system on $\A$.
\end{theorem}
\begin{proof}
    See, for example, \cite[2.2 Theorem]{fajstruprosicky2008convenient}.
\end{proof}

\begin{notation}
    Given a class of morphisms $\Lambda$, denote by $\Lambda_\lambda$ the essentially small class
    \begin{equation*}
        \Lambda_\lambda\coloneq\{ f\in\Lambda \mid \dom f, \cod f\text{ are }\lambda\text{-presentable} \}.\qedhere
    \end{equation*}
\end{notation}

\begin{lemma}\label{lem:mono_rlp_wrt_retracts}
    Let $\A$ be a locally $\lambda$-presentable category.
    Denote by $\mathbf{Ret}$ the class of all retractions in $\A$ and by $\mathbf{Mono}$ the class of all monomorphisms in $\A$.
    Then, $\rorth(\mathbf{Ret}_\lambda)=\mathbf{Mono}$ holds.
\end{lemma}
\begin{proof}
    In general, a morphism $f$ is monic if and only if $\nabla_X \perp f$ holds for every object $X$, where $\nabla_X\colon X+X\to X$ denotes the codiagonal morphism.
    This equivalence still holds if we restrict the codiagonals to those of all $\lambda$-presentable objects, because the class of all $\lambda$-presentable objects forms a generator in any locally $\lambda$-presentable category.
    This shows $\rorth(\mathbf{Ret}_\lambda)\subseteq\mathbf{Mono}$, and the converse inclusion is immediate.
\end{proof}

\begin{definition}
    Let $\bfM$ be a class of monomorphisms.
    A morphism $f\colon X\to Y$ is \emph{$\bfM$-extremal} if $f$ factors through no proper $\bfM$-subobject of $Y$, i.e., 
    if $f$ has a factorization $f=mg$ with $m\in\bfM$, then $m$ is an isomorphism.
\end{definition}

\begin{lemma}\label{lem:M-extremal_M-strong}
    Let $\C$ be a category with pullbacks. Let $\bfM$ be a class of monomorphisms in $\C$ which is closed under pullbacks, i.e., 
    for every pullback square
    \begin{equation*}
        \begin{tikzcd}
            \cdot\arrow[r]\arrow[d,"m'"']\arrow[rd,pos=0.1,phantom,"\lrcorner"] & \cdot\arrow[d,"m"] \\
            \cdot\arrow[r] & \cdot
        \end{tikzcd}\incat{\C},
    \end{equation*}
    $m\in\bfM$ implies $m'\in\bfM$.
    Then, the class $\lorth(\bfM)$ coincides with the class of all $\bfM$-extremal morphisms in $\C$.
\end{lemma}
\begin{proof}
    The proof is straightforward.
\end{proof}

The following definition is due to \cite{hebert2004algebraically}.

\begin{definition}
    An orthogonal factorization system $(\bfE,\bfM)$ is \emph{$\lambda$-presentable} if $\rorth(\bfE_\lambda)\subseteq\bfM$.
\end{definition}

\begin{theorem}\label{thm:PPFS}
    Let $\A$ be a locally $\lambda$-presentable category.
    Let $\Lambda$ be a class of epimorphisms between $\lambda$-presentable objects in $\A$.
    Consider the following two classes of morphisms in $\A$.
    \begin{itemize}
        \item $\bfM$: the class of all monomorphisms in $\A$ belonging to $\rorth(\Lambda)$,
        \item $\bfE$: the class of all $\bfM$-extremal morphisms in $\A$.
    \end{itemize}
    Then $(\bfE,\bfM)$ is a $\lambda$-presentable proper factorization system.
    Conversely, every $\lambda$-presentable proper factorization system on $\A$ is constructed from some $\Lambda$ in this way.
\end{theorem}
\begin{proof}
    Define $\Lambda^*\coloneq\mathbf{Ret}_\lambda \cup\Lambda$.
    Since $\Lambda^*$ is essentially small, \cref{thm:OFS_in_LPcategory} shows that
    \[
    (\,\lorth\rorth(\Lambda^*),\rorth(\Lambda^*) \,)
    \]
    is a $\lambda$-presentable orthogonal factorization system on $\A$.
    \cref{lem:mono_rlp_wrt_retracts} implies $\rorth(\Lambda^*)=\bfM$, and \cref{lem:M-extremal_M-strong} implies $\lorth\rorth(\Lambda^*)=\bfE$.

    Since every morphism in $\Lambda$ is an epimorphism, the class $\bfM$ has the following property:
    \[
        gf\in\bfM~\implies~f\in\bfM,
    \]
    which is equivalent to that every morphism in $\bfE$ is an epimorphism \cite[{}14.11]{adamekherrlichstrecker2006joy}.

    Conversely, if we are given an arbitrary $\lambda$-presentable proper factorization system $(\bfE,\bfM)$ on $\A$, take $\Lambda\coloneq\bfE_\lambda$.
    Then, the construction above recovers $(\bfE,\bfM)$.
\end{proof}

The proof of the following is essentially the same as \cite[Theorem 2.10]{adamek2009orthogonal}.
\begin{theorem}\label{thm:co-intersection}
    Let $(\bfE,\bfM)$ be a $\lambda$-presentable proper factorization system on a locally $\lambda$-presentable category $\A$.
    Then every morphism $f\colon A\to X$ in $\bfE$ with $\lambda$-presentable domain $A$ is a $\lambda$-filtered colimit (in $A/\A$) of morphisms in $\bfE$ with $\lambda$-presentable codomain.
\end{theorem}
\begin{proof}
    Let $e_I\colon A\to Y_I\,(I\in\bI)$ be the family of all morphisms belonging to $\bfE$ through which $f$ factors and whose codomains are $\lambda$-presentable.
    Since $\bI$ yields a $\lambda$-filtered essentially small category, 
    there is a colimit $(e,Y)$ of $(e_I,Y_I)_{I\in\bI}$ in the coslice category $A/\A$.
    By the universality of the colimit $(e,Y)$, $f$ has a factorization $f=ge$ as follows:
    \begin{equation*}
        \begin{tikzcd}
            A\arrow[d,"e_I"']\arrow[rd,"e"description]\arrow[r,"f"] & X \\
            Y_I\arrow[r,"\kappa_I"'] & Y\arrow[u,dashed,"g"']
        \end{tikzcd}\incat{\A},
    \end{equation*}
    where $\kappa_I$ is the coprojection of the colimit.
    Since $\bfE$ coincides with the $\bfM$-extremals by \cref{thm:PPFS}, $f\in\bfE$ implies $g\in\bfE$.

    To prove that $g$ is an isomorphism, it suffices to show $g\in\bfM$.
    To show this, take the following commutative square arbitrarily:
    \begin{equation}\label{eq:comm_square_dg}
        \begin{tikzcd}
            B\arrow[d,"\bfE\ni\,d"']\arrow[r,"u"] & Y\arrow[d,"g"] \\
            C\arrow[r,"v"'] & X
        \end{tikzcd}\incat{\A}.
    \end{equation}
    We have to construct a unique diagonal filler for the above square \cref{eq:comm_square_dg}.
    Since $(\bfE,\bfM)$ is a $\lambda$-presentable proper factorization system, the uniqueness always holds and we can assume that $B$ and $C$ are $\lambda$-presentable.
    Now $u$ has a factorization $u=\kappa_I u'$ for some $I\in\bI$ because $B$ is $\lambda$-presentable.
    Take a pushout $Z$ of $d$ and $u'$, and consider the following canonical morphism $h\colon Z\to X$:
    \begin{equation*}
        \begin{tikzcd}
            B\arrow[rd,pos=0.95,bend right=20,phantom,"\ulcorner"]\arrow[rr,shift left=4,"u"]\arrow[dd,"d"']\arrow[r,"u'"'] & Y_I\arrow[r,"\kappa_I"']\arrow[d,"\rho"] & Y\arrow[dd,"g"] \\
            & Z\arrow[rd,dashed,"h"] & \\
            C\arrow[ru]\arrow[rr,"v"'] & & X
        \end{tikzcd}\incat{\A}.
    \end{equation*}
    Since $B$, $C$, and $Y_I$ are $\lambda$-presentable, $Z$ is also $\lambda$-presentable.
    Now $d\in\bfE$ implies $\rho\in\bfE$.
    Thus, 
    \[
        A\longarr[e_I]Y_I\longarr[\rho]Z
    \]
    is a morphism belonging to $\bfE$ through which $f$ factors.
    Therefore, there exists $J\in\bI$ satisfying $Z=Y_J$ and $\rho e_I=e_J$.
    Then the following diagram commutes:
    \begin{equation*}
        \begin{tikzcd}
            B\arrow[rd,pos=0.95,bend right=20,phantom,"\ulcorner"]\arrow[rr,shift left=4,"u"]\arrow[dd,"d"']\arrow[r,"u'"'] & Y_I\arrow[r,"\kappa_I"']\arrow[d,"\rho"] & Y \\
            & Y_J\arrow[ur,"\kappa_J"'] & \\
            C\arrow[ru] & &
        \end{tikzcd}\incat{\A}.
    \end{equation*}
    Since $d$ is an epimorphism, we have constructed a diagonal filler for the square \cref{eq:comm_square_dg}.
\end{proof}

\subsection{Closed monomorphisms}
In this subsection, we introduce a \emph{(dense, closed--mono)-factorization system} on $\PMod\bT$ for each partial Horn theory $\bT$.
It should be emphasized that the notion of (dense, closed--mono) strongly depends on the syntax $\bT$ (\cref{eg:PHT_for_monoids}).
Indeed, $\bT$-closedness may differ from $\bT'$-closedness even if $\PMod\bT\simeq\PMod\bT'$, and the same holds for density.
\pagebreak

\begin{remark}
    Let $\bT$ be a $\lambda$-ary partial Horn theory over an $S$-sorted $\lambda$-ary signature $\Sigma$.
    Then, for every morphism $h\colon A\to B$ in $\PMod\bT$, the following are equivalent:
    \begin{enumerate}
        \item
        $h$ is a monomorphism in $\PMod\bT$,
        \item
        $h_s\colon A_s\to B_s$ is injective for every sort $s\in S$.
    \end{enumerate}
    Thus, a subobject in $\PMod\bT$ is just a submodel.
\end{remark}

\begin{definition}
    Let $\bT$ be a $\lambda$-ary partial Horn theory over an $S$-sorted $\lambda$-ary signature $\Sigma$.
    \begin{enumerate}
        \item
        A monomorphism $A\hookrightarrow B$ in $\PMod\bT$ is called \emph{$\bT$-closed} (or \emph{$\Sigma$-closed}) if the following diagrams form pullback squares for any $f,R\in\Sigma$.
        \begin{equation*}
            \begin{tikzcd}
                \mathrm{Dom}(\intpn{f}{A})\arrow[d,hook']\arrow[r,hook]\arrow[rd,pos=0.1,phantom,"\lrcorner"] &[-10pt] \prod_{i<\alpha}A_{s_i}\arrow[d,hook'] \\
                \mathrm{Dom}(\intpn{f}{B})\arrow[r,hook] & \prod_{i<\alpha}B_{s_i}
            \end{tikzcd}
            \quad\quad
            \begin{tikzcd}
                \intpn{R}{A}\arrow[d,hook']\arrow[r,hook]\arrow[r,hook]\arrow[rd,pos=0.1,phantom,"\lrcorner"] &[-10pt] \prod_{i<\alpha}A_{s_i}\arrow[d,hook'] \\
                {\intpn{R}{B}} \arrow[r,hook] & \prod_{i<\alpha}B_{s_i}
            \end{tikzcd}
        \end{equation*}
        \item
        A morphism $h\colon A\to B$ in $\PMod\bT$ is called \emph{$\bT$-dense} (or \emph{$\Sigma$-dense}) if $h$ factors through no $\bT$-closed proper subobject of $B$.\qedhere
    \end{enumerate}
\end{definition}

$\bT$-closed monomorphisms are the so-called embeddings in model theory and play an important role in our generalized Birkhoff theorem.

\begin{remark}\label{rem:explaination_closedmono_dense}
    Let $\bT$ be a $\lambda$-ary partial Horn theory over an $S$-sorted $\lambda$-ary signature $\Sigma$.
    \begin{enumerate}
        \item
        $\bT$-closedness of a submodel $A\subseteq B$ in $\PMod\bT$ is equivalent to saying that given a family $\tup{a}$ of elements of $A$, if $\intpn{f}{B}(\tup{a})$ is defined for a function symbol $f$, then $\intpn{f}{A}(\tup{a})$ is also defined, and if $\tup{a}\in\intpn{R}{B}$ holds for a relation symbol $R$, then $\tup{a}\in\intpn{R}{A}$ also holds.
        \item
        Let $C\subseteq B$ be a subobject of $B\in\PMod\bT$ in $\Set^S$.
        Denote by $A$ the $S$-sorted set of all elements of $B$ which can be written as $\intpn{\tup{x}.\tau}{B}(\tup{c})$ by a family $\tup{c}$ of elements of $C$ and a term $\tau$ over $\Sigma$.
        Then the $\Sigma$-structure of $B$ induces a $\Sigma$-structure on $A$, which makes $A$ the smallest $\bT$-closed submodel of $B$ containing $C$.
        This submodel $A$ is called the \emph{$\bT$-closed submodel of $B$ generated by $C$}.
        \item\label{rem:explaination_closedmono_dense-3}
        A morphism $h\colon A\to B$ in $\PMod\bT$ is $\bT$-dense if and only if the $\bT$-closed submodel generated by the image of $h$ coincides with $B$.\qedhere
    \end{enumerate}
\end{remark}

\begin{example}\label{eg:PHT_for_monoids}\quad
    \begin{enumerate}
        \item
        Let us define an ordinary finitary partial Horn theory $\bT_\mon$ (over $\Sigma_\mon$) for monoids as follows:
        \begin{gather*}
            S\coloneq\{\sort\},\quad \Sigma_\mon\coloneq\{ e\colon ()\to \sort,\quad \cdot\colon \sort\sqcap\sort\to\sort \},
            \\
            \bT_\mon\coloneq\left\{
            \begin{gathered}
                \top\seq{}e\defined,\quad \top\seq{x,y}x\cdot y\defined,\\
                \top\seq{x,y,z}(x\cdot y)\cdot z=x\cdot(y\cdot z),\\
                \top\seq{x}x\cdot e=x \wedge e\cdot x=x
            \end{gathered}
            \right\}.
        \end{gather*}
        We have $\PMod\bT_\mon\cong\Mon$, where $\Mon$ is the category of monoids.
        Then a $\bT_\mon$-closed subobject is just a submonoid, and a $\bT_\mon$-dense morphism is just a surjective homomorphism.
        \item
        Let us define another finitary partial Horn theory $\bT'_\mon$ (over $\Sigma'_\mon$) for monoids as follows:
        \begin{gather*}
            S\coloneq\{\sort\},\quad \Sigma'_\mon\coloneq\Sigma_\mon+\{ \bullet^{-1}\colon \sort\to\sort \},
            \\
            \bT'_\mon\coloneq\bT_\mon+\left\{
            \begin{gathered}
                x^{-1}\defined\seq{x}x^{-1}\cdot x=e \wedge x\cdot x^{-1}=e,\\
                y\cdot x=e \wedge x\cdot y=e\seq{x,y}x^{-1}=y
            \end{gathered}
            \right\}.
        \end{gather*}
        A $\bT'_\mon$-model is just a monoid with the partial inverse function, and we also have $\PMod\bT'_\mon\cong\Mon$.
        Then a submonoid $\bN\subseteq\bZ$ is not $\bT'_\mon$-closed even though it is $\bT_\mon$-closed.
        Therefore closedness of monomorphisms depends on $\bT$.\qedhere
    \end{enumerate}
\end{example}

\begin{theorem}\label{thm:dense_closedmono_factorization}
    Let $\bT$ be a $\lambda$-ary partial Horn theory over an $S$-sorted $\lambda$-ary signature $\Sigma$.
    Then, the pair of the class of all $\bT$-dense morphisms and the class of all $\bT$-closed monomorphisms becomes a $\lambda$-presentable proper factorization system on $\PMod\bT$.
\end{theorem}
\begin{proof}
    Denote by $\Lambda$ the class of morphisms in $\PMod\bT$ consisting of the following:
    \begin{itemize}
        \item
        A morphism $\repn{\tup{x}.\top}_\bT\longarr[\repn{\tup{x}}_\bT]\repn{\tup{x}.f(\tup{x})\defined}_\bT$ for each function symbol $f\in\Sigma$,
        \item
        A morphism $\repn{\tup{x}.\top}_\bT\longarr[\repn{\tup{x}}_\bT]\repn{\tup{x}.R(\tup{x})}_\bT$ for each relation symbol $R\in\Sigma$.
    \end{itemize}
    Then $\bT$-closedness of a monomorphism $m$ is equivalent to whether $m$ belongs to $\rorth(\Lambda)$.
    Since $\Lambda$ is a small class of epimorphisms between $\lambda$-presentable objects, the statement follows from \cref{thm:PPFS}.
\end{proof}

\begin{proposition}\label{prop:closedmono_underlying}
    Let $\bS$ be a $\lambda$-ary partial Horn theory over an $S$-sorted $\lambda$-ary signature $\Sigma$.
    Let $(\Omega,E)$ be an $\bS$-relative $\lambda$-ary algebraic theory with the forgetful functor $U\colon\Alg(\Omega,E)\to\PMod\bS$.
    Then, the following hold:
    \begin{enumerate}
        \item\label{prop:closedmono_underlying-1}
            A morphism $f$ in $\Alg(\Omega,E)$ is $\pht{\Omega}{E}$-closed mono if and only if $Uf$ is $\bS$-closed mono.
        \item\label{prop:closedmono_underlying-2}
            A morphism $f$ in $\Alg(\Omega,E)$ is $\pht{\Omega}{E}$-dense if $Uf$ is $\bS$-dense.
    \end{enumerate}
    Here, $\pht{\Omega}{E}$ is the partial Horn theory as in \cref{def:pht_for_rat}.
\end{proposition}
\begin{proof}
    Since the domain of every partial operator in $\Omega$ is determined by a Horn formula over $\Sigma$, $\pht{\Omega}{E}$-closedness is equivalent to $\bS$-closedness, hence \cref{prop:closedmono_underlying-1} follows.
    To show \cref{prop:closedmono_underlying-2}, take a morphism $f$ in $\Alg(\Omega,E)$ such that $Uf$ is $\bS$-dense.
    By \cref{thm:dense_closedmono_factorization}, we can take a factorization $f=me$ such that $e$ is $\pht{\Omega}{E}$-dense and $m$ is $\pht{\Omega}{E}$-closed mono.
    By \cref{prop:closedmono_underlying-1}, $Um$ is $\bS$-closed mono but $Uf$ is $\bS$-dense, hence $Um$ becomes an isomorphism.
    One can prove that the forgetful functor $U$ is conservative by the definition of relative algebraic theories.
    Therefore, $m$ is an isomorphism, which proves \cref{prop:closedmono_underlying-2}.
\end{proof}

\begin{remark}\label{rem:dense_underlying}
    Unlike in the case of closed monomorphisms, the converse of \cref{prop:closedmono_underlying}\cref{prop:closedmono_underlying-2} fails.
    To present a counterexample, let us consider the $\bS_\quiv$-relative finitary algebraic theory $(\Omega,E)$ of \cref{eg:rat_smallcat}.
    Let $\Sigma\bN$ be the category obtained by regarding the additive monoid $\bN$ as a single object category.
    Let $\2\coloneq\{0<1\}$ and let $q\colon\2\to\Sigma\bN$ be the functor choosing an endomorphism $1$ in $\Sigma\bN$.
    Then, $q$ is dense as a morphism of small categories but not as a morphism of quivers.
    More seriously, $q$ is not even epi as a morphism of quivers.
\end{remark}

\begin{lemma}\label{lem:dense_between_presn}
    Let $\bT$ be a $\lambda$-ary partial Horn theory over an $S$-sorted $\lambda$-ary signature $\Sigma$.
    Then, every $\bT$-dense morphism between $\lambda$-presentable objects has the following expression:
    \begin{equation*}
        \repn{\tup{x}.\phi}_\bT \longarr[\repn{\tup{x}}_\bT] \repn{\tup{x}.\psi}_\bT
    \end{equation*}
\end{lemma}
\begin{proof}
    By \cref{cor:morphism_between_repn}, a morphism between $\lambda$-presentable objects has the following expression:
    \begin{equation}\label{eq:dense_between_presn_tau}
        \repn{\tup{x}.\phi}_\bT \longarr[\repn{\tup{\tau}}_\bT] \repn{\tup{y}.\psi}_\bT,
    \end{equation}
    where $\tup{x}.\phi$ and $\tup{y}.\psi$ are Horn formulas over $\Sigma$ with $\tup{x}=(x_i)_{i<\alpha}$ and $\tup{y}=(y_j)_{j<\beta}$.
    Suppose the morphism \cref{eq:dense_between_presn_tau} is $\bT$-dense.
    By \cref{rem:explaination_closedmono_dense}\cref{rem:explaination_closedmono_dense-3},
    for each $j<\beta$, we can take a term $\tup{x}.\sigma_j$ satisfying $[y_j]_\bT=[\sigma_j(\tup{\tau}/\tup{x})]_\bT$ in $\repn{\tup{y}.\psi}_\bT$, 
    i.e., the following is a $\PHL_\lambda$-theorem of $\bT$:
    \begin{equation*}
        \psi \seq{\tup{y}} y_j=\sigma_j(\tup{\tau}/\tup{x}).
    \end{equation*}
    Denote by $\tup{x}.\chi$ the following Horn formula:
    \begin{equation*}
        \tup{x}.\chi\coloneq\quad \tup{x}.\left( \psi(\tup{\sigma}/\tup{y})\wedge\bigwedge_{i<\alpha}x_i=\tau_i(\tup{\sigma}/\tup{y})\wedge\bigwedge_{j<\beta}\sigma_j\defined \right).
    \end{equation*}

    By virtue of the completeness theorem (\cref{thm:completeness_thm_for_PHL}), it is easy to check that the following are $\PHL_\lambda$-theorems of $\bT$:
    \begin{gather}
        \psi \seq{\tup{y}} \bigwedge_{i<\alpha}\tau_i\defined,
        \quad\quad
        \psi \seq{\tup{y}} \chi(\tup{\tau}/\tup{x});
        \label{eq:welldefinedness_tau}
        \\
        \chi \seq{\tup{x}} \bigwedge_{j<\beta}\sigma_j\defined,
        \quad\quad
        \chi \seq{\tup{x}} \psi(\tup{\sigma}/\tup{y});
        \label{eq:welldefinedness_sigma}
        \\
        \psi\seq{\tup{y}}\bigwedge_{j<\beta}y_j=\sigma_j(\tup{\tau}/\tup{x}),
        \quad\quad
        \chi\seq{\tup{x}}\bigwedge_{i<\alpha}x_i=\tau_i(\tup{\sigma}/\tup{y});
        \label{eq:isomorphism_sigma_tau}
        \\
        \chi\seq{\tup{x}}\phi.
        \label{eq:phi_implies_chi}
    \end{gather}
    \cref{eq:welldefinedness_sigma} and \cref{eq:welldefinedness_tau} being $\PHL_\lambda$-theorems of $\bT$ implies the well-definedness of the following morphisms:
    \begin{equation}\label{eq:mor_sigma_tau}
        \begin{tikzcd}
            \repn{\tup{x}.\chi}_\bT & \repn{\tup{y}.\psi}_\bT
            \arrow[from=1-2,to=1-1,shift left=2,"\repn{\tup{\sigma}}_\bT"]
            \arrow[from=1-1,to=1-2,shift left=2,"\repn{\tup{\tau}}_\bT"]
        \end{tikzcd}
    \end{equation}
    Since \cref{eq:isomorphism_sigma_tau} are $\PHL_\lambda$-theorems of $\bT$, two morphisms in \cref{eq:mor_sigma_tau} are inverses of each other.
    \cref{eq:phi_implies_chi} yields a morphism $\repn{\tup{x}.\phi}_\bT\longarr[\repn{\tup{x}}_\bT]\repn{\tup{x}.\chi}_\bT$, and we have the following commutative diagram:
    \begin{equation*}
        \begin{tikzcd}
            &[-20pt] \repn{\tup{x}.\phi}_\bT &[-20pt] \\
            \repn{\tup{x}.\chi}_\bT & & \repn{\tup{y}.\psi}_\bT
            \arrow[from=1-2,to=2-1,"\repn{\tup{x}}_\bT"']
            \arrow[from=1-2,to=2-3,"\repn{\tup{\tau}}_\bT"]
            \arrow[from=2-1,to=2-3,"\repn{\tup{\tau}}_\bT"',"\cong"]
        \end{tikzcd}
    \end{equation*}
    This completes the proof.
\end{proof}

\subsection{Birkhoff's variety theorem for partial Horn theories}
We now generalize Birkhoff's variety theorem to partial Horn theories.

\begin{definition}
    A full subcategory $\E\subseteq\C$ is \emph{replete} if $X\in\E$ whenever $X\cong Y$ in $\C$ and $Y\in\E$. 
\end{definition}

\begin{proposition}\label{prop:E-reflective}
    Let $\A$ be a locally presentable category with a proper factorization system $(\bfE,\bfM)$.
    For any replete full subcategory $\E\subseteq\A$, the following are equivalent:
    \begin{enumerate}
        \item
        $\E\subseteq\A$ is $\bfE$-reflective.
        \item
        $\E\subseteq\A$ is closed under products and $\bfM$-subobjects.
    \end{enumerate}
\end{proposition}
\begin{proof}
    Since every locally presentable category is co-wellpowered (see \cite[1.58 Theorem]{adamek1994locally}), this follows from \cite[{}16.8]{adamekherrlichstrecker2006joy}.
\end{proof}

We omit the proof of the following well-known fact:
\begin{lemma}\label{lem:orth_closed_under_filcolim}
    Let $\C$ be a category with $\lambda$-filtered colimits and let $\Lambda\subseteq\mor\C$ be a class of epimorphisms with $\lambda$-presentable domain.
    Then, the orthogonality class $\orth{\Lambda}\subseteq\C$ is closed under $\lambda$-filtered colimits.
\end{lemma}

\begin{definition}
    Let $\rho\colon (S,\Sigma,\bS)\to (S',\Sigma',\bT)$ be a $\lambda$-ary theory morphism between $\lambda$-ary partial Horn theories.
    A \emph{$\rho$-relative ($\lambda$-ary) judgment} is a $\lambda$-ary Horn sequent $\phi^\rho\seq{\tup{x}^\rho}\psi$, where $\tup{x}.\phi$ is a Horn formula over $\Sigma$ and $\tup{x}^\rho.\psi$ is a Horn formula over $\Sigma'$.
\end{definition}

\begin{remark}
    $\rho$-relative judgments generalize the concept of $\bS$-relative judgments as in \cref{def:relative_alg_theory}\cref{def:relative_alg_theory-2}.
    Indeed, for an $\bS$-relative algebraic theory $(\Omega,E)$, we can define the associated theory morphism $\rho\colon\bS\to\pht{\Omega}{E}$ as the inclusion.
    Then, an $\bS$-relative judgment is the same as a $\rho$-relative judgment.
\end{remark}

\begin{definition}
    Let $U\colon\C\to\A$ be a functor.
    A morphism $f$ in $\C$ is a \emph{$U$-retraction} if $Uf$ is a retraction, i.e., there exists a morphism $s$ in $\A$ such that $(Uf)\circ s=\id$.
    Given a $U$-retraction $f\colon X\to Y$, $Y$ is called a \emph{$U$-retract} of $X$.
\end{definition}

We can now formulate our main result.
To the author's knowledge, the following theorem is a new result and not known in any kind of logic equivalent to partial Horn logic because the theorem strongly depends on the syntax of partial Horn logic, as mentioned in \cref{eg:birkhoff_thm_depends_on_syntax} later.

\begin{theorem}[Birkhoff-type theorem for partial Horn theories I]\label{thm:birkhoff_partialHorn}
    Let $\rho\colon\bS\to\bT$ be a theory morphism between $\lambda$-ary partial Horn theories.
    Then, for every replete full subcategory $\E\subseteq\PMod\bT$, the following are equivalent:
    \begin{enumerate}
        \item\label{thm:birkhoff_partialHorn-1}
        $\E$ is definable by $\rho$-relative $\lambda$-ary judgments, i.e.,
        there exists a set $\bT'$ of $\rho$-relative $\lambda$-ary judgments satisfying $\E=\PMod (\bT+\bT')$.
        \item\label{thm:birkhoff_partialHorn-2}
        $\E$ is a small-orthogonality class with respect to a family of $\bT$-dense morphisms
        \[
            \Lambda=\{F^\rho (\Gamma_i)\longarr[e_i]\Delta_i\}_{i\in I}
        \]
        such that all $\Gamma_i$ are $\lambda$-presentable in $\PMod\bS$.
        \item\label{thm:birkhoff_partialHorn-3}
        $\E\subseteq\PMod\bT$ is closed under products, $\bT$-closed subobjects, $U^\rho$-retracts, and $\lambda$-filtered colimits.
    \end{enumerate}
\end{theorem}
\begin{proof}
    {[\cref{thm:birkhoff_partialHorn-1}$\implies$\cref{thm:birkhoff_partialHorn-2}]}
    Consider the adjunction of \cref{thm:adjunction_induced_by_theory_mor}:
    \begin{equation}\label{eq:adj_F^rhoU^rho}
        \begin{tikzcd}[large]
            \PMod\bS\arrow[r,"F^\rho",shift left=7pt]\arrow[r,"\perp"pos=0.5,phantom] &[10pt]\PMod\bT\arrow[l,"U^\rho",shift left=7pt]
        \end{tikzcd}
    \end{equation}
    Let $\bT'=\{\,\phi^\rho_i\seq{\tup{x}^\rho_i}\psi_i\,\}_{i\in I}$.
    By the construction of $F^\rho$ in \cref{thm:adjunction_induced_by_theory_mor}, we can consider $F^\rho\repn{\tup{x}_i.\phi_i}_\bS=\repn{\tup{x}^\rho_i.\phi^\rho_i}_\bT$.
    By \cref{prop:validity_for_PHL}, $\PMod (\bT+\bT')\subseteq\PMod\bT$ is an orthogonality class with respect to the following class of morphisms:
    \begin{equation*}
        \Lambda\coloneq\{~F^\rho\repn{\tup{x}_i.\phi_i}_\bS \arr[\repn{\tup{x}^\rho_i}_\bT] \repn{\tup{x}^\rho_i.\phi^\rho_i\wedge\psi_i}_\bT~\}_{i\in I}.
    \end{equation*}
    By \cref{rem:explaination_closedmono_dense}\cref{rem:explaination_closedmono_dense-3}, it follows that all $\repn{\tup{x}^\rho_i}_\bT$ are $\bT$-dense; hence \cref{thm:birkhoff_partialHorn-2} holds.

    {[\cref{thm:birkhoff_partialHorn-2}$\implies$\cref{thm:birkhoff_partialHorn-1}]}
    For each $i\in I$, choose a Horn formula $\tup{x}_i.\phi_i$ satisfying $\Gamma_i\cong\repn{\tup{x}_i.\phi_i}_\bS$.
    By \cref{thm:co-intersection}, $e_i\in{F^\rho (\Gamma_i)}/\PMod\bT$ can be presented as the following $\lambda$-filtered colimit of $\bT$-dense morphisms $e_i^{(j)}$ whose codomain is $\lambda$-presentable:
    \[
        e_i=\Colim{j}e_i^{(j)} \incat{{F^\rho (\Gamma_i)}/\PMod\bT}.
    \]
    By \cref{lem:dense_between_presn}, each $e_i^{(j)}$ has the following presentation:
    \begin{equation*}
        F^\rho\repn{\tup{x}_i.\phi_i}_\bS \longarr[\repn{\tup{x}^\rho_i}_\bT] \repn{\tup{x}^\rho_i.\psi_i^{(j)}}_\bT.
    \end{equation*}
    Orthogonality to all $e_i^{(j)}$ is equivalent to orthogonality to all $e_i$.
    Thus, taking
    \begin{equation*}
        \bT'\coloneq\{\, \phi^\rho_i\seq{\tup{x}^\rho_i}\psi_i^{(j)} \,\}_{i,j},
    \end{equation*}
    we have $\PMod (\bT+\bT')=\orth{\Lambda}=\E$.
    
    {[\cref{thm:birkhoff_partialHorn-2}$\implies$\cref{thm:birkhoff_partialHorn-3}]}
    Orthogonality classes are in general closed under products.
    Since $U^\rho$ preserves $\lambda$-filtered colimits, $F^\rho$ preserves $\lambda$-presentable objects.
    Thus, $\E=\orth{\Lambda}\subseteq\PMod\bT$ is closed under $\lambda$-filtered colimits by \cref{lem:orth_closed_under_filcolim}.
    In the following, we show that it is also closed under $\bT$-closed subobjects and $U^\rho$-retracts.
    
    We first show that $\orth{\Lambda}\subseteq\PMod\bT$ is closed under $\bT$-closed subobjects.
    Let $m\colon M\hookrightarrow N$ be a $\bT$-closed monomorphism in $\PMod\bT$, and assume $N\in\orth{\Lambda}$.
    Take a morphism $f\colon F^\rho (\Gamma_i)\to M$ arbitrarily.
    By $N\in\orth{\Lambda}$, there exists a unique morphism $g\colon\Delta_i\to N$ which makes the following diagram
    \begin{equation}\label{eq:comm_square_xm}
        \begin{tikzcd}[huge]
            F^\rho (\Gamma_i) & M \\
            \Delta_i & N
            \arrow[from=1-1,to=2-1,"e_i"']
            \arrow[from=1-2,to=2-2,hook',"m"]
            \arrow[from=1-1,to=1-2,"f"]
            \arrow[from=2-1,to=2-2,dashed,"\exists !g"']
        \end{tikzcd}
    \end{equation}
    commute in $\PMod\bT$.
    Since $e_i$ is $\bT$-dense, \cref{thm:dense_closedmono_factorization} ensures the unique existence of a diagonal filler for \cref{eq:comm_square_xm}.
    This proves that $M$ belongs to $\orth{\Lambda}$.

    We next show that $\orth{\Lambda}\subseteq\PMod\bT$ is closed under $U^\rho$-retracts.
    Let $p\colon M\to N$ in $\PMod\bT$ be a $U^\rho$-retraction and assume $M\in\orth{\Lambda}$.
    To prove $N\in\orth{\Lambda}$, take a morphism $f\colon F^\rho (\Gamma_i)\to N$.
    Consider the morphism $f^\flat\colon\Gamma_i\to U^\rho N$ corresponding to $f$ by the adjunction $F^\rho\dashv U^\rho$.
    Since $U^\rho p$ is a retraction, there exists a morphism $g\colon\Gamma_i\to U^\rho M$ satisfying $(U^\rho p)g=f^\flat$.
    Consider the morphism $g^\sharp\colon F^\rho (\Gamma_i)\to M$ corresponding to $g$ by $F^\rho\dashv U^\rho$.
    Since $M\in\orth{\Lambda}$, there exists a unique $h\colon\Delta_i\to M$ such that $he_i=g^\sharp$.
    Then $ph e_i=pg^\sharp=f$ holds.
    Moreover, such $ph$ is unique since $ e_i$ is an epimorphism.
    Therefore, $N$ is orthogonal to $ e_i\,(\forall i\in I)$.
    \begin{equation*}
        \begin{tikzcd}[huge]
            \Gamma_i\arrow[r,"\exists g",dashed]\arrow[rd,"f^\flat"',pos=0.7]  &  U^\rho M\arrow[d,"U^\rho p"]  \\
            &  U^\rho N
        \end{tikzcd}\incat{\PMod\bS}
        \quad\quad
        \begin{tikzcd}[huge]
            F^\rho (\Gamma_i)\arrow[d," e_i"']\arrow[r,"g^\sharp"]\arrow[rd,"f"',pos=0.7]  &  M\arrow[d,"p"]  \\
            \Delta_i\arrow[ru,"\exists !h",pos=0.7,crossing over,dashed]  &  N
        \end{tikzcd}\incat{\PMod\bT}
    \end{equation*}

    {[\cref{thm:birkhoff_partialHorn-3}$\implies$\cref{thm:birkhoff_partialHorn-2}]}
    In what follows, let $\A\coloneq\PMod\bS$.
    Since $\E\subseteq\PMod\bT$ is closed under products and $\bT$-closed subobjects, 
    \cref{prop:E-reflective} claims that the adjoint
    \begin{equation*}
        \begin{tikzcd}[large]
            \PMod\bT\arrow[r,"r",shift left=7pt]\arrow[r,"\perp"pos=0.5,phantom] &[10pt]\E\arrow[l,shift left=7pt,hook']
        \end{tikzcd}
    \end{equation*}
    exists, and the component of its unit $e$ at $M$
    \begin{equation*}
        M\longarr[e_M]rM\incat{\PMod\bT}
    \end{equation*}
    is $\bT$-dense.
    Consider the following classes of morphisms in $\PMod\bT$:
    \begin{gather*}
        \Lambda\coloneq\{\, F^\rho A\longarr[e_{F^\rho A}]rF^\rho A \,\}_{A\in\lpresn{\A}},
        \\
        \Lambda^*\coloneq\{\, F^\rho X\longarr[e_{F^\rho X}]rF^\rho X \,\}_{X\in\A}.
    \end{gather*}
    Here, $\lpresn{\A}$ denotes the full subcategory of $\A$ consisting of all $\lambda$-presentable objects.
    Every $X\in\A$ can be presented as a $\lambda$-filtered colimit $X=\Colim{I\in\bI}X_I$ in which $X_I$ are $\lambda$-presentable.
    Since $\E\subseteq\PMod\bT$ is closed under $\lambda$-filtered colimits, 
    $e_{F^\rho X}=\Colim{I\in\bI}e_{F^\rho X_I}$ is a $\lambda$-filtered colimit in the arrow category $(\PMod\bT)^\to$.
    Thus, $\orth{\Lambda}=\orth{\Lambda^*}$ holds.

    Take an object $M\in\PMod\bT$ satisfying $M\in\orth{\Lambda^*}$.
    Let $\epsilon$ denote the counit of the adjunction \cref{eq:adj_F^rhoU^rho}.
    By $e_{F^\rho U^\rho (M)}\in\Lambda^*$, there exists a unique morphism $p$ which makes the following diagram
    \begin{equation*}
        \begin{tikzcd}[huge]
            F^\rho U^\rho(M)\arrow[d,"e_{F^\rho U^\rho(M)}"']\arrow[r,"\epsilon_M"] & M \\
            rF^\rho U^\rho(M)\arrow[ur,dashed,"\exists !p"'] &
        \end{tikzcd}
    \end{equation*}
    commute in $\PMod\bT$.
    Since $U^\rho\epsilon_M$ is a retraction, $U^\rho p$ is also a retraction, which implies that $p$ is a $U^\rho$-retraction.
    Since $\E\subseteq\PMod\bT$ is closed under $U^\rho$-retracts, we have $M\in\E$.
    By the above argument, we have $\orth{\Lambda}=\E$.
\end{proof}

The above theorem has several useful corollaries.
We obtain the first corollary by taking $\rho$ as the trivial one $(S,\varnothing,\varnothing)\to (S,\Sigma,\bT)$:
\begin{corollary}\label{cor:cor_of_birkhoff_1}
    Let $\bT$ be a $\lambda$-ary partial Horn theory over $\Sigma$.
    Then, for every replete full subcategory $\E\subseteq\PMod\bT$, the following are equivalent:
    \begin{enumerate}
        \item
        $\E$ is definable by $\lambda$-ary Horn formulas, i.e.,
        there exists a set $E$ of $\lambda$-ary Horn formulas over $\Sigma$ satisfying $\E=\PMod (\bT+\bT')$, where $\bT'\coloneq\{ \top\seq{\tup{x}}\phi\}_{\tup{x}.\phi\in E}.$
        \item
        $\E\subseteq\PMod\bT$ is closed under products, $\Sigma$-closed subobjects, surjections, and $\lambda$-filtered colimits.
    \end{enumerate}
\end{corollary}

Taking $\rho$ as the identity $\bT\to\bT$, we obtain the second corollary:
\begin{corollary}\label{cor:cor_of_birkhoff_2}
    Let $\bT$ be a $\lambda$-ary partial Horn theory over $\Sigma$.
    Then, for every replete full subcategory $\E\subseteq\PMod\bT$, the following are equivalent:
    \begin{enumerate}
        \item
        $\E$ is definable by $\lambda$-ary Horn sequents, i.e.,
        there exists a set $\bT'$ of $\lambda$-ary Horn sequents over $\Sigma$ satisfying $\E=\PMod (\bT+\bT')$.
        \item
        $\E\subseteq\PMod\bT$ is closed under products, $\Sigma$-closed subobjects, and $\lambda$-filtered colimits.
    \end{enumerate}
\end{corollary}
\begin{proof}
    Since every split monomorphism is $\Sigma$-closed, being closed under $\Sigma$-closed subobjects implies being closed under retracts.
    Thus, this follows from \cref{thm:birkhoff_partialHorn}.
\end{proof}

\begin{example}\label{eg:birkhoff_thm_depends_on_syntax}
    Given a replete full subcategory $\E$ of a locally presentable category $\A$, 
    whether $\E\subseteq\A$ is definable by Horn sequents depends on the choice of a partial Horn theory $\bT$ for $\A$.
    For example, let $\A\coloneq\Mon$ be the category of monoids and let $\E\coloneq\Grp$ be the category of groups, and consider the finitary partial Horn theories $\bT_\mon$ and $\bT'_\mon$ as in \cref{eg:PHT_for_monoids}.
    Then $\Grp\subseteq\PMod\bT'_\mon$ is definable by the following Horn sequent:
    \begin{gather*}
        \top\seq{x} x^{-1}\defined.
    \end{gather*}
    On the other hand, $\Grp\subseteq\PMod\bT_\mon$ cannot be definable by Horn sequents, because $\bN$ is a $\bT_\mon$-closed subobject of $\bZ$ and not a group even though $\bZ$ is a group.
\end{example}

Taking $\rho$ as a ``relative algebraic theory,'' we get the following theorem.
This theorem is a generalization of Birkhoff's variety theorem from classical algebraic theories to our relative algebraic theories.

\begin{theorem}[Birkhoff-type theorem for relative algebraic theories I]\label{thm:birkhoff_for_rat}
    Let $\bS$ be a $\lambda$-ary partial Horn theory over an $S$-sorted $\lambda$-ary signature $\Sigma$.
    Let $(\Omega,E)$ be an $\bS$-relative $\lambda$-ary algebraic theory with the forgetful functor $U\colon \Alg(\Omega,E)\to\PMod\bS$.
    Then, for every replete full subcategory $\E\subseteq\Alg(\Omega,E)$, the following are equivalent:
    \begin{enumerate}
        \item
        $\E$ is definable by $\bS$-relative $\lambda$-ary judgments, i.e.,
        there exists a set $E'$ of $\bS$-relative $\lambda$-ary judgments satisfying $\E=\Alg(\Omega,E+E')$.
        \item
        $\E\subseteq\Alg(\Omega,E)$ is closed under products, $\Sigma$-closed subobjects, $U$-retracts, and $\lambda$-filtered colimits.
    \end{enumerate}
\end{theorem}
\begin{proof}
    Let $\pht{\Omega}{E}$ be the partial Horn theory as in \cref{def:pht_for_rat} and consider the extension $\rho\colon (S,\Sigma,\bS)\to (S,\Sigma+\Omega,\pht{\Omega}{E})$.
    Then, the forgetful functor $U$ coincides with $U^\rho$.
    Since the domain of every operator in $\Omega$ is determined by a Horn formula over $\Sigma$, $(\Sigma+\Omega)$-closedness is equivalent to $\Sigma$-closedness.
    Thus, we have proved this theorem by \cref{thm:birkhoff_partialHorn}.
\end{proof}

We also obtain two arity-free versions of \cref{thm:birkhoff_partialHorn}: bounded one (\cref{cor:bounded_birkhoff}) and unbounded one (\cref{cor:unbounded_birkhoff}).
Under Vop\v{e}nka's principle, these two versions coincide.
\begin{corollary}\label{cor:bounded_birkhoff}
    Let $\rho\colon\bS\to\bT$ be a theory morphism between $\lambda$-ary partial Horn theories.
    Then, for every replete full subcategory $\E\subseteq\PMod\bT$, the following are equivalent:
    \begin{enumerate}
        \item
        $\E$ is definable by $\rho$-relative $\mu$-ary judgments for some infinite regular cardinal $\mu\ge\lambda$.
        \item
        $\E$ is a small-orthogonality class with respect to a family of $\bT$-dense morphisms
        \[
            \Lambda=\{F^\rho (\Gamma_i)\longarr[e_i]\Delta_i\}_{i\in I}.
        \]
        \item
        $\E\subseteq\PMod\bT$ is closed under products, $\bT$-closed subobjects, $U^\rho$-retracts, and $\mu$-filtered colimits for some infinite regular cardinal $\mu\ge\lambda$.
    \end{enumerate}
\end{corollary}

\begin{corollary}\label{cor:unbounded_birkhoff}
    Let $\rho\colon\bS\to\bT$ be a theory morphism between $\lambda$-ary partial Horn theories.
    Then, for every replete full subcategory $\E\subseteq\PMod\bT$, the following are equivalent:
    \begin{enumerate}
        \item
        $\E$ is definable by a (not necessarily small) family of $\rho$-relative judgments.
        \item\label{cor:unbounded_birkhoff-2}
        $\E$ is an orthogonality class with respect to a (not necessarily small) family of $\bT$-dense morphisms
        \[
            \Lambda=\{F^\rho (\Gamma_i)\longarr[e_i]\Delta_i\}_{i\in I}.
        \]
        \item\label{cor:unbounded_birkhoff-3}
        $\E\subseteq\PMod\bT$ is closed under products, $\bT$-closed subobjects, and $U^\rho$-retracts.
    \end{enumerate}
\end{corollary}
\begin{proof}
    The proof is essentially the same as \cref{thm:birkhoff_partialHorn}.
\end{proof}

\begin{remark}
    The unbounded Birkhoff theorem (\cref{cor:unbounded_birkhoff}) also follows from a related result \cite[Theorem 3.16]{miliusurbat2019equational} if $U^\rho$ is conservative, which includes the case that $\rho$ is a ``relative algebraic theory.''
    To apply the framework \cite{miliusurbat2019equational}, let $\A_0=\A\coloneq\PMod\bT$, and let $(\E,\M)$ be the orthogonal factorization system on $\A$ of $\bT$-dense morphisms and $\bT$-closed monomorphisms, $\Lambda$ be the class of all (small) cardinal numbers, and $\X$ be the class of all free algebras $F^\rho(X)$.
    Then, the class of morphisms $\E_\X$ as in \cite{miliusurbat2019equational} coincides with the class of $\bT$-dense $U^\rho$-retractions.
    
    We now need the fact that, if $U^\rho$ is conservative, $U^\rho$-retractions are strong epi, hence $\bT$-dense.
    To show this, let $q$ be a $U^\rho$-retraction.
    Since every locally presentable category has (strong epi, mono)-factorization system \cite[1.61 Proposition]{adamek1994locally}, we can take a factorization $q=me$ such that $e$ is strong epi and $m$ is mono.
    Then, $Um$ is also mono since $U$ is a right adjoint.
    Since $Uq$ is a retraction, $Um$ must be an isomorphism.
    Thus, if $U$ is conservative, $m$ is also an isomorphism, which shows that $q$ is a strong epimorphism.

    As a result, $\E_\X$ coincides with the class of all $U^\rho$-retractions, and we can verify all assumptions required in \cite{miliusurbat2019equational}.
    Then, the notion of \emph{varieties} considered in \cite{miliusurbat2019equational} coincides with ours, and the equivalence between \cref{cor:unbounded_birkhoff}\cref{cor:unbounded_birkhoff-2} and \cref{cor:unbounded_birkhoff-3} follows from \cite[Theorem 3.16]{miliusurbat2019equational} directly.
\end{remark}

\section{A syntactic description of accessible monads}\label{section:monad}
Our next goal is to establish an equivalence between our relative algebraic theories and accessible monads (\cref{thm:equiv_between_monad_and_rat}).
We split its proof into two directions: from relative algebraic theories to monads, and from monads to relative algebraic theories.
In the latter direction, Birkhoff's theorem for relative algebras (\cref{thm:birkhoff_for_rat}) plays a crucial role.

\subsection{From relative algebraic theories to monads}
Throughout this subsection, we fix an $S$-sorted $\lambda$-ary signature $\Sigma$ and a $\lambda$-ary partial Horn theory $\bS$ over $\Sigma$.
We now prove that every $\bS$-relative $\lambda$-ary algebraic theory yields a $\lambda$-ary monad on $\PMod\bS$.

\begin{definition}
    Let $\lambda$ be an infinite regular cardinal.
    A functor (or monad) is called \emph{$\lambda$-ary} or \emph{$\lambda$-accessible} if it preserves $\lambda$-filtered colimits.
\end{definition}

Our goal in this subsection is to prove that the category of models of an $\bS$-relative $\lambda$-ary algebraic theory is (strictly) $\lambda$-ary monadic over $\PMod\bS$.
This is one direction of our main theorem (\cref{thm:equiv_between_monad_and_rat}).

\begin{notation}
    Given an endofunctor $H\colon\C\to\C$, we will denote by $\Alg H$ the inserter from $H$ to $\Id_\C$, i.e., 
    $\Alg H$ is the category whose objects are a pairs $(X,x)$ consisting of:
    \begin{itemize}
        \item
        an object $X\in\C$ and
        \item
        a morphism $H(X)\arr[x] X$ in $\C$,
    \end{itemize}
    and whose morphisms $f\colon (X,x)\to (Y,y)$ are given by a morphism $X\arr[f] Y$ in $\C$ such that the following diagram
    \begin{equation*}
        \begin{tikzcd}
            H(X)\arrow[d,"x"']\arrow[r,"H(f)"] & H(Y)\arrow[d,"y"] \\
            X\arrow[r,"f"'] & Y
        \end{tikzcd}
    \end{equation*}
    commutes in $\C$.
\end{notation}

\begin{definition}\label{def:endofunc_for_signature}
    Let $\Omega$ be an $\bS$-relative $\lambda$-ary signature.
    We now define a $\lambda$-ary endofunctor $H_\Omega\colon \PMod\bS\to\PMod\bS$ by the following:
    \begin{equation*}
        H_\Omega(A)\coloneq\coprod_{\omega\in\Omega}\intpn{\ar(\omega)}{A}\copower\repn{x\ofsort\outsort(\omega).\top}_\bS.
    \end{equation*}
    Here, $\intpn{\ar(\omega)}{A}\copower\repn{x\ofsort\outsort(\omega).\top}_\bS$ is the copower of $\repn{x\ofsort\outsort(\omega).\top}_\bS$ by the set $\intpn{\ar(\omega)}{A}$.
\end{definition}

\begin{lemma}\label{lem:iso_alg_alg}
    For every $\bS$-relative $\lambda$-ary signature $\Omega$, there exists an isomorphism of categories $\Alg H_\Omega\cong\Alg\Omega$ that commutes with the forgetful functors.
\end{lemma}
\begin{proof}
    For each $A\in\PMod\bS$, the following data correspond bijectively:
    \begin{center}
    \renewcommand{\arraystretch}{1.3}
    \begin{tabular}{c}
        $H_\Omega(A)\longarr A\incat{\PMod\bS}$
        \\
        \hline\hline
        $\intpn{\ar(\omega)}{A}\copower\repn{x\ofsort\outsort(\omega).\top}_\bS \longarr A\incat{\PMod\bS}\quad(\omega\in\Omega)$
        \\
        \hline\hline
        $\intpn{\ar(\omega)}{A}\longarr\PMod\bS ( \repn{x\ofsort\outsort(\omega).\top}_\bS , A )\incat{\Set}\quad(\omega\in\Omega)$
        \\
        \hline\hline
        $\intpn{\ar(\omega)}{A}\longarr A_{\outsort(\omega)}\incat{\Set}\quad(\omega\in\Omega)$
    \end{tabular}
    \end{center}
    This gives a desired isomorphism of categories.
\end{proof}

\begin{remark}\label{rem:alg_omegaE_orthogonal}
    Let $(\Omega,E)$ be an $\bS$-relative $\lambda$-ary algebraic theory and let $\A\coloneq\PMod\bS$.
    By \cref{lem:iso_alg_alg}, we have a left adjoint $\bF$ to the forgetful functor $U$ (See \cite[{}2.11]{Adamekporst2001fromvarieties}):
    \begin{equation*}
        \begin{tikzcd}[large]
            \A\arrow[r,"\bF",shift left=7pt]\arrow[r,"\perp"pos=0.5,phantom] &[10pt]\Alg\Omega.\arrow[l,"U",shift left=7pt]
        \end{tikzcd}
    \end{equation*}
    Note that $U$ is strictly monadic (see \cite[{}20.57]{adamekherrlichstrecker2006joy}) and that $U$ preserves $\lambda$-filtered colimits.
    Let $E=\{\phi_i\seq{\tup{x}_i}\psi_i\}_{i\in I}$.
    Since
    \begin{equation*}
        \Alg\Omega(\bF\repn{\tup{x}_i.\phi_i}_\bS , \bA)
        \cong \A(\repn{\tup{x}_i.\phi_i}_\bS , U\bA)
        \cong \intpn{\tup{x}_i.\phi_i}{U\bA}
        = \intpn{\tup{x}_i.\phi_i}{\bA}
    \end{equation*}
    holds naturally, we can assume $\bF\repn{\tup{x}_i.\phi_i}_\bS = \repn{\tup{x}_i.\phi_i}$ in $\Alg\Omega$.
    Here, $\repn{\tup{x}_i.\phi_i}$ is the abbreviation for $\repn{\tup{x}_i.\phi_i}_{\pht{\Omega}{\varnothing}}$, where $\pht{\Omega}{\varnothing}$ is the partial Horn theory for $\Alg\Omega$ as in \cref{def:pht_for_rat}.
    By \cref{prop:validity_for_PHL}, $\Alg(\Omega,E)\subseteq\Alg\Omega$ is the full subcategory of objects orthogonal to the following:
    \begin{equation*}
        \Lambda\coloneq\{~\bF\repn{\tup{x}_i.\phi_i}_\bS \arr[\repn{\tup{x}_i}] \repn{\tup{x}_i.\psi_i}~\}_{i\in I}.
    \end{equation*}
    Since every small-orthogonality class of a locally presentable category is reflective (\cite[Theorem 5.4.7]{borceux1994handbook1}), we get the following adjunction:
    \begin{equation*}
        \begin{tikzcd}[large]
            \Alg\Omega\arrow[r,"r",shift left=7pt]\arrow[r,"\perp"pos=0.5,phantom] &[10pt]\Alg(\Omega,E)~(=\orth{\Lambda}).\arrow[l,shift left=7pt,hook']
        \end{tikzcd}\qedhere
    \end{equation*}
\end{remark}

\begin{theorem}\label{thm:from_alg_to_monad}
    Let $(\Omega,E)$ be an $\bS$-relative $\lambda$-ary algebraic theory.
    Then, $\Alg(\Omega,E)$ is a strictly $\lambda$-ary monadic category over $\PMod\bS$.
\end{theorem}
\begin{proof}
    We follow the notation used in \cref{rem:alg_omegaE_orthogonal}.
    We have the following adjunctions:
    \begin{equation*}
        \begin{tikzcd}
            \A\arrow[r,"\bF",shift left=7pt]\arrow[r,"\perp"pos=0.5,phantom]
            &[10pt]\Alg\Omega\arrow[l,"U",shift left=7pt]\arrow[r,"r",shift left=7pt]\arrow[r,"\perp"pos=0.5,phantom]
            &[10pt]\Alg(\Omega,E)\arrow[l,shift left=7pt,hook',"\iota"]
        \end{tikzcd}
    \end{equation*}
    and $U\iota$ preserves $\lambda$-filtered colimits.
    To prove that $U\iota$ is monadic, we use Beck's strict monadicity theorem (\cite{maclane1998working} VI.7 Theorem 1).
    We claim that $\Alg(\Omega,E)\subseteq\Alg\Omega$ is closed under $U$-split coequalizers; hence $U\iota$ strictly creates $U$-split coequalizers.
    Indeed, given a $U$-split coequalizer:
    \begin{equation*}
    \begin{tikzcd}[column sep=large, row sep=large]
        \bA\arrow[r,shift left=4pt,"f"]\arrow[r,shift right=4pt,"g"']  &  \bB\arrow[r,"q"]  &  \bC
    \end{tikzcd}\incat{\Alg\Omega}
    \end{equation*}
    with $\bA,\bB\in\Alg(\Omega,E)$, $Uq$ is a retraction and thus $q$ is a $U$-retraction.
    By \cref{thm:birkhoff_for_rat}, $\Alg(\Omega,E)\subseteq\Alg\Omega$ is closed under $U$-retracts, which shows that $\bC\in\Alg(\Omega,E)$.
\end{proof}

\subsection{From monads to relative algebraic theories}
We now prove that every $\lambda$-ary monad on a locally $\lambda$-presentable category arises from a relative $\lambda$-ary algebraic theory.
This is a converse of the result in the previous subsection.
In this subsection, we fix an $S$-sorted $\lambda$-ary signature $\Sigma$ and a $\lambda$-ary partial Horn theory $\bS$ over $\Sigma$ again.

\begin{definition}
    Let $T$ be a $\lambda$-ary monad on $\PMod\bS$.
    We now define an $\bS$-relative $\lambda$-ary signature $\Omega_T$ for the $\lambda$-ary monad $T$:
    For each Horn formula $\tup{x}.\phi$ over $\Sigma$ and each sort $s\in S$, we have a set $(T\repn{\tup{x}.\phi}_\bS)_s$.
    We regard each element $\omega\in(T\repn{\tup{x}.\phi}_\bS)_s$ as an operator with $\ar(\omega)\coloneq\tup{x}.\phi$ and $\outsort(\omega)\coloneq s$, and define $\Omega_T$ to be the set of all such operators:
    \begin{equation*}
        \Omega_T\coloneq\coprod_{\substack{\text{Horn formula }\tup{x}.\phi\text{ over }\Sigma  \\  \text{sort }s\in S}} (T\repn{\tup{x}.\phi}_\bS)_s
    \end{equation*}
\end{definition}

\begin{definition}
    Let $T$ be a $\lambda$-ary monad on $\A\coloneq\PMod\bS$.
    We now define a natural transformation $\alpha_T\colon H_{\Omega_T}\Rightarrow T$, where $H_{\Omega_T}$ is the $\lambda$-ary endofunctor for $\Omega_T$ as in \cref{def:endofunc_for_signature}.
    In what follows, we regard each $\omega\in\Omega_T$ as a morphism $\repn{x\ofsort\outsort(\omega).\top}_\bS\to T\repn{\ar(\omega)}_\bS$ in $\A$ by \cref{prop:repn_obj_represents_intpn}.
    First, for each $A\in\A$ and $\omega\in\Omega_T$, we have the following map:
    \begin{equation*}
        \intpn{\ar(\omega)}{A} \cong \A(\repn{\ar(\omega)}_\bS, A) \ni f\quad\mapsto\quad (Tf)\circ\omega\in \A(\repn{x\ofsort\outsort(\omega).\top}_\bS, TA)
    \end{equation*}
    Then, the following bijective correspondence yields a natural transformation $\alpha\colon H_{\Omega_T}\Rightarrow T$:
    \begin{center}
        \renewcommand{\arraystretch}{1.5}
        \begin{tabular}{c}
            $\intpn{\ar(\omega)}{A}\longarr\A( \repn{x\ofsort \outsort(\omega).\top}_\bS , TA )\incat{\Set}\quad(\omega\in\Omega_T)$
            \\
            \hline\hline
            $\intpn{\ar(\omega)}{A}\copower\repn{x\ofsort \outsort(\omega).\top}_\bS \longarr TA\incat{\A}\quad(\omega\in\Omega_T)$
            \\
            \hline\hline
            $H_{\Omega_T}(A)\longarr[\alpha_{T,A}]TA\incat{\A}$
        \end{tabular}
    \end{center}
\end{definition}

\begin{lemma}\label{lem:alpha_is_dense}
    Let $T$ be a $\lambda$-ary monad on $\A\coloneq\PMod\bS$.
    Then, every component of the natural transformation $\alpha_T$ is $\Sigma$-dense and in particular epic.
\end{lemma}
\begin{proof}
    We have to show that $\alpha_{T,A}\colon H_{\Omega_T}(A)\to TA$ is $\Sigma$-dense for every $A\in\A$.
    We first show the case where $A$ is $\lambda$-presentable.
    By \cref{thm:repn_enumerates_presn}, it suffices to show the case $A=\repn{\tup{x}.\phi}_\bS$, where $\tup{x}.\phi$ is a Horn formula over $\Sigma$.
    Let $\omega\in (T\repn{\tup{x}.\phi}_\bS)_s$ be an arbitrary element.
    Then, $\omega$ is an operator with $\ar(\omega)=\tup{x}.\phi$ and $\outsort(\omega)=s$.
    Let
    \[
        \iota\colon\repn{x\ofsort s.\top}_\bS \to \coprod_{\omega'}\A(\repn{\ar(\omega')}_\bS , A)\copower\repn{x\ofsort \outsort(\omega').\top}_\bS \cong H_{\Omega_T}(A)
    \]
    be the coprojection for the pair of $\omega$ and $\id\in\A(\repn{\tup{x}.\phi}_\bS, A)=\A(\repn{\ar(\omega)}_\bS,A)$.
    Then, by definition of $\alpha_T$, the following diagram
    \begin{equation*}
        \begin{tikzcd}
            \repn{x\ofsort s.\top}_\bS\arrow[rd,"\omega"']\arrow[r,"\iota"] & H_{\Omega_T}(A)\arrow[d,"\alpha_{T,A}"] \\
            & TA
        \end{tikzcd}
    \end{equation*}
    commutes in $\A$.
    In particular, $\omega\in (TA)_s$ lies in the image of $\alpha_{T,A}$.
    By \cref{rem:explaination_closedmono_dense}\cref{rem:explaination_closedmono_dense-3}, we now concludes that $\alpha_{T,A}$ is $\Sigma$-dense for every $\lambda$-presentable object $A\in\A$.

    We now turn to the general case.
    Given $A\in\A$, we can take a $\lambda$-filtered colimit $A=\Colim{I\in\bI}A_I$ such that each $A_I$ is $\lambda$-presentable.
    Since $H_{\Omega_T}$ and $T$ are $\lambda$-ary, we have $\alpha_{T,A}=\Colim{I\in\bI}\alpha_{T,A_I}$ in the arrow category $\A^\to$.
    Since the class of all $\Sigma$-dense morphisms is a left class of an orthogonal factorization system on $\A$ by \cref{thm:dense_closedmono_factorization}, any colimit of $\Sigma$-dense morphisms is also $\Sigma$-dense.
    Thus, we concludes that $\alpha_{T,A}$ is $\Sigma$-dense.
\end{proof}

\begin{lemma}\label{lem:componentwize_epi_induce_fullyfaithful}
    Let $H$ and $K$ be endofunctors on a category $\C$ and let $\alpha\colon H\Rightarrow K$ be a natural transformation whose components are epimorphisms in $\C$.
    Let $\funcAlg\alpha\colon\Alg K\to\Alg H$ be the induced functor which assigns to each object $(X,x)$, $(X,x\circ\alpha_X)$.
    Then, the following hold:
    \begin{enumerate}
        \item
        $\funcAlg\alpha$ is fully faithful and injective on objects.
        \item
        The image of $\funcAlg\alpha$ is replete in $\Alg H$.
    \end{enumerate}
\end{lemma}
\begin{proof}\quad
    \begin{enumerate}
        \item
        If $(X,x\circ\alpha_X)=(Y,y\circ\alpha_Y)$, then $X=Y$ and $x=y$ since $\alpha_X=\alpha_Y$ is an epimorphism.
        Thus, $\funcAlg\alpha$ is injective on objects.
        To prove fully faithfulness, let $(X,x)$ and $(Y,y)$ be objects in $\Alg K$ and let $f\colon (X,x\circ\alpha_X)\to (Y,y\circ\alpha_Y)$ be a morphism in $\Alg H$.
        Then, in the following diagram, both the outer rectangle and the upper square commute:
        \begin{equation*}
            \begin{tikzcd}
                HX\arrow[d,"\alpha_X"']\arrow[r,"Hf"] & HY\arrow[d,"\alpha_Y"] \\
                KX\arrow[d,"x"']\arrow[r,"Kf"] & KY\arrow[d,"y"] \\
                X\arrow[r,"f"] & Y
            \end{tikzcd}\incat{\C}
        \end{equation*}
        Since $\alpha_X$ is an epimorphism, the lower square of the above diagram also commutes, which finishes the proof.
        \item
        Let $(X,x)$ be an object in $\Alg K$ and let $f\colon (X,x\circ\alpha_X)\to (Y,y)$ be an isomorphism in $\Alg H$.
        Then, we can easily show that $y=f\circ x\circ (Kf^{-1})\circ\alpha_Y$.
        Thus, $(Y,y)$ lies in the image of $\funcAlg\alpha$.\qedhere
    \end{enumerate}
\end{proof}

\begin{lemma}\label{lem:nat_square_pushout}
    Let $F,G\colon\C\to\D$ be functors and let $\alpha\colon F\Rightarrow G$ be a natural transformation.
    Let $r\colon X\to Y$ be a retraction in $\C$ and assume that $\alpha_X$ is an epimorphism in $\D$.
    Then, the following diagram forms a pushout square:
    \begin{equation*}
        \begin{tikzcd}
            FX\arrow[d,"\alpha_X"']\arrow[r,"Fr"] & FY\arrow[d,"\alpha_Y"] \\
            GX\arrow[r,"Gr"] & GY
        \end{tikzcd}\incat{\D}
    \end{equation*}
\end{lemma}
\begin{proof}
    Let $f\colon FY\to Z$ and $g\colon GX\to Z$ be morphisms in $\D$ such that $f\circ Fr=g\circ\alpha_X$.
    We have to construct a unique morphism $h$ which makes the following commute:
    \begin{equation}\label{eq:canonical_morphism_h}
        \begin{tikzcd}
            FX\arrow[d,"\alpha_X"']\arrow[r,"Fr"] & FY\arrow[d,"\alpha_Y"']\arrow[rdd,"f"] & \\
            GX\arrow[r,"Gr"]\arrow[rrd,"g"'] & GY\arrow[rd,dashed,"h"description] & \\
            & & Z
        \end{tikzcd}\incat{\D}
    \end{equation}
    Since the uniqueness of $h$ follows trivially, we only need to construct $h$.
    Let $s\colon Y\to X$ be a section of $r$ and define $h\coloneq g\circ Gs$.
    Then, the following commutes:
    \begin{equation*}
        \begin{tikzcd}
            FY\arrow[dd,"\alpha_Y"']\arrow[rd,"Fs"']\arrow[rr,equal] & & FY\arrow[dd,"f"] \\
            & FX\arrow[d,"\alpha_X"]\arrow[ru,"Fr"'] & \\
            GY\arrow[rr,"h"',shift right=3]\arrow[r,"Gs"] & GX\arrow[r,"g"] & Z
        \end{tikzcd}\incat{\D}
    \end{equation*}
    Since $\alpha_X$ is an epimorphism, $h\circ Gr=g$ holds.
    Thus, the diagram \cref{eq:canonical_morphism_h} commutes.
\end{proof}

\begin{lemma}\label{lem:equifier_closed_under}
    Let $F,G\colon\C\to\D$ be functors and let $\alpha,\beta\colon F\Rightarrow G$ be natural transformations.
    Let $\E\subseteq\C$ be the equifier of $\alpha$ and $\beta$, i.e.,
    $\E\subseteq\C$ is a full subcategory of $\C$ defined by $\E\coloneq\{ X\in\C \mid \alpha_X=\beta_X \}$.
    Then, the following hold:
    \begin{enumerate}
        \item\label{lem:equifier_closed_under-1}
        Let $m\colon X\to Y$ be a morphism in $\C$ such that $Gm$ is a monomorphism in $\D$.
        Then, $Y\in\E$ implies $X\in\E$.
        \item\label{lem:equifier_closed_under-2}
        Let $p\colon X\to Y$ be a morphism in $\C$ such that $Fp$ is an epimorphism in $\D$.
        Then, $X\in\E$ implies $Y\in\E$.
    \end{enumerate}
\end{lemma}
\begin{proof}
    Let $m\colon X\to Y$ be a morphism in $\C$ such that $Gm$ is a monomorphism.
    Consider the following diagram:
    \begin{equation*}
        \begin{tikzcd}
            FX\arrow[d,"\alpha_X"',shift right=1]\arrow[d,"\beta_X",shift left=1]\arrow[r,"Fm"] & FY\arrow[d,"\alpha_{Y}"',shift right=1]\arrow[d,"\beta_{Y}",shift left=1] \\
            GX\arrow[r,"Gm"] & GY
        \end{tikzcd}\incat{\D}
    \end{equation*}
    If $\alpha_Y=\beta_Y$ holds, then $(Gm)\circ\alpha_X=(Gm)\circ\beta_X$ holds; hence $\alpha_X=\beta_X$.
    This proves \cref{lem:equifier_closed_under-1}.
    \cref{lem:equifier_closed_under-2} is dual to \cref{lem:equifier_closed_under-1}.
\end{proof}

\begin{theorem}\label{thm:from_monad_to_alg}
    Let $T$ be a $\lambda$-ary monad on $\A\coloneq\PMod\bS$.
    Then, there exist a set $E$ of $\bS$-relative judgments and an isomorphism $\A^T\cong\Alg(\Omega_T,E)$ which commutes with forgetful functors.
    \begin{equation*}
        \begin{tikzcd}
            \A^T\arrow[rd,"U^T"']\arrow[r,phantom,"\cong"] &[-10pt] \Alg(\Omega_T,E)\arrow[d,"\text{forgetful}"sloped]\arrow[r,phantom,"\subseteq"] &[-10pt] \Alg\Omega_T\arrow[ld,"U"] \\[15pt]
            & \A &
        \end{tikzcd}
    \end{equation*}
\end{theorem}
\begin{proof}
    By \cref{lem:componentwize_epi_induce_fullyfaithful,lem:alpha_is_dense}, the functor $\funcAlg\alpha_T\colon\Alg T\to\Alg H_{\Omega_T}$ is fully faithful and injective on objects and has a replete image.
    Since $\Alg H_{\Omega_T}\cong\Alg\Omega_T$ by \cref{lem:iso_alg_alg}, $\Alg T$ is isomorphic to a replete full subcategory of $\Alg\Omega_T$.
    The Eilenberg--Moore category $\A^T$ is now isomorphic to a replete full subcategory of $\Alg\Omega_T$ since it is a replete full subcategory of $\Alg T$.
    Note that the inclusion $\A^T\hookrightarrow\Alg\Omega_T$ commutes with forgetful functors; that is, the following commutes:
    \begin{equation*}
        \begin{tikzcd}
            \A^T\arrow[rrd,"U^T"']\arrow[r,phantom,"\subseteq"] &[-10pt] \Alg T\arrow[rd,"\text{forgetful}"sloped]\arrow[r,"\funcAlg\alpha_T",hook] &[-10pt] \Alg H_{\Omega_T}\arrow[d,"\text{forgetful}"sloped]\arrow[r,phantom,"\cong"] &[-10pt] \Alg\Omega_T\arrow[ld,"U"] \\[25pt]
            & & \A &
        \end{tikzcd}
    \end{equation*}
    To conclude that there exists a set $E$ of $\bS$-relative judgments such that $\A^T\cong\Alg(\Omega_T,E)$, we use Birkhoff's theorem for relative algebraic theories (\cref{thm:birkhoff_for_rat}).
    That is, it suffices to show that both $\A^T\subseteq\Alg T$ and $\funcAlg\alpha_T\colon\Alg T\hookrightarrow\Alg H_{\Omega_T}$ are closed under products, $\Sigma$-closed subobjects, $U$-retracts and $\lambda$-filtered colimits.

    Since $\A^T$, $\Alg T$, and $\Alg H_{\Omega_T}$ are $\lambda$-ary monadic over $\A$, their forgetful functors create products and $\lambda$-filtered colimits.
    Therefore, both inclusions $\A^T\subseteq\Alg T$ and $\funcAlg\alpha_T\colon\Alg T\hookrightarrow\Alg H_{\Omega_T}$ are closed under products and $\lambda$-filtered colimits.

    To show that $\funcAlg\alpha_T\colon\Alg T\hookrightarrow\Alg H_{\Omega_T}$ is closed under $\Sigma$-closed subobjects, let $(Y,y)$ be an object in $\Alg T$ and let $m\colon (X,x)\to (Y,y\circ\alpha_{T,Y})$ be a morphism in $\Alg H_{\Omega_T}$ such that $m\colon X\to Y$ is a $\Sigma$-closed monomorphism in $\A$.
    Then, the following commutes:
    \begin{equation*}
        \begin{tikzcd}
            H_{\Omega_T}X\arrow[dd,"\alpha_{T,X}"']\arrow[rd,"H_{\Omega_T}(m)"description]\arrow[rr,"x"] & & X\arrow[dd,"m"] \\
            & H_{\Omega_T}Y\arrow[d,"\alpha_{T,Y}"] & \\
            TX\arrow[r,"Tm"] & TY\arrow[r,"y"] & Y
        \end{tikzcd}\incat{\A}
    \end{equation*}
    Since $\alpha_{T,X}$ is $\Sigma$-dense by \cref{lem:alpha_is_dense} and $m$ is $\Sigma$-closed, the above rectangle has a unique diagonal filler by \cref{thm:dense_closedmono_factorization}.
    Thus, $(X,x)$ lies in the image of $\funcAlg\alpha_T$; hence the inclusion $\funcAlg\alpha_T\colon\Alg T\hookrightarrow\Alg H_{\Omega_T}$ is closed under $\Sigma$-closed subobjects.

    To show that $\funcAlg\alpha_T\colon\Alg T\hookrightarrow\Alg H_{\Omega_T}$ is closed under $U$-retracts, let $(X,x)$ be an object in $\Alg T$ and let $r\colon (X,x\circ\alpha_{T,X})\to (Y,y)$ be a morphism in $\Alg H_{\Omega_T}$ such that $r\colon X\to Y$ is a retraction in $\A$.
    By \cref{lem:nat_square_pushout}, the naturality square of $\alpha_T$ for $r$ forms a pushout.
    Thus, we have the following canonical morphism $y'\colon TY\to Y$:
    \begin{equation*}
        \begin{tikzcd}
            H_{\Omega_T}X\arrow[rd,pos=1.00,phantom,"\ulcorner"]\arrow[d,"\alpha_{T,X}"']\arrow[r,"H_{\Omega_T}(r)"] & H_{\Omega_T}Y\arrow[d,"\alpha_{T,Y}"'description]\arrow[rdd,"y"] & \\
            TX\arrow[d,"x"']\arrow[r,"Tr"'] & TY\arrow[rd,dashed,"y'"'pos=0.3] & \\
            X\arrow[rr,"r"'] & & Y
        \end{tikzcd}\incat{\A}
    \end{equation*}
    Therefore, $(Y,y)$ lies in the image of $\funcAlg\alpha_T$; hence $\funcAlg\alpha_T\colon\Alg T\hookrightarrow\Alg H_{\Omega_T}$ is closed under $U$-retracts.

    It remains to show that $\A^T\subseteq\Alg T$ is closed under $\Sigma$-closed subobjects and $U$-retracts.
    Let $V\colon\Alg T\to\A$ be the forgetful functor and let $\xi\colon TV\Rightarrow V$ be the natural transformation defined by $\xi_{(X,x)}\colon TX\arr[x]X$.
    Let $\eta$ and $\mu$ denote the unit and multiplication of the monad $T$.
    The Eilenberg--Moore category $\A^T$ is now the double equifier of a pair $\id\colon V\Rightarrow V$ and $\xi\circ\eta V\colon V\Rightarrow V$ and a pair $\xi\circ T\xi\colon TTV\Rightarrow V$ and $\xi\circ\mu V\colon TTV\Rightarrow V$, i.e., $\A^T$ is the full subcategory of $\Alg T$ defined by 
    \[
        \A^T\coloneq\{ (X,x)\in\Alg T \mid \id_{(X,x)}=(\xi\circ\eta V)_{(X,x)}\text{~and~} (\xi\circ T\xi)_{(X,x)}=(\xi\circ\mu V)_{(X,x)} \}.
    \]
    Note that every $\Sigma$-closed monomorphism in $\Alg T$ is transferred to a $\Sigma$-closed monomorphism by $V$ and that every $U$-retraction in $\Alg T$ is transferred to a retraction by both $TTV$ and $V$.
    Thus, \cref{lem:equifier_closed_under} now shows that $\A^T\subseteq\Alg T$ is closed under $\Sigma$-closed subobjects and $U$-retracts.
    This completes the proof.
\end{proof}

\begin{notation}
    Given a category $\C$, let us denote by $\Mndl(\C)$ the category of $\lambda$-ary monads on $\C$ and monad morphisms in the sense of \cite[3 Section 6]{barr2005ttt}.
\end{notation}

Combining \cref{thm:from_alg_to_monad,thm:from_monad_to_alg}, we obtain the characterization theorem for our relative algebraic theories:

\begin{theorem}\label{thm:equiv_between_monad_and_rat}
    Let $\bS$ be a $\lambda$-ary partial Horn theory over $\Sigma$.
    Then, there is an equivalence $\Thl^\bS\simeq\Mndl(\PMod\bS)$ which makes the following commute up to isomorphism:
    \begin{equation*}
        \begin{tikzcd}
            \Thl^\bS\arrow[rd,"\funcAlg"'name=Alg]\arrow[rr,"\simeq"] &[-20pt] &[-20pt] \Mndl(\PMod\bS)\arrow[ld,"\funcEM"] \\
            & (\CAT/\PMod\bS)^\op &
            \arrow[from=Alg,to=1-3,phantom,"\cong"sloped]
        \end{tikzcd}
    \end{equation*}
    Here, $\funcEM$ is the functor that assigns to each $\lambda$-ary monad, its Eilenberg--Moore category.
\end{theorem}
\begin{proof}
    We now define a functor $K\colon\Thl^\bS\to\Mndl(\PMod\bS)$.
    For each $\bS$-relative algebraic theory $(\Omega,E)$, let $K(\Omega,E)\coloneq\mnd{\Omega}{E}$ be the monad induced by the following adjunction:
    \begin{equation*}
    \begin{tikzcd}[column sep=large, row sep=large]
        \PMod\bS\arrow[r,"\bF",shift left=7pt]\arrow[r,"\perp"pos=0.5,phantom] &[10pt]\Alg(\Omega,E)\arrow[l,"U",shift left=7pt]
    \end{tikzcd}
    \end{equation*}
    By \cref{thm:from_alg_to_monad}, we have $\Alg(\Omega,E)\cong\EM(\mnd{\Omega}{E})$.
    Since the functor $\funcEM$ is fully faithful by \cite[3 Theorem 6.3]{barr2005ttt}, for each morphism $[\rho]\colon (\Omega,E)\to (\Omega',E')$ in $\Thl^\bS$, we have a unique monad morphism $K[\rho]:\mnd{\Omega}{E}\to\mnd{\Omega'}{E'}$ which makes the following commute:
    \begin{equation*}
        \begin{tikzcd}
            \EM(\mnd{\Omega}{E})\arrow[d,phantom,"\cong"sloped] & & \EM(\mnd{\Omega'}{E'})\arrow[ll,"\funcEM(K\lbrack\rho\rbrack)"']\arrow[d,phantom,"\cong"sloped] \\[-10pt]
            \Alg(\Omega,E)\arrow[rd,"U"'] & & \Alg(\Omega',E')\arrow[ll,"\funcAlg\rho"']\arrow[ld,"U'"] \\
            & \PMod\bS &
        \end{tikzcd}
    \end{equation*}
    This yields a desired functor $K\colon\Thl^\bS\to\Mndl(\PMod\bS)$.
    Then, $K$ is fully faithful by \cref{thm:alg_fully_faithful} and essentially surjective by \cref{thm:from_monad_to_alg}; hence $K$ is an equivalence.
\end{proof}

\cref{thm:equiv_between_monad_and_rat} answers the question of when a theory morphism between partial Horn theories induces a monadic functor.
Indeed, the theorem concludes that, for a theory morphism $\rho\colon\bS\to\bT$ between partial Horn theories, the functor $U^\rho$ becomes (strictly) monadic precisely when $\rho$ is isomorphic to some $\bS$-relative algebraic theory.
Although this kind of correspondence between theories and monads is folklore, to the author's knowledge, this is not known explicitly in other kinds of logic equivalent to partial Horn theory, such as essentially algebraic theories or cartesian theories.

As a result of \cref{thm:equiv_between_monad_and_rat}, we get a syntactic presentation of $\lambda$-ary monadic categories over locally $\lambda$-presentable categories:
\pagebreak
\begin{corollary}\label{cor:equiv_between_monadicity_and_rat}
    Let $\bS$ be a $\lambda$-ary partial Horn theory.
    For each category $\C$, the following are equivalent:
    \begin{enumerate}
        \item $\C$ is $\lambda$-ary monadic (resp. strictly $\lambda$-ary monadic) over $\PMod\bS$.
        \item $\C$ is equivalent (resp. isomorphic) to $\Alg(\Omega,E)$ for some $\bS$-relative $\lambda$-ary algebraic theory $(\Omega,E)$.
    \end{enumerate}
\end{corollary}

\begin{corollary}\label{cor:S-rat_is_independent_of_S}
    The concept of relative algebraic theories is independent of the choice of $\bS$, i.e.,
    if there exists an equivalence $\PMod\bS\simeq\PMod\bT$ with two $\lambda$-ary partial Horn theories $\bS$ and $\bT$, then the following classes of categories coincide (up to equivalence):
    \begin{enumerate}
        \item Categories of models of $\bS$-relative $\lambda$-ary algebraic theories.
        \item Categories of models of $\bT$-relative $\lambda$-ary algebraic theories.
    \end{enumerate}
\end{corollary}

Thus, given a locally $\lambda$-presentable category $\A$, we may call an $\bS$-relative $\lambda$-ary algebraic theory an \emph{$\A$-relative $\lambda$-ary algebraic theory} as long as categories of models are concerned, taking an arbitrary $\lambda$-ary partial Horn theory $\bS$ such that $\A\simeq\PMod\bS$.

\section{Filtered colimit elimination from Birkhoff's variety theorem}\label{section:elimination}
Our last goal is to give a sufficient condition for eliminating closure under filtered colimits from Birkhoff's variety theorem.

\subsection{Pure quotients in locally presentable categories}
We recall the notion of \emph{pure epimorphisms} and basic properties of them.
Pure epimorphisms are a kind of epimorphic counterpart of pure monomorphisms and will play an important role in filtered colimit elimination.
Moreover, we will see that in the hierarchy of epimorphisms, pure epimorphisms lie between retractions and regular epimorphisms
in a locally presentable category just as pure monomorphisms lie between sections and regular monomorphisms.

\begin{definition}
    Let $\lambda$ be an infinite regular cardinal.
    A morphism $p\colon X\to Y$ in a category $\A$ is called a \emph{$\lambda$-pure epimorphism} if for every $\lambda$-presentable object $\Gamma\in\A$ and every morphism $f\colon\Gamma\to Y$, there exists a morphism $g\colon\Gamma\to X$ such that $p\circ g=f$.
\begin{equation*}
\begin{tikzcd}
    & X\arrow[d,"p"]\\
    \Gamma\arrow[ur,"\exists g",dashed]\arrow[r,"f"'] & Y
\end{tikzcd} 
\end{equation*}
Given a $\lambda$-pure epimorphism $p\colon X\to Y$, we say that $Y$ is a \emph{$\lambda$-pure quotient} of $X$.
\end{definition}

Here are some elementary properties of pure epimorphisms.
\begin{lemma}\label{lem:basic_property_locret}
    For every category $\A$, the following hold.
    \begin{enumerate}
        \item
        A composite $p\circ h$ is a $\lambda$-pure epimorphism whenever both $p$ and $h$ are $\lambda$-pure epimorphisms.
        \item
        If a composite $p\circ h$ is a $\lambda$-pure epimorphism, then so is $p$.
        \item\label{lem:basic_property_locret-2}
        Every retraction is a $\lambda$-pure epimorphism.
        \item\label{lem:basic_property_locret-3}
        If $p\colon X\to Y$ is a $\lambda$-pure epimorphism and $Y$ is $\lambda$-presentable, then $p$ is a retraction.
        \item\label{lem:basic_property_locret-4}
        $\lambda$-pure epimorphisms are stable under pullback, i.e., if a pullback square
        \begin{equation*}
            \begin{tikzcd}
                \cdot\arrow[r]\arrow[d,"p'"']\arrow[rd,pos=0.1,phantom,"\lrcorner"] &[-8pt] \cdot\arrow[d,"p"]\\[-5pt]
                \cdot\arrow[r] & \cdot
            \end{tikzcd}
        \end{equation*}
        is given and $p$ is a $\lambda$-pure epimorphism, then $p'$ is also a $\lambda$-pure epimorphism.
        \item\label{lem:basic_property_locret-5}
        Let $p=\Colim{I\in\bI}p_I$ be a $\lambda$-filtered colimit in the arrow category $\A^\to$ such that $p_I$ is a $\lambda$-pure epimorphism in $\A$ for all $I\in\bI$.
        Then $p$ is a $\lambda$-pure epimorphism in $\A$.
    \end{enumerate}
\end{lemma}
\begin{proof}
    The proof is immediate.
\end{proof}

\begin{proposition}[{\cite[Proposition 3]{adamek2004purequotient}}]\label{prop:locret_filtered_colim_of_retracts}
    Let $\A$ be a locally $\lambda$-presentable category.
    Then for a morphism $p\colon X\to Y$ in $\A$, the following are equivalent:
    \begin{enumerate}
        \item\label{prop:locret_filtered_colim_of_retracts-1}
        $p$ is a $\lambda$-pure epimorphism.
        \item\label{prop:locret_filtered_colim_of_retracts-2}
        $p$ is a $\lambda$-filtered colimit of retractions, i.e., $p$ can be written as a $\lambda$-filtered colimit $p=\Colim{I\in\bI}p_I$ in the arrow category $\A^\to$ such that every $p_I$ is a retract in $\A$.
    \end{enumerate}
\end{proposition}

\begin{proposition}[{\cite[Proposition 4]{adamek2004purequotient}}]\label{prop:locret_implies_reg.epi}
    In a locally $\lambda$-presentable category, every $\lambda$-pure epimorphism is a regular epimorphism.
\end{proposition}

\begin{example}\quad
    \begin{enumerate}
        \item
        In $\Set$, for any $\lambda$, a $\lambda$-pure epimorphism is simply a surjection.
        \item
        Let $\Pos$ be the category of partially ordered sets and monotone maps.
        Considering the minimal limit ordinal $\omega=\Colim{n<\omega}n\in\Pos$ as a colimit of finite ordinals, we have a canonical morphism $p\colon\coprod_{n<\omega}n\to\omega$ from the coproduct.
        This $p$ is an $\aleph_0$-pure epimorphism in $\Pos$.
        Indeed, since $\omega=\Colim{n<\omega}n$ is a filtered colimit, every morphism $f\colon X\to\omega$ from finitely presentable object $X$ factors through some finite ordinal $n$, i.e., $f$ can be written as $f\colon X\arr[f_n]n\to\omega$ for some $f_n$.
        Then $X\arr[f_n]n\to\coprod_{n<\omega}n$ gives a lift of $f$ along $p$.\qedhere
    \end{enumerate}
\end{example}

\subsection{The ascending chain condition for categories}
We study \emph{strongly connected components} in categories and introduce a noetherian-like condition for a category.
Later, we will show that such a condition is sufficient and nearly necessary to eliminate the closure property under filtered colimits from Birkhoff's theorem.
More on strongly connected components in categories is discussed in \cite{kawase2025filtered}.

\begin{definition}\quad
    \begin{enumerate}
        \item
        Objects $X$ and $Y$ in a category are \emph{strongly connected} if there exist morphisms $X\to Y$ and $Y\to X$.
        \item
        Strong connectedness is an equivalence relation on the class of all objects.
        An equivalent class under strong connectedness is called a \emph{strongly connected component}.\qedhere
    \end{enumerate}
\end{definition}

\begin{remark}
    Every (not necessarily small) poset can be considered as a category by regarding each order as a morphism, which yields a functor $\iota\colon\POS\to\CAT$.
    Here, $\POS$ is the category of large posets, and $\CAT$ is the category of large categories.
    The functor $\iota$ has a left adjoint $\sigma\colon\CAT\to\POS$, which is called the \emph{posetification}.
    The underlying class of the large poset $\sigma(\C)$ is the class of all strongly connected components in $\C$.
    There is an order $[X]\le [Y]$ in $\sigma(\C)$ if and only if there is a morphism $X\to Y$ in $\C$.
\end{remark}

\begin{definition}\label{def:ACC_for_cat}
    We say a category $\C$ satisfies the \emph{ascending chain condition} (ACC) if the poset $\sigma(\C)$ satisfies the ordinary ascending chain condition, i.e., $\sigma(\C)$ has no infinite strictly ascending sequence.
    Equivalently, a category $\C$ satisfies ACC if for every $\bbomega$-chain $X_0\to X_1\to\cdots$ in $\C$, there exists $N\in\bN$ such that $(X_n)_{n\ge N}$ are strongly connected to each other.
\end{definition}

\subsection{The filtered colimit elimination}
We now show that the ascending chain condition as in the previous subsection is sufficient to eliminate closure under filtered colimits from Birkhoff's theorem (\cref{thm:filcolim_elim}).
Furthermore, we get an HSP-type formalization of Birkhoff's theorem for partial Horn theories (\cref{cor:HSP-type_formalization}).

\begin{definition}
    Let $\lambda$ be an infinite regular cardinal and let $U\colon\A\to\C$ be a functor.
    A morphism $p$ in $\A$ is called a \emph{$(U,\lambda)$-pure epimorphism} if $Up$ is a $\lambda$-pure epimorphism in $\C$.
\end{definition}

\begin{theorem}[Birkhoff-type theorem for partial Horn theories II]\label{thm:filcolim_elim}
    Let $\rho\colon\bS\to\bT$ be a theory morphism between $\lambda$-ary partial Horn theories.
    Assume that $\PMod\bS$ satisfies the ascending chain condition.
    Then, for every replete full subcategory $\E\subseteq\PMod\bT$, the following are equivalent:
    \begin{enumerate}
        \item\label{thm:filcolim_elim-1}
        $\E$ is definable by $\rho$-relative $\lambda$-ary judgments, i.e., there exists a set $\bT'$ of $\rho$-relative $\lambda$-ary judgments satisfying $\E=\PMod (\bT+\bT')$.
        \item\label{thm:filcolim_elim-2}
        $\E\subseteq\PMod\bT$ is closed under products, $\bT$-closed subobjects, and $(U^\rho,\lambda)$-pure quotients.
        \item\label{thm:filcolim_elim-3}
        $\E\subseteq\PMod\bT$ is closed under products, $\bT$-closed subobjects, $U^\rho$-retracts, and $\lambda$-filtered colimits.
    \end{enumerate}
\end{theorem}
\begin{proof}
    The same proof works as \cite[Theorem 5.4]{kawase2025filtered}.
\end{proof}

The above theorem brings alternative versions of \cref{cor:cor_of_birkhoff_1,cor:cor_of_birkhoff_2,thm:birkhoff_for_rat}:
\begin{corollary}\label{cor:birkhoff_1}
    Let $(S,\Sigma,\bT)$ be a $\lambda$-ary partial Horn theory such that the set $S$ of sorts is finite.
    Then, for every replete full subcategory $\E\subseteq\PMod\bT$, the following are equivalent:
    \begin{enumerate}
        \item
        $\E$ is definable by $\lambda$-ary Horn formulas, i.e.,
        there exists a set $E$ of $\lambda$-ary Horn formulas over $\Sigma$ satisfying $\E=\PMod (\bT+\bT')$, where $\bT'\coloneq\{ \top\seq{\tup{x}}\phi\}_{\tup{x}.\phi\in E}.$
        \item
        $\E\subseteq\PMod\bT$ is closed under products, $\Sigma$-closed subobjects, and surjections.
    \end{enumerate}
\end{corollary}

\begin{corollary}\label{cor:birkhoff_2}
    Let $(S,\Sigma,\bT)$ be a $\lambda$-ary partial Horn theory such that $\PMod\bT$ satisfies ACC.
    Then, for every replete full subcategory $\E\subseteq\PMod\bT$, the following are equivalent:
    \begin{enumerate}
        \item
        $\E$ is definable by $\lambda$-ary Horn sequents, i.e.,
        there exists a set $\bT'$ of $\lambda$-ary Horn sequents over $\Sigma$ satisfying $\E=\PMod (\bT+\bT')$.
        \item
        $\E\subseteq\PMod\bT$ is closed under products, $\Sigma$-closed subobjects, and $\lambda$-pure quotients.
    \end{enumerate}
\end{corollary}

\begin{corollary}[Birkhoff-type theorem for relative algebraic theories II]\label{cor:birkhoff_3}
    Let $(S,\Sigma,\bS)$ be a $\lambda$-ary partial Horn theory such that $\PMod\bS$ satisfies ACC.
    Let $(\Omega,E)$ be an $\bS$-relative $\lambda$-ary algebraic theory with the forgetful functor $U\colon\Alg(\Omega,E)\to\PMod\bS$.
    Then, for every replete full subcategory $\E\subseteq\Alg(\Omega,E)$, the following are equivalent:
    \begin{enumerate}
        \item
        $\E$ is definable by $\bS$-relative $\lambda$-ary judgments, i.e.,
        there exists a set $E'$ of $\bS$-relative $\lambda$-ary judgments satisfying $\E=\Alg(\Omega,E+E')$.
        \item
        $\E\subseteq\Alg(\Omega,E)$ is closed under products, $\Sigma$-closed subobjects, and $(U,\lambda)$-pure quotients.
    \end{enumerate}
\end{corollary}

\begin{remark}
    \cref{cor:birkhoff_3} generalizes the $\Set$-enriched case of \cite[Proposition 6.11]{rosickytendas2026towards}.
\end{remark}

By filtered colimit elimination, we can reformulate Birkhoff's theorem for relative algebras in the HSP-type form (\cref{cor:HSP-type_formalization}).

\begin{notation}
    Let $\rho\colon \bS\to\bT$ be a $\lambda$-ary theory morphism between $\lambda$-ary partial Horn theories.
    For each replete full subcategory $\E\subseteq\PMod\bT$, let $\cloprod(\E)\subseteq\PMod\bT$ denote the full subcategory consisting of all products of objects in $\E$.
    Similarly, let $\clocsub{\bT}(\E)$ and $\clolret{\rho}{\lambda}(\E)$ denote the full subcategory consisting of all $\bT$-closed subobjects of objects in $\E$ and of all $(U^\rho,\lambda)$-pure quotients of objects in $\E$, respectively.
\end{notation}

\begin{corollary}\label{cor:HSP-type_formalization}
    Let $\rho\colon\bS\to\bT$ be a theory morphism between $\lambda$-ary partial Horn theories and assume that the category $\PMod\bS$ satisfies ACC.
    Then, a replete full subcategory $\E\subseteq\PMod\bT$ is definable by $\rho$-relative $\lambda$-ary judgments if and only if $\clolret{\rho}{\lambda}\clocsub{\bT}\cloprod(\E)=\E$ holds.
\end{corollary}
\begin{proof}
    This follows from a $\lambda$-ary version of \cite[Lemma 5.13]{kawase2025filtered}.
\end{proof}

\begin{remark}\label{rem:converse_fil_colim_elim}
    If limited to the finitary case, our sufficient condition for filtered colimit elimination is also nearly necessary.
    For more details, see \cite[Theorem 5.16]{kawase2025filtered}.
    However, the author does not have any weak converse result in the infinitary case.
    At least, the same proof as the finitary case does not work.
\end{remark}

For convenience, we summarize the relationship among the notion of ``quotients'' that appeared in this paper.
\begin{proposition}\label{prop:relation_between_quotients}
    Let $\rho\colon\bS\to\bT$ be a theory morphism between $\lambda$-ary partial Horn theories.
    Assume that $U^\rho$ is of descent type, i.e., the Eilenberg--Moore comparison functor of the adjunction $F^\rho\dashv U^\rho$ is fully faithful.
    Then, the following holds:
    \begin{enumerate}
        \item\label{prop:relation_between_quotients-1}
        $U^\rho$-retractions are regular epimorphisms in $\PMod\bT$.
        \item\label{prop:relation_between_quotients-2}
        $(U^\rho,\lambda)$-pure epimorphisms are composites of two regular epimorphisms in $\PMod\bT$.
    \end{enumerate}
\end{proposition}
\begin{proof}
    \cref{prop:relation_between_quotients-1} is one of equivalent conditions to being of descent type \cite[Theorem 2.4]{kelly1993adjunctions}.
    To show \cref{prop:relation_between_quotients-2}, take a $(U^\rho,\lambda)$-pure epimorphism $p\colon X\to A$ in $\PMod\bT$.
    Let $\pi,\pi'\colon Y\to X$ be a kernel pair of $p$, and let $q\colon X\to B$ be a coequalizer of them and $r\colon B\to A$ be a unique morphism such that $rq=p$.
    Since pure epimorphisms are regular epimorphisms (\cref{prop:locret_implies_reg.epi}), $U^\rho p$ is a regular epimorphism.
    Since $U^\rho$ preserves kernel pairs, $U^\rho p$ is a coequalizer of $U^\rho\pi$ and $U^\rho\pi'$.
    Then, we can take a unique morphism $s\colon U^\rho A\to U^\rho B$ such that $sU^\rho p=U^\rho q$.
    Then, the universal property of the coequalizer implies that $(U^\rho r)s=\id$, hence $U^\rho r$ is a retraction.
    Since $U^\rho$ is of descent type, $r$ becomes a regular epimorphism, which finishes the proof.
\end{proof}

\begin{remark}
    Let $\rho\colon\bS\to\bT$ be a theory morphism between $\lambda$-ary partial Horn theories.
    Assume that $U^\rho$ is of descent type, which includes in the case that $\rho$ is a ``relative algebraic theory.''
    Then, the following implications hold for morphisms in $\PMod\bT$:
    \begin{equation*}
        \begin{tikzcd}[small]
            \text{$U^\rho$-retraction}\ar[rr,Rightarrow,shorten=10]\ar[rd,Rightarrow] &[-10pt] &[-10pt] \text{regular epi}\ar[rd,Rightarrow] &[-10pt] \\
            & \text{$(U^\rho,\lambda)$-pure epimorphism}\ar[rr,Rightarrow,shorten=10] & & \text{composite of two regular epis}\ar[d,Rightarrow] \\[10pt]
            & \text{epi} & \text{$\bT$-dense}\ar[l,Rightarrow,shorten=10] & \text{strong epi}\ar[l,Rightarrow,shorten=10]
        \end{tikzcd}
    \end{equation*}
    The author do not know any implication between regular epi and $(U^\rho,\lambda)$-pure epimorphism.
    At least, regular epi does not imply $(U^\rho,\lambda)$-pure epimorphism.
    Indeed, a morphism $q$ considered in \cref{rem:dense_underlying} is a regular epimorphism but not a $U^\rho$-pure epimorphism.
\end{remark}

\subsection{Some applications of filtered colimit elimination}
\subsubsection{Finite-sorted algebras}
Considering a partial Horn theory $(S,\varnothing,\varnothing)$ such that $S$ is finite, 
we get Birkhoff's theorem in finite sorts \cite{adamek2012birkhoffs} as a consequence of \cref{cor:birkhoff_3}:
\begin{corollary}
    Let $(\Omega,E)$ be an $S$-sorted $\lambda$-ary algebraic theory.
    Then, for every replete full subcategory $\E\subseteq\Alg(\Omega,E)$, the following are equivalent:
    \begin{itemize}
        \item
        $\E\subseteq\Alg(\Omega,E)$ is definable by equations.
        \item
        $\E\subseteq\Alg(\Omega,E)$ is closed under products, subobjects, and surjections.
    \end{itemize}
    This subsumes the original version of Birkhoff's theorem \cite{birkhoff1935structure}.
\end{corollary}

\begin{remark}
    The assumption that $\PMod\bS$ satisfies ACC cannot be removed from \cref{thm:filcolim_elim} even though $\bS$ is finite-sorted.
    To show this, let $\bS$ be the single-sorted finitary partial Horn theory for sets with countably many constants as follows:
    \begin{gather*}
        \Sigma\coloneq\{ c_n\colon\text{constant} \}_{n\in\bN},\quad \bS\coloneq\{ \top\seq{()}c_n\defined \}_{n\in\bN}.
    \end{gather*}
    Consider the full subcategory of $\PMod\bS$
    \[
        \E\coloneq\{1\}\cup\{ M\in\PMod\bS \mid \exists i,j\text{~s.t.}\intpn{c_i}{M}\neq\intpn{c_j}{M} \},
    \]
    where $1$ denotes the terminal.
    An $\aleph_0$-pure epimorphism in $\PMod\bS$ is simply a surjection which does not merge any constants.
    Thus, we can see that $\E\subseteq\PMod\bS$ is closed under products, $\Sigma$-closed subobjects, and $\aleph_0$-pure quotients.
    We now show that it is not closed under filtered colimits.
    For each $n\in\bN$, define $A_n\coloneq\bN\cup\{\infty\}\in\PMod\bS$ by $\intpn{c_i}{A_n}\coloneq\max(i-n,0)$.
    Let $f_n\colon A_n\to A_{n+1}$ be the morphism
    \begin{equation*}
        f_n(x)\coloneq
        \begin{cases*}
            \max(x-1,0) & if $x\neq\infty$, \\
            \infty & if $x=\infty$.
        \end{cases*}
    \end{equation*}
    Then, the filtered colimit $A_\infty$ of $A_0\arr[f_0]A_1\arr[f_1]\cdots$ consists of two different points, and all constants are merged there; hence $A_\infty\not\in\E$ even though $A_n\in\E$.
\end{remark}

\subsubsection{Ordered algebras}
Considering the partial Horn theory $\bS_\pos$ for posets as in \cref{eg:pht_for_posets}, we get Birkhoff's theorem for ordered algebras as a consequence of \cref{cor:birkhoff_3}:
\begin{corollary}
    Let $(\Omega,E)$ be an $\bS_\pos$-relative $\lambda$-ary algebraic theory with the forgetful functor $U\colon\Alg(\Omega,E)\to\Pos$.
    Then, for every replete full subcategory $\E\subseteq\Alg(\Omega,E)$, the following are equivalent:
    \begin{itemize}
        \item
        $\E\subseteq\Alg(\Omega,E)$ is definable by $\bS_\pos$-relative $\lambda$-ary judgments.
        \item
        $\E\subseteq\Alg(\Omega,E)$ is closed under products, $\Sigma_\pos$-closed subobjects, and $(U,\lambda)$-pure quotients.
    \end{itemize}
\end{corollary}

Moreover, there is another version of Birkhoff's theorem for ordered algebras.
Given an $\bS_\pos$-relative algebraic theory $(\Omega,E)$, we have the associated finitary partial Horn theory $\pht{\Omega}{E}$ as in \cref{def:pht_for_rat}.
Applying \cref{cor:birkhoff_1} to $\pht{\Omega}{E}$, we get the following:
\begin{corollary}
    Let $(\Omega,E)$ be an $\bS_\pos$-relative $\lambda$-ary algebraic theory.
    Then, for every replete full subcategory $\E\subseteq\Alg(\Omega,E)$, the following are equivalent:
    \begin{itemize}
        \item
        $\E\subseteq\Alg(\Omega,E)$ is definable by inequalities (Horn formulas over $\Sigma_\pos+\Omega$).
        \item
        $\E\subseteq\Alg(\Omega,E)$ is closed under products, $\Sigma_\pos$-closed subobjects, and surjections.
    \end{itemize}
\end{corollary}
Note that a $\pht{\Omega}{E}$-closed subobject is simply an $\Sigma_\pos$-closed subobject, i.e., an embedding.
This corollary is a direct generalization of a classical Birkhoff theorem for ordered algebras in \cite{bloom1976varieties}.
This is because a morphism $f\colon\bA\to\bB$ between ordered algebras is surjective if and only if $f$ can be written as a quotient $\bA\to\bA/\theta$ of a congruence $\theta$ determined by an \emph{admissible preorder} on $\bA$ in the sense of \cite{bloom1976varieties}.

\subsubsection{Metric algebras}
We now apply our result to $\gMet$-relative algebras, which are called \emph{metric algebras} in \cite{weaver1995metric}.
Let $\bS_\met$ be the $\aleph_1$-ary partial Horn theories for generalized metric spaces as in \cref{eg:pht_for_met}.
Since the category $\gMet$ satisfies ACC, we get the following as a consequence of \cref{cor:birkhoff_1}:
\begin{corollary}
    Let $(\Omega,E)$ be an $\bS_\met$-relative $\lambda$-ary algebraic theory.
    Then, for every replete full subcategory $\E\subseteq\Alg(\Omega,E)$, the following are equivalent:
    \begin{itemize}
        \item
        $\E\subseteq\Alg(\Omega,E)$ is definable by \emph{atomic inequality} \cite{weaver1995metric} (Horn formulas over $\Sigma_\met+\Omega$).
        \item
        $\E\subseteq\Alg(\Omega,E)$ is closed under products, $\Sigma_\met$-closed subobjects, and surjections.
    \end{itemize}
    This subsumes Birkhoff's theorem for metric algebras as in \cite[Theorem 3.8]{hino2016varieties}.
\end{corollary}

\appendix
\section{Infinitary partial Horn logic}\label[appendix]{appendix:phl}
We introduce $\lambda$-ary partial Horn logic ($\PHL_\lambda$) for an infinite regular cardinal, which is a direct generalization of finitary ($\aleph_0$-ary) partial Horn logic \cite{palmgren2007partial}.

\subsection{Partial models}
\begin{definition}\label{def:inf_signature}
    Let $\lambda$ be an infinite regular cardinal and let $S$ be a set.
    An \emph{$S$-sorted $\lambda$-ary signature} $\Sigma$ consists of:
    \begin{itemize}
        \item
        a set $\Sigma_\mathrm{f}$ of function symbols,
        \item
        a set $\Sigma_\mathrm{r}$ of relation symbols
    \end{itemize}
    such that
    \begin{itemize}
        \item
        for each $f\in\Sigma_\mathrm{f}$, an arity (and a sort) $f\colon \sqcap_{i<\alpha}s_i\to s$ ($\alpha<\lambda; s_i,s\in S$) is given;
        \item
        for each $R\in\Sigma_\mathrm{r}$, an arity $R\colon \sqcap_{i<\alpha}s_i$ ($\alpha<\lambda; s_i\in S$) is given.\qedhere
    \end{itemize}
\end{definition}

\begin{notation}
    If $\alpha=0$, we write $()$ instead of $\sqcap_{i<0}s_i$.
    If $\alpha=n\ge 1$, we also write
    \[
        s_0\sqcap s_1\sqcap\dots\sqcap s_{n-1}
    \]
    instead of $\sqcap_{i<n}s_i$.
\end{notation}

Given the set $S$ of sorts, we fix an $S$-sorted set $\mathrm{Var}=(\mathrm{Var}_s)_{s\in S}$ such that $\mathrm{Var}_s$ has cardinal $\lambda$ for each $s\in S$.
We assume $\mathrm{Var}_s\cap\mathrm{Var}_{s'}=\varnothing$ if $s\neq s'$.
An element $x\in\mathrm{Var}_s$ is called a \emph{variable of sort $s$}.
The notation $x\ofsort s$ means that $x$ is a variable of sort $s$.

\begin{definition}
    Let $\Sigma$ be an $S$-sorted $\lambda$-ary signature.
    \begin{enumerate}
        \item
        \emph{Raw terms} (over $\Sigma$) and their \emph{sorts} are defined by the following inductive rules:
        \begin{itemize}
            \item
            Given a variable $x$ of sort $s\in S$, $x$ is a raw term of sort $s$;
            \item
            Given a function symbol $f\in\Sigma$ with arity $\sqcap_{i<\alpha}s_i\to s$ and raw terms $\tau_i$ $(i<\alpha)$ where $\tau_i$ is of sort $s_i$, 
            then the expression $f(\tau_i)_{i<\alpha}$ is a raw term of sort $s$.
        \end{itemize}
        \item
        \emph{Raw ($\lambda$-ary) Horn formulas} (over $\Sigma$) are defined by the following inductive rules:
        \begin{itemize}
            \item
            Given a relation symbol $R\in\Sigma$ with arity $\sqcap_{i<\alpha}s_i$ and raw terms $\tau_i$ $(i<\alpha)$ where $\tau_i$ is of sort $s_i$, 
            then the expression $R(\tau_i)_{i<\alpha}$ is a raw $\lambda$-ary Horn formula;
            \item
            Given two raw terms $\tau$ and $\tau'$ of the same sort $s$,
            then $\tau=\tau'$ is a raw $\lambda$-ary Horn formula;
            \item
            The truth constant $\top$ is a raw $\lambda$-ary Horn formula;
            \item
            Given raw $\lambda$-ary Horn formulas $\phi_i$ $(i<\alpha)$ with $\alpha<\lambda$, the expression $\bigwedge_{i<\alpha}\phi_i$ is a raw $\lambda$-ary Horn formula.
            When $\alpha=0$, $\bigwedge_{i<\alpha}\phi_i$ expresses $\top$.
        \end{itemize}
        \item
        A \emph{($\lambda$-ary) context} is a tuple $\tup{x}=(x_i\ofsort s_i)_{i<\alpha}$ of distinct variables with $\alpha<\lambda$.
        \item
        A \emph{($\lambda$-ary) term} (over $\Sigma$) is a pair of a $\lambda$-ary context $\tup{x}$ and a raw term $\tau$ (over $\Sigma$), written as $\tup{x}.\tau$, where all variables appearing in $\tau$ occur in $\tup{x}$.
        The \emph{sort} of $\tup{x}.\tau$ is defined as that of $\tau$.
        \item
        A \emph{($\lambda$-ary) Horn formula} (over $\Sigma$) is a pair of a $\lambda$-ary context $\tup{x}$ and a raw $\lambda$-ary Horn formula $\phi$ (over $\Sigma$), written as $\tup{x}.\phi$, where all variables appearing in $\phi$ occur in $\tup{x}$.
        \item
        A \emph{($\lambda$-ary) Horn sequent} (over $\Sigma$) is a pair of two $\lambda$-ary Horn formulas $\tup{x}.\phi$ and $\tup{x}.\psi$ (over $\Sigma$) with the same context, written as
        \begin{equation*}
            \phi \seq{\tup{x}} \psi.
        \end{equation*}
        \item
        A \emph{$\lambda$-ary partial Horn theory} $\bT$ (over $\Sigma$) is a set of $\lambda$-ary Horn sequents (over $\Sigma$).\qedhere
    \end{enumerate}
\end{definition}

\begin{remark}
    Our terminologies raw terms and raw Horn formulas are simply called ``terms'' and ``Horn formulas'' in usual;
    our terminologies terms and Horn formulas are ordinarily called terms-in-context and Horn formulas-in-context \cite{johnstone2002sketches,palmgren2007partial}.
    Since we do not deal with terms and formulas with no context, the author believes our terminologies are convenient.
\end{remark}

Note that we do not consider the equal sign ``$=$'' to be a relation symbol.
We informally use the abbreviation $\phi\biseq{\tup{x}}\psi$ for ``$(\phi\seq{\tup{x}}\psi)$ and $(\psi\seq{\tup{x}}\phi)$,'' and $\tau\defined$ for $\tau=\tau$.

\begin{definition}
    Let $\Sigma$ be an $S$-sorted $\lambda$-ary signature.
    A \emph{partial $\Sigma$-structure} $M$ consists of:
    \begin{itemize}
        \item
        a set $M_s$ for each sort $s\in S$,
        \item
        a partial map
        \[
        \intpn{f}{M}~(\text{or}~\intpn{\tup{x}.f(\tup{x})}{M})\colon \prod_{i<\alpha}M_{s_i}\pto M_s
        \]
        for each function symbol $f\colon \sqcap_{i<\alpha}s_i\to s$ in $\Sigma$,
        \item
        a subset $\intpn{R}{M}$ (or $\intpn{\tup{x}.R(\tup{x})}{M}$) $\subseteq \prod_{i<\alpha}M_{s_i}$ for each relation symbol $R\colon \sqcap_{i<\alpha}s_i$ in $\Sigma$.\qedhere
    \end{itemize}
\end{definition}

We can extend the above definitions of $\intpn{\tup{x}.f(\tup{x})}{M}$ and $\intpn{\tup{x}.R(\tup{x})}{M}$ to arbitrary terms and Horn formulas:
\begin{definition}
    Let $\Sigma$ be an $S$-sorted $\lambda$-ary signature and let $M$ be a partial $\Sigma$-structure.
    Fix a context $\tup{x}=(x_i\ofsort s_i)_{i<\alpha}$.
    \begin{enumerate}
        \item
        For an arbitrary $\lambda$-ary term $\tup{x}.\tau$ of sort $s$ over $\Sigma$, we define a partial map
        \begin{equation*}
            \intpn{\tup{x}.\tau}{M}\colon \prod_{i<\alpha}M_{s_i}\pto M_s
        \end{equation*}
        as follows:
        \begin{itemize}
            \item
            For each $i<\alpha$, $\intpn{\tup{x}.x_i}{M}\colon\prod_{i<\alpha}M_{s_i}\to M_{s_i}$ is the $i$-th projection;
            \item
            For a function symbol $f\colon\sqcap_{j<\beta}s_j\to s$ in $\Sigma$ and terms $\tau_j$ of sort $s_j$, $\intpn{\tup{x}.f(\tau_j)_{j<\beta}}{M}(\tup{m})$ is defined if and only if all $\intpn{\tup{x}.\tau_j}{M}(\tup{m})$ are defined and $\intpn{f}{M}(\intpn{\tup{x}.\tau_j}{M}(\tup{m}))_{j<\beta}$ is also defined,
            and then $\intpn{\tup{x}.f(\tau_j)_{j<\beta}}{M}(\tup{m})\coloneq\intpn{f}{M}(\intpn{\tup{x}.\tau_j}{M}(\tup{m}))_{j<\beta}.$ 
        \end{itemize}
        \item
        For an arbitrary $\lambda$-ary Horn formula $\tup{x}.\phi$ over $\Sigma$, we define a subset
        \begin{equation*}
            \intpn{\tup{x}.\phi}{M}\subseteq\prod_{i<\alpha}M_{s_i}
        \end{equation*}
        as follows:
        \begin{itemize}
            \item
            For a relation symbol $R\colon\sqcap_{j<\beta}s_j$ in $\Sigma$ and terms $\tup{x}.\tau_j$ of sort $s_j$,
            $\tup{m}$ belongs to $\intpn{\tup{x}.R(\tau_j)_{j<\beta}}{M}$ if and only if all $\intpn{\tup{x}.\tau_j}{M}(\tup{m})$ are defined and $(\intpn{\tup{x}.\tau_j}{M}(\tup{m}))_{j<\beta}$ belongs to $\intpn{R}{M}$;
            \item
            For two terms $\tup{x}.\tau$ and $\tup{x}.\tau'$ of the same sort,
            $\tup{m}$ belongs to $\intpn{\tup{x}.\tau=\tau'}{M}$ if and only if both $\intpn{\tup{x}.\tau}{M}(\tup{m})$ and $\intpn{\tup{x}.\tau'}{M}(\tup{m})$ are defined and equal to each other;
            \item
            $\intpn{\tup{x}.\top}{M}\coloneq\prod_{i<\alpha} M_{s_i}$;
            \item
            For Horn formulas $(\tup{x}.\phi_j)_{j<\beta}$,
            $\intpn{\tup{x}.\bigwedge_{j<\beta}\phi_j}{M}\coloneq\bigcap_{j<\beta}\intpn{\tup{x}.\phi_j}{M}$.\qedhere
        \end{itemize}
    \end{enumerate}
\end{definition}

\begin{definition}
    We say that a Horn sequent $\phi\seq{\tup{x}}\psi$ over $\Sigma$ is \emph{valid} in a partial $\Sigma$-structure $M$ and write
    \[
        M \vDash (\phi\seq{\tup{x}}\psi)
    \]
    if $\intpn{\tup{x}.\phi}{M}\subseteq\intpn{\tup{x}.\psi}{M}$.
    A partial $\Sigma$-structure $M$ is called a \emph{partial $\bT$-model} for a $\lambda$-ary partial Horn theory $\bT$ over $\Sigma$ if all Horn sequents in $\bT$ are valid in $M$.
\end{definition}

\begin{definition}
    Let $\Sigma$ be an $S$-sorted $\lambda$-ary signature.
    A \emph{$\Sigma$-homomorphism} $h\colon M\to N$ between partial $\Sigma$-structures consists of:
    \begin{itemize}
        \item a total map $h_s\colon M_s\to N_s$ for each sort $s\in S$
    \end{itemize}
    such that for each function symbol $f\colon \sqcap_{i<\alpha}s_i\to s$ in $\Sigma$ and relation symbol $R\colon \sqcap_{j<\beta}s_j$ in $\Sigma$, there exist (necessarily unique) total maps (denoted by dashed arrows) making the following diagrams commute:
    \[
    \begin{tikzcd}[large]
        \prod_{i<\alpha}M_{s_i}\arrow[d,"\prod_{i<\alpha}h_{s_i}"'] &[-10pt] \mathrm{Dom}(\intpn{f}{M})\arrow[d,"\exists"',dashed]\arrow[l,hook']\arrow[r,"\intpn{f}{M}"] &[20pt] M_s\arrow[d,"h_s"] \\
        \prod_{i<\alpha}N_{s_i} & \mathrm{Dom}(\intpn{f}{N})\arrow[l,hook']\arrow[r,"\intpn{f}{N}"'] & N_s
    \end{tikzcd}
    \]
    \[
    \begin{tikzcd}[large]
        \prod_{j<\beta}M_{s_j}\arrow[d,"{\prod_{j<\beta}h_{s_j}}"'] &[-10pt] \intpn{R}{M}\arrow[d,"\exists"',dashed]\arrow[l,hook'] \\
        \prod_{j<\beta}N_{s_j} & {\intpn{R}{N}} \arrow[l,hook']
    \end{tikzcd}
    \]
\end{definition}

\begin{notation}
    Let $\bT$ be a $\lambda$-ary partial Horn theory over an $S$-sorted $\lambda$-ary signature $\Sigma$.
    We will denote by $\PStr\Sigma$ the category of partial $\Sigma$-structures and $\Sigma$-homomorphisms and by $\PMod\bT$ the full subcategory of $\PStr\Sigma$ consisting of all partial $\bT$-models.
\end{notation}

\begin{example}
    To explain what a partial structure is, we begin with an artificial example.
    Let $S\coloneq\bN$, the set of all natural numbers.
    The $S$-sorted finitary signature $\Sigma$ consists of:
    \begin{equation*}
        P\colon (),\quad
        c\colon ()\to 0.
    \end{equation*}
    Then, a partial $\Sigma$-structure $M$ consists of an $\bN$-sorted set $(M_n)_{n\in\bN}$ and a subset $\intpn{P}{M}\subseteq 1$, where $1$ denotes the singleton, and may have a constant $\intpn{c}{M}\in M_0$.
    That is, the nullary relation symbol $P$ is interpreted as a ``proposition,'' and the nullary function symbol $c$ is interpreted as a ``partial constant.''

    We next explain what a $\Sigma$-homomorphism is.
    A $\Sigma$-homomorphism $h\colon M\to N$ exists if and only if the following conditions hold:
    \begin{itemize}
        \item
        $\intpn{P}{M}\subseteq\intpn{P}{N}$;
        \item
        If $\intpn{c}{M}\in M_0$ exists, then $\intpn{c}{N}\in N_0$ also exists.
    \end{itemize}
    Then, $h$ is a family of maps $(M_n\arr[h_n]N_n)_{n\in\bN}$ such that $h_0(\intpn{c}{M})=\intpn{c}{N}$ holds whenever $\intpn{c}{M}$ exists.
\end{example}

\begin{example}[Posets]\label{eg:pht_for_posets}
    We present a finitary partial Horn theory $\bS_\pos$ for posets.
    Let $S\coloneq\{\sort\}$, $\Sigma_\pos\coloneq$\mbox{$\{ \le\colon \sort\sqcap\sort \}$}.
    The finitary partial Horn theory $\bS_\pos$ over $\Sigma_\pos$ consists of:
    \begin{gather*}
        \top\seq{x}x\le x,\quad
        x\le y\wedge y\le x\seq{x,y}x=y,\quad
        x\le y\wedge y\le z\seq{x,y,z}x\le z.
    \end{gather*}
Then, we have $\PMod\bS_\pos\cong\Pos$, where $\Pos$ denotes the category of partially ordered sets and monotone maps.
\end{example}

\begin{example}[Small categories]\label{eg:pht_for_smallcat}
    We present a finitary partial Horn theory $\bS_\cat$ for small categories.
    Let $S\coloneq\{ \ob, \mor\}$.
    The $S$-sorted finitary signature $\Sigma_\cat$ consists of:
    \begin{gather*}
        \id\colon\ob\to\mor,\quad
        \mathrm{d}\colon\mor\to\ob,\quad
        \mathrm{c}\colon\mor\to\ob,\quad
        \circ\colon\mor\times\mor\to\mor.
    \end{gather*}
    The finitary partial Horn theory $\bS_\cat$ over $\Sigma_\cat$ consists of:
    \begin{gather*}
        \top\seq{x\ofsort \ob}\id(x)\defined,\quad
        \top\seq{f\ofsort\mor}\mathrm{d}(f)\defined\wedge\mathrm{c}(f)\defined,\quad
        \mathrm{d}(g)=\mathrm{c}(f)\biseq{g{,}f\ofsort\mor}(g\circ f)\defined;\\
        \top\seq{x\ofsort\ob}\mathrm{d}(\id(x))=x\wedge\mathrm{c}(\id(x))=x;\\
        \mathrm{d}(g)=\mathrm{c}(f)\seq{g{,}f\ofsort\mor}\mathrm{d}(g\circ f)=\mathrm{d}(f)\wedge\mathrm{c}(g\circ f)=\mathrm{c}(g);\\
        \top\seq{f\ofsort\mor}f\circ\id(\mathrm{d}(f))=f\wedge \id(\mathrm{c}(f))\circ f=f;\\
        \mathrm{d}(h)=\mathrm{c}(g)\wedge\mathrm{d}(g)=\mathrm{c}(f)\seq{h{,}g{,}f\ofsort\mor}(h\circ g)\circ f=h\circ (g\circ f).
    \end{gather*}
    Then, we have $\PMod\bS_\cat\simeq\Cat$.
\end{example}

\begin{example}[Generalized metric spaces]\label{eg:pht_for_met}
    We present an $\aleph_1$-ary partial Horn theory $\bS_\met$.
    Let $S\coloneq\{\sort\}$, $\Sigma_\met\coloneq\{ R^a\colon \sort\sqcap\sort \}_{a\in [0,\infty)}$.
    Here $[0,\infty)$ denotes the set of all non-negative real numbers.
    The $\aleph_1$-ary partial Horn theory $\bS_\met$ consists of:
    \begin{gather*}
        R^a(x,y)\seq{x,y}R^b(x,y)\qquad(\text{for any }a\le b);\\
        \bigwedge_{n<\omega}R^{a_n}(x,y)\seq{x,y}R^{\inf_n a_n}(x,y)\qquad(\text{for any sequent } (a_n)_{n<\omega});\\
        \top\seq{x}R^a(x,x)\qquad(\text{for any }a\in\bR_\ge);\\
        R^a(x,y)\seq{x,y}R^a(y,x);\\
        R^a(x,y)\wedge R^b(y,z)\seq{x,y,z}R^{a+b}(x,z);\\
        R^0(x,y)\seq{x,y} x=y.
    \end{gather*}
    Regarding ``$R^a(x,y)$'' as ``$d(x,y)\le a$'', we observe that a partial $\bS_\met$-model is precisely a \emph{generalized metric space}, i.e., a set $X$ with a map $d\colon X\to [0,\infty]$ satisfying appropriate axioms (see \cite[Examples 4.5(3)]{lieberman2017metric}).
    Furthermore, we have $\PMod\bS_\met\cong \gMet$, where $\gMet$ denotes the category of generalized metric spaces and contractions.
\end{example}

\begin{example}[$[0,\infty)$-fuzzy sets]\label{eg:pht_for_fuzzy}
    A set with a map from itself to $[0,\infty)$ is called a \emph{$[0,\infty)$-fuzzy set}.
    For two fuzzy sets $v\colon X\to [0,\infty)$ and $w\colon Y\to [0,\infty)$, a morphism $(X,v)\to (Y,w)$ is a map $f\colon X\to Y$ such that $v(x)\ge wf(x)$ for all $x\in X$, which yields a category $\Fuz$.
    
    We now present an $\aleph_1$-ary partial Horn theory for $[0,\infty)$-fuzzy sets.
    The $[0,\infty)$-sorted $\aleph_1$-ary signature $\Sigma_\fuz$ consists of:
    \begin{itemize}
        \item
        for each pair of $r,r'\in [0,\infty)$ such that $r\le r'$, a function symbol
        \begin{equation*}
            \iota_r^{r'}\colon r\to r';
        \end{equation*}
        \item
        for each decreasing sequence $\tup{r}=(r_n)_{n<\omega}$ in $[0,\infty)$, a function symbol
        \begin{equation*}
            \delta_{\tup{r}}\colon \sqcap_{n<\omega}r_n \to \lim_{n\to\infty}r_n.
        \end{equation*}
    \end{itemize}
    The $\aleph_1$-ary partial Horn theory $\bS_\fuz$ over $\Sigma_\fuz$ consists of the following Horn sequents:
    \begin{align}
        \label{eq:fuz_incl_total}
        \top&\seq{x\ofsort r_0}\iota_{r_0}^{r_1}(x)\defined;\\
        \label{eq:fuz_incl_monic}
        \iota_{r_0}^{r_1}(x)=\iota_{r_0}^{r_1}(y)&\seq{x{,}y\ofsort r_0}x=y;\\
        \label{eq:fuz_incl_identity}
        \top&\seq{x\ofsort r}\iota_r^r(x)=x;\\
        \label{eq:fuz_incl_composition}
        \top&\seq{x\ofsort r_0}\iota_{r_1}^{r_2}(\iota_{r_0}^{r_1}(x))=\iota_{r_0}^{r_2}(x);\\
        \label{eq:fuz_domain_iota}
        \delta_{\tup{r}}(\tup{x})\defined&\biseq{\tup{x}}\bigwedge_{n<\omega} \iota_{r_{n+1}}^{r_n}(x_{n+1})=x_n;\\
        \label{eq:fuz_property_iota}
        \delta_{\tup{r}}(\tup{x})\defined&\seq{\tup{x}}\bigwedge_{n<\omega} \iota_{\lim \tup{r}}^{r_n}(\delta_{\tup{r}}(\tup{x}))=x_n.
    \end{align}
    In the following, we will describe how to consider a partial $\bS_\fuz$-model as a $[0,\infty)$-fuzzy set, i.e., a map from a set to $[0,\infty)$.
    Suppose that a partial $\bS_\fuz$-model $M$ is given.
    We now regard $M_r$ as ``the set of all elements whose value is at most $r$.''
    By \cref{eq:fuz_incl_total}, \cref{eq:fuz_incl_monic}, \cref{eq:fuz_incl_identity}, and \cref{eq:fuz_incl_composition}, we have injections $\intpn{\iota_r^{r'}}{M}\colon M_r\hookrightarrow M_{r'}$ $(r<r')$ and observe that those make $(M_r)_{r\ge 0}$ to be a diagram of shape $[0,\infty)$.
    Let $X$ be a colimit of the diagram $(M_r)_{r\ge 0}$ in $\Set$.
    For simplicity, we consider each $M_r$ as a subset of $X$.
    Now, for each $x\in X$, we can define its value $v(x)$ as the smallest $r\ge 0$ such that $x\in M_r$.
    This definition makes sense by \cref{eq:fuz_domain_iota} and \cref{eq:fuz_property_iota}.
    Finally, we get a fuzzy set $v\colon X\to [0,\infty)$.
    Moreover, this construction yields an equivalence $\PMod\bS_\fuz\simeq\Fuz$.
\end{example}

\begin{example}[Pointed metric spaces]\label{eg:pht_for_pmet}
    We present an $\aleph_1$-ary partial Horn theory for pointed metric spaces.
    The $[0,\infty)$-sorted $\aleph_1$-ary signature $\Sigma_\pmet$ contains all symbols of $\Sigma_\fuz$ and the following additional symbols:
    \begin{gather*}
        \dot{0}\colon ()\to 0,\quad
        R_r^a\colon r\sqcap r,\quad
        j_r^a\colon r\to a \qquad (r,a\in [0,\infty)).
    \end{gather*}
    The $\aleph_1$-ary partial Horn theory $\bS_\pmet$ consists of all sequents in $\bS_\fuz$ and the following:
    \begin{align}
        \label{eq:pmet_structure-1}
        R_r^a(x,y) &\biseq{x{,}y\ofsort r} R_{r'}^a(\iota_r^{r'}(x),\iota_r^{r'}(y));\\
        R_r^a(x,y) &\seq{x{,}y\ofsort r} R_r^b(x,y) \qquad (a\le b);\\
        \bigwedge_{n<\omega}R_r^{a_n}(x,y) &\seq{x{,}y\ofsort r} R_r^{\inf\tup{a}}(x,y);\label{eq:pmet_structure-2}\\
        \label{eq:pmet_compatibility-1}
        \top &\seq{x\ofsort r} R_r^r(x,\iota_0^r(\dot{0}));\\
        j_r^a(x)\defined &\biseq{x\ofsort r} R_r^a(x,\iota_0^r(\dot{0}));\\
        j_r^a(x)\defined &\seq{x\ofsort r} \iota_a^r(j_r^a(x))=x;\label{eq:pmet_compatibility-2}
    \end{align}
    \begin{align}
        \label{eq:pmet_metric_axioms-1}
        R_r^a(x,y) &\biseq{x{,}y\ofsort r} R_r^a(y,x);\\
        R_r^a(x,y)\wedge R_r^b(y,z) &\longseq{x{,}y{,}z\ofsort r} R_r^{a+b}(x+z);\\
        \top &\seq{x\ofsort r} R_r^a(x,x);\\
        R_r^0(x,y) &\seq{x{,}y\ofsort r} x=y.\label{eq:pmet_metric_axioms-2}
    \end{align}
    As described in \cref{eg:pht_for_fuzzy}, each partial $\bS_\pmet$-model has a structure of fuzzy set $(X,v)$.
    The axioms \cref{eq:pmet_structure-1} to \cref{eq:pmet_structure-2} induce a well-defined map $d\colon X\times X\to [0,\infty]$, and the axioms \cref{eq:pmet_metric_axioms-1} to \cref{eq:pmet_metric_axioms-2} ensure that this is a generalized metric on $X$.
    In addition, by \cref{eq:pmet_compatibility-1} to \cref{eq:pmet_compatibility-2}, $d$ is compatible with the fuzzy set structure, i.e., $d(0,x)=v(x)$ holds, where $0$ denotes the constant defined by $\dot{0}$.
    Furthermore, we can see that $d$ is an ordinary metric.
    Indeed,
    \[
        d(x,y)\le d(x,0)+d(0,y)=v(x)+v(y)
    \]
    proves that $d$ always takes a bounded value.
    Thus, a partial $\bS_\pmet$-model is precisely an ordinary metric space with a constant.
    Moreover, there is an equivalence $\PMod\bS_\pmet\simeq\pMet$, where $\pMet$ is the category of pointed metric spaces and contractions.
\end{example}

\subsection{Inference rules for infinitary PHL}
In this subsection, we develop the syntax for infinitary partial Horn logic.
The inference rule given here is a direct generalization of finitary one as in \cite{palmgren2007partial}.

\begin{definition}
    Let $\Sigma$ be an $S$-sorted $\lambda$-ary signature.
    \begin{enumerate}
        \item
        A \emph{rule} over $\Sigma$ consists of:
        \begin{itemize}
            \item
            a family $(\phi_i\seq{\tup{x}_i}\psi_i)_{i<\alpha}$ of Horn sequents over $\Sigma$ with $\alpha<\lambda$,
            \item
            a Horn sequent $\phi\seq{\tup{x}}\psi$ over $\Sigma$.
        \end{itemize}
        Such a rule is expressed by
        \[
            \infer{
            \phi\seq{\tup{x}}\psi
            }{
            (\phi_i\seq{\tup{x}_i}\psi_i)_{i<\alpha}
            }
        \]
        \item
        The \emph{inference rules} of $\lambda$-ary partial Horn logic ($\PHL_\lambda$) over $\Sigma$ are the following rules.
        \begin{enumerate}
            \myitem{(Id)}\label{phl:identity}
            For each Horn formula $\tup{x}.\phi$,
            \[
                \infer[\ref*{phl:identity}]{
                \phi\seq{\tup{x}}\phi
                }{}
            \]
            \myitem{(Cut)}\label{phl:cut}
            For Horn formulas $\tup{x}.\phi,\tup{x}.\psi,\tup{x}.\chi$ with the same context $\tup{x}$,
            \[
                \infer[\ref*{phl:cut}]{
                    \phi\seq{\tup{x}}\chi
                }{
                    \phi\seq{\tup{x}}\psi &
                    \psi\seq{\tup{x}}\chi
                }
            \]
            \myitem{(Subst)}\label{phl:substitution}
            Let $\tup{x}=(x_i\ofsort s_i)_{i<\alpha}$ and $\tup{y}$ be contexts.
            Let $\tup{x}.\phi,\tup{x}.\psi$ be Horn formulas, and for each $i<\alpha$, a term $\tup{y}.\tau_i$ of sort $s_i$ is given.
            Then, the rule
            \[
                \infer[\ref*{phl:substitution}]{
                    \phi(\tup{\tau}/\tup{x})\wedge\bigwedge_{i<\alpha}\tau_i\defined \seq{\tup{y}} \psi(\tup{\tau}/\tup{x})
                }{
                    \phi\seq{\tup{x}}\psi
                }
            \]
            is applicable, where $\phi(\tup{\tau}/\tup{x})$ and $\psi(\tup{\tau}/\tup{x})$ are the Horn formulas obtained by replacing all $x_i$ to $\tau_i$ simultaneously.
            \myitem{(Refl)}\label{phl:reflexivity}
            For a context $\tup{x}=(x_i)_{i<\alpha}$ and each $i<\alpha$,
            \[
                \infer[\ref*{phl:reflexivity}]{
                \top\seq{\tup{x}}x_i\defined
                }{&}
            \]
            \myitem{(Eq)}\label{phl:equality}
            Let $\tup{x}=(x_i\ofsort s_i)_{i<\alpha}$ and $\tup{y}=(y_i\ofsort s_i)_{i<\alpha}$ be contexts of the same sort.
            Let $\tup{x}.\phi$ be a Horn formula and let $\tup{z}$ be a context containing all $x_i$ and $y_i$.
            Then, the following rule is applicable:
            \[
                \infer[\ref*{phl:equality}]{
                \phi\wedge\bigwedge_{i<\alpha}x_i=y_i \seq{\tup{z}} \phi(\tup{y}/\tup{x})
                }{&}
            \]
            \myitem{(SRel)}\label{phl:strictness_relation}
            For a relation symbol $R\colon\sqcap_{i<\alpha}s_i$, terms $\tup{x}.\tau_i$ of sort $s_i$, and $j<\alpha$,
            \[
                \infer[\ref*{phl:strictness_relation}]{
                R(\tau_i)_{i<\alpha} \seq{\tup{x}} \tau_j\defined
                }{&}
            \]
            \myitem{(SEq)}\label{phl:strictness_equality}
            For terms $\tup{x}.\tau$ and $\tup{x}.\sigma$ of the same sort,
            \[
                \infer[\ref*{phl:strictness_equality}]{
                \tau=\sigma \seq{\tup{x}} \tau\defined
                }{&}
                \qquad\qquad
                \infer[\ref*{phl:strictness_equality}]{
                \tau=\sigma \seq{\tup{x}} \sigma\defined
                }{&}
            \]
            \myitem{(SFun)}\label{phl:strictness_function}
            For a function symbol $f\colon\sqcap_{i<\alpha}s_i\to s$, terms $\tup{x}.\tau_i$ of sort $s_i$, and $j<\alpha$,
            \[
                \infer[\ref*{phl:strictness_function}]{
                f(\tau_i)_{i<\alpha} \seq{\tup{x}} \tau_j\defined
                }{&}
            \]
            \myitem{(EConj)}\label{phl:elim_conjunction}
            For Horn formulas $(\tup{x}.\phi_i)_{i<\alpha}$ and $j<\alpha$,
            \[
                \infer[\ref*{phl:elim_conjunction}]{
                \bigwedge_{i<\alpha}\phi_i \seq{\tup{x}} \phi_j
                }{&}
            \]
            \myitem{(IConj)}\label{phl:intro_conjunction}
            For Horn formulas $\tup{x}.\phi$ and $(\tup{x}.\psi_i)_{i<\alpha}$ with $\alpha<\lambda$,
            \[
                \infer[\ref*{phl:intro_conjunction}]{
                \phi\seq{\tup{x}}\bigwedge_{i<\alpha}\psi_i
                }{
                (\phi\seq{\tup{x}}\psi_i)_{i<\alpha}
                }
            \]
        \end{enumerate}
    \end{enumerate}
\end{definition}

\begin{definition}
    Let $\bT$ be a $\lambda$-ary partial Horn theory over an $S$-sorted $\lambda$-ary signature $\Sigma$.
    \begin{enumerate}
        \item
        A \emph{derivation} from $\bT$ over $\Sigma$ is a well-founded rooted tree of Horn sequents over $\Sigma$ such that
        for every node, a pair of its children and itself exhibits one of the following rules:
        \begin{itemize}
            \item
            for some $\phi\seq{\tup{x}}\psi\in\bT$,
            \[
                \infer{
                    \phi\seq{\tup{x}}\psi
                }{&}
            \]
            \item
            the inference rules of $\PHL_\lambda$ over $\Sigma$.
        \end{itemize}
        \item
        A rule over $\Sigma$
        \[
            \infer{
                \phi\seq{\tup{x}}\psi
            }{
                (\phi_i\seq{\tup{x}_i}\psi_i)_{i<\alpha}
            }
        \]
        is \emph{derivable} from $\bT$ if there exists a derivation from $\bT\cup\{ \phi_i\seq{\tup{x}_i}\psi_i \}_{i<\alpha}$ over $\Sigma$ whose root is $\phi\seq{\tup{x}}\psi$.
        A rule is simply called \emph{derivable} when it is derivable from the empty theory.
        \item
        A Horn sequent $\phi\seq{\tup{x}}\psi$ is called a \emph{$\PHL_\lambda$-theorem of $\bT$} and written as
        \[
            \bT \vdash (\phi\seq{\tup{x}}\psi)
        \]
        if the rule
        \[
            \infer{
                \phi\seq{\tup{x}}\psi
            }{&}
        \]
        is derivable from $\bT$.
        A Horn sequent is simply called a \emph{$\PHL_\lambda$-theorem} when it is a $\PHL_\lambda$-theorem of the empty theory.\qedhere
    \end{enumerate}
\end{definition}

\begin{remark}
    We can easily check that $\bT\vdash(\phi\seq{\tup{x}}\psi)$ always implies $\bT\vDash(\phi\seq{\tup{x}}\psi)$, i.e., for every $M\in\PMod\bT$, $M\vDash (\phi\seq{\tup{x}}\psi)$.
    This is the soundness theorem for $\PHL_\lambda$.
\end{remark}

The proofs of the following lemmas are omitted.

\begin{lemma}[The context weakening rule]
    Let $\tup{x}\subseteq\tup{y}$ be contexts and let $\tup{x}.\phi,\tup{x}.\psi$ be Horn formulas.
    Then, the following rule is derivable:
    \[
        \infer{
            \phi\seq{\tup{y}}\psi
        }{
            \phi\seq{\tup{x}}\psi
        }
    \]
\end{lemma}

\begin{lemma}[The cut rule]
    For Horn formulas $\tup{x}.\phi,\tup{x}.\psi,\tup{x}.\chi$, the following rule is derivable:
    \[
        \infer{
            \chi\wedge\bigwedge_{i<\alpha}\phi_i\seq{\tup{x}}\theta
        }{
            (\phi_i\seq{\tup{x}}\psi_i)_{i<\alpha}
            &
            \chi\wedge\bigwedge_{i<\alpha}\psi_i\seq{\tup{x}}\theta
        }
    \]
\end{lemma}

\begin{lemma}\label{lem:equality}
    Let $\Sigma$ be an $S$-sorted $\lambda$-ary signature.
    Then, the following are $\PHL_\lambda$-theorems:
    \begin{enumerate}
        \item
        $\tau=\sigma \seq{\tup{x}} \sigma=\tau$
        \item
        $\tau=\sigma \wedge \sigma=\rho \seq{\tup{x}} \tau=\rho$
    \end{enumerate}
\end{lemma}

\begin{lemma}\label{lem:equality_term_formula}
    Let $\Sigma$ be an $S$-sorted $\lambda$-ary signature and let $\tup{y}=(y_j\ofsort s_j)_{j<\beta}$ be a context.
    Suppose that terms $\tup{x}.\sigma_j$ and $\tup{x}.\rho_j$ of sort $s_j$ are given for each $j<\beta$.
    \begin{enumerate}
        \item
        For every term $\tup{y}.\tau$ over $\Sigma$, the following is a $\PHL_\lambda$-theorem:
        \[
            \tau(\tup{\sigma}/\tup{y})\defined\wedge\bigwedge_{j<\beta}\sigma_j=\rho_j \seq{\tup{x}} \tau(\tup{\sigma}/\tup{y})=\tau(\tup{\rho}/\tup{y})
        \]
        \item
        For every formula $\tup{y}.\phi$ over $\Sigma$, the following is a $\PHL_\lambda$-theorem:
        \[
            \phi(\tup{\sigma}/\tup{y})\wedge\bigwedge_{j<\beta}\sigma_j=\rho_j \seq{\tup{x}} \phi(\tup{\rho}/\tup{y})
        \]
    \end{enumerate}
\end{lemma}

\subsection{Completeness theorem for infinitary PHL}
We now prove the completeness theorem for infinitary partial Horn logic by the syntax given in the previous subsection.
The finitary version already appears in \cite{palmgren2007partial}.

\begin{definition}\label{def:T-term}
Let $\bT$ be a $\lambda$-ary partial Horn theory over an $S$-sorted $\lambda$-ary signature $\Sigma$.
Let $\tup{x}.\phi$ be a Horn formula over $\Sigma$.
\begin{enumerate}
    \item
    A $\lambda$-ary term $\tup{x}.\tau$ over $\Sigma$ is called a \emph{$\bT$-term generated by $\tup{x}.\phi$} if $\phi\seq{\tup{x}}\tau\defined$ is a $\PHL_\lambda$-theorem of $\bT$.
    We write $\bT\-\Term(\tup{x}.\phi)$ for the $S$-sorted set of all $\bT$-terms generated by $\tup{x}.\phi$.
    \item
    Define a binary relation $\sim_\bT$ on $\bT\-\Term(\tup{x}.\phi)$ as follows: 
    $\tup{x}.\tau\sim_\bT \tup{x}.\tau'$ if and only if $\phi\seq{\tup{x}}\tau =\tau'$ is a $\PHL_\lambda$-theorem of $\bT$.\qedhere
\end{enumerate}
\end{definition}

\begin{lemma}
    In \cref{def:T-term}, the binary relation $\sim_\bT$ is a congruence on $\bT\-\Term(\tup{x}.\phi)$.
\end{lemma}
\begin{proof}
    The reflexivity follows from the definition of $\bT$-terms.
    The symmetricity and the transitivity follow from \cref{lem:equality}.
\end{proof}

\begin{notation}
    We will denote by $[\tup{x}.\tau]_\bT$ the equivalence class of $\tup{x}.\tau$ under $\sim_\bT$.
\end{notation}

\begin{lemma}\label{lem:structure_repn_model}
    In \cref{def:T-term}, the quotient $S$-sorted set $M\coloneq\bT\-\Term(\tup{x}.\phi)/{\sim_\bT}$ becomes a partial $\Sigma$-structure as follows:
    \begin{itemize}
        \item
        For each function symbol $f\in\Sigma$, define the partial function
        \[
        \intpn{f}{M}\colon ([\tup{x}.\tau_j]_\bT)_{j<\beta}\mapsto [\tup{x}.f(\tau_j)_{j<\beta}]_\bT
        \]
        whose value is defined if and only if $\phi\seq{\tup{x}}f(\tau_j)_{j<\beta}\defined$ is a $\PHL_\lambda$-theorem of $\bT$.
        \item
        For each relation symbol $R\in\Sigma$, 
        \begin{equation*}
            \intpn{R}{M}\coloneq\{([\tup{x}.\tau_j]_\bT)_{j<\beta} \mid \phi\seq{\tup{x}}R(\tau_j)_{j<\beta}\text{ is a }\PHL_\lambda\text{-theorem of }\bT\}.
        \end{equation*}
    \end{itemize}
\end{lemma}
\begin{proof}
    By \cref{lem:equality_term_formula}, the following rules are derivable:
    \begin{gather*}
        \infer{
            \phi\seq{\tup{x}}f(\tau'_j)_{j<\beta}\defined
        }{
            \phi\seq{\tup{x}}f(\tau_j)_{j<\beta}\defined
            &
            (\phi\seq{\tup{x}}\tau_j=\tau'_j)_{j<\beta}
        }
        \quad
        \infer{
            \phi\seq{\tup{x}}f(\tau_j)_{j<\beta}=f(\tau'_j)_{j<\beta}
        }{
            \phi\seq{\tup{x}}f(\tau_j)_{j<\beta}\defined
            &
            (\phi\seq{\tup{x}}\tau_j=\tau'_j)_{j<\beta}
        }
        \\
        \infer{
            \phi\seq{\tup{x}}R(\tau'_j)_{j<\beta}
        }{
            \phi\seq{\tup{x}}R(\tau_j)_{j<\beta}
            &
            (\phi\seq{\tup{x}}\tau_j=\tau'_j)_{j<\beta}
        }
    \end{gather*}
    These complete the proof.
\end{proof}

\begin{lemma}\label{lem:basic_property_repn_model}
    Let $\bT$ be a $\lambda$-ary partial Horn theory over an $S$-sorted $\lambda$-ary signature $\Sigma$.
    Let $\tup{x}.\phi$ be a Horn formula with $\tup{x}=(x_i\ofsort s_i)_{i<\alpha}$ and let $M\coloneq\bT\-\Term(\tup{x}.\phi)/{\sim_\bT}$ be the partial $\Sigma$-structure in \cref{lem:structure_repn_model}.
    \begin{enumerate}
        \item\label{lem:basic_property_repn_model-1}
        For any Horn formula $\tup{y}.\psi$,
        \begin{equation*}
            \intpn{\tup{y}.\psi}{M} = \{([\tup{x}.\tau_j]_\bT)_{j} \mid \phi\seq{\tup{x}}\psi(\tup{\tau}/\tup{y})\text{ is a }\PHL_\lambda\text{-theorem of }\bT\}.
        \end{equation*}
        \item
        $M$ is a $\bT$-model.
        \item
        $([\tup{x}.x_i]_\bT)_{i<\alpha} \in \intpn{\tup{x}.\phi}{M}$ holds.
    \end{enumerate}
\end{lemma}
\begin{proof}
    Straightforward.
\end{proof}

\begin{definition}\label{def:representing_model}
    Let $\bT$ be a $\lambda$-ary partial Horn theory over an $S$-sorted $\lambda$-ary signature $\Sigma$.
    For each $\lambda$-ary Horn formula $\tup{x}.\phi$ over $\Sigma$, define
    \begin{equation*}
        \repn{\tup{x}.\phi}_\bT\coloneq\bT\-\Term(\tup{x}.\phi)/{\sim_{\bT}}
        \quad\in\PMod\bT.
    \end{equation*}
    This $\repn{\tup{x}.\phi}_\bT$ is called the \emph{representing $\bT$-model for $\tup{x}.\phi$}.
\end{definition}

\begin{theorem}[Completeness theorem for infinitary PHL]\label{thm:completeness_thm_for_PHL}
    Let $\bT$ be a $\lambda$-ary partial Horn theory over an $S$-sorted $\lambda$-ary signature $\Sigma$.
    For any $\lambda$-ary Horn sequent $\phi\seq{\tup{x}}\psi$ over $\Sigma$, the following are equivalent:
    \begin{enumerate}
        \item\label{thm:completeness_thm_for_PHL-1}
        $\bT\vdash (\phi\seq{\tup{x}}\psi)$.
        \item\label{thm:completeness_thm_for_PHL-2}
        $\bT\vDash (\phi\seq{\tup{x}}\psi)$, i.e., for every $M\in\PMod\bT$, $M\vDash (\phi\seq{\tup{x}}\psi)$.
    \end{enumerate}
\end{theorem}
\begin{proof}
    {[\cref{thm:completeness_thm_for_PHL-2}$\implies$\cref{thm:completeness_thm_for_PHL-1}]}
    By assumption, we particularly get $\repn{\tup{x}.\phi}_\bT\vDash (\phi\seq{\tup{x}}\psi)$.
    Equivalently, for $\bT$-terms $(\tup{x}.\tau_i)_i$ generated by $\tup{x}.\phi$, $\bT\vdash(\phi\seq{\tup{x}}\phi(\tup{\tau}/\tup{x}))$ implies $\bT\vdash(\phi\seq{\tup{x}}\psi(\tup{\tau}/\tup{x}))$ by \cref{lem:basic_property_repn_model}\cref{lem:basic_property_repn_model-1}.
    Taking $\tau_i$ to be $x_i$, we have $\bT\vdash(\phi\seq{\tup{x}}\psi)$.
\end{proof}

\subsection{Local presentability of categories of partial models}
In this subsection, we establish the equivalence between partial Horn theories and locally presentable categories.
Since $\lambda$-ary partial Horn theories can equivalently be translated into ($\lambda$-ary) limit theories in the sense of \cite{adamek1994locally}, all the contents in this subsection are well-known in the literature.

\begin{proposition}\label{prop:repn_obj_represents_intpn}
    Let $\bT$ be a $\lambda$-ary partial Horn theory over an $S$-sorted $\lambda$-ary signature $\Sigma$.
    Let $\tup{x}.\phi$ be a $\lambda$-ary Horn formula over $\Sigma$.
    Then, the representing model $\repn{\tup{x}.\phi}_\bT$ represents the interpretation functor $\intpn{\tup{x}.\phi}{\bullet}\colon\PMod\bT\ni M\mapsto\intpn{\tup{x}.\phi}{M}\in\Set$,
    i.e., for every $\bT$-model $M$, we have the following natural isomorphism:
    \begin{equation*}
        \intpn{\tup{x}.\phi}{M} \cong \PMod\bT(\repn{\tup{x}.\phi}_\bT,M).
    \end{equation*}
\end{proposition}
\begin{proof}
    Take a $\bT$-model $M$ and a $\Sigma$-homomorphism $h\colon\repn{\tup{x}.\phi}_\bT\to M$.
    By induction on the structure of $\bT$-term $[\tup{x}.\tau]_\bT$, $h([\tup{x}.\tau]_\bT)$ can only be $\intpn{\tup{x}.\tau}{M}(h([\tup{x}.x_i]_\bT)_i)$.
    Conversely, given a tuple $\tup{m}\in\intpn{\tup{x}.\phi}{M}$, 
    $h([\tup{x}.\tau]_\bT)\coloneq\intpn{\tup{x}.\tau}{M}(\tup{m})$ yields a well-defined $\Sigma$-homomorphism $h\colon\repn{\tup{x}.\phi}_\bT\to M$.
    Thus, a $\Sigma$-homomorphism $h\colon\repn{\tup{x}.\phi}_\bT\to M$ bijectively corresponds to $\tup{m}\in\intpn{\tup{x}.\phi}{M}$.
\end{proof}

\begin{corollary}\label{cor:morphism_between_repn}
    Let $\bT$ be a $\lambda$-ary partial Horn theory over an $S$-sorted $\lambda$-ary signature $\Sigma$ and let $\tup{x}.\phi$ and $\tup{y}.\psi$ be $\lambda$-ary Horn formulas over $\Sigma$ with $\tup{x}=(x_i)_{i<\alpha}$.
    Then, the following data bijectively correspond to each other:
    \begin{enumerate}
        \item
        A $\Sigma$-homomorphism $h\colon\repn{\tup{x}.\phi}_\bT\to\repn{\tup{y}.\psi}_\bT$,
        \item
        Equivalence classes $([\tup{y}.\tau_i]_\bT)_{i<\alpha}$ of $\bT$-terms generated by $\tup{y}.\psi$ such that $\psi\seq{\tup{y}}\phi(\tup{\tau}/\tup{x})$ is a $\PHL_\lambda$-theorem of $\bT$.
    \end{enumerate}
\end{corollary}
\begin{proof}
    By \cref{prop:repn_obj_represents_intpn} and \cref{lem:basic_property_repn_model}\cref{lem:basic_property_repn_model-1}, we have:
    \begin{align*}
        \PMod\bT(\repn{\tup{x}.\phi}_\bT , \repn{\tup{y}.\psi}_\bT)
        &\cong \intpn{\tup{x}.\phi}{\repn{\tup{y}.\psi}_\bT}
        \\
        &= \{([\tup{y}.\tau_i]_\bT)_{i<\alpha} \mid \bT\vdash (\psi\seq{\tup{y}}\phi(\tup{\tau}/\tup{x})) \}.
    \end{align*}
    This completes the proof.
\end{proof}

\begin{notation}
    Denote by
    \[
        \repn{\tup{\tau}}_\bT\colon \repn{\tup{x}.\phi}_\bT\to\repn{\tup{y}.\psi}_\bT
    \]
    the morphism corresponding to $\bT$-terms $(\tup{y}.\tau_i)_{i<\alpha}$ by \cref{cor:morphism_between_repn}.
\end{notation}

\begin{remark}\label{rem:filtered_colim_partial_model}
    Given a $\lambda$-filtered diagram $M_\bullet \colon\bI\to\PMod\bT$, let us construct a colimit $N=\Colim{I\in\bI}M_I$ in $\PMod\bT$.
    For each sort $s\in S$ define $N_s\coloneq\Colim{I\in\bI}(M_I)_s$ as a colimit in $\Set$, i.e.,
    $N_s$ is the quotient set of the disjoint union $\coprod_{I\in\bI}(M_I)_s=\{ (I;a) \mid I\in\bI, a\in (M_I)_s \}$ by an equivalence relation $\sim_s$.
    Here $(I;a)\sim_s (J;b)$ holds if and only if there exists a cospan
    $
    \begin{tikzcd}[scriptsize]
        I\arrow[r,"p"] & K & J\arrow[l,"q"']
    \end{tikzcd}
    $
    in $\bI$ satisfying $M_p(a)=M_q(b)$.
    Let $[I;a]$ denote the equivalence class with respect to $\sim_s$ containing $(I;a)$.
    
    Now the $S$-sorted set $N=(N_s)_{s\in S}$ yields a partial $\Sigma$-structure as follows:
    \begin{itemize}
        \item
        For each function symbol $f\colon \sqcap_{i<\alpha} s_i\to s$ in $\Sigma$, $\intpn{f}{N}([I_i;a_i])_{i<\alpha}$ is defined if and only if there exists a cocone $(I_i\longarr[p_i]I)_{i<\alpha}$ such that $(M_{p_i}(a_i))_{i<\alpha}$ belongs to the domain of $\intpn{f}{M_I}$, and then define $\intpn{f}{N}([I_i;a_i])_{i<\alpha}\coloneq [I~;~\intpn{f}{M_I}(M_{p_i}(a_i))_{i<\alpha}]$;
        \item
        For each relation symbol $R\colon \sqcap_{i<\alpha}s_i$ in $\Sigma$, $([I_i;a_i])_{i<\alpha}$ belongs to $\intpn{R}{N}$ if and only if there exists a cocone $(I_i\longarr[p_i]I)_{i<\alpha}$ such that $(M_{p_i}(a_i))_{i<\alpha}$ belongs to $\intpn{R}{M_I}$.
    \end{itemize}
    Since each $M_I$ is a $\bT$-model, we see that $N$ is a $\bT$-model and a colimit of the diagram $M_\bullet$.
\end{remark}

\begin{theorem}\label{thm:PMod_is_loc_presn}
    Let $\bT$ be a $\lambda$-ary partial Horn theory over an $S$-sorted $\lambda$-ary signature $\Sigma$.
    Then, the category $\PMod\bT$ is locally $\lambda$-presentable.
\end{theorem}
\begin{proof}
    By \cref{rem:filtered_colim_partial_model}, the category $\PMod\bT$ has $\lambda$-filtered colimits.
    We first observe that all representing models $\repn{\tup{x}.\phi}_\bT$ are $\lambda$-presentable objects in $\PMod\bT$.
    By \cref{prop:repn_obj_represents_intpn}, this assertion is equivalent to saying that for every Horn formula $\tup{x}.\phi$, the interpretation functor $\intpn{\tup{x}.\phi}{\bullet}\colon\PMod\bT\to\Set$ preserves $\lambda$-filtered colimits, which follows from the construction of $\lambda$-filtered colimits in \cref{rem:filtered_colim_partial_model}.

    We next prove that every $\bT$-model $M$ is a $\lambda$-filtered colimit of representing models.
    To prove this, consider the following small category $\C_M$ for each $\bT$-model $M$:
    \begin{itemize}
        \item
        An object in $\C_M$ is a pair $(\tup{x}.\phi,\tup{a})$ of a Horn formula $\tup{x}.\phi$ over $\Sigma$ and elements $\tup{a}\in\intpn{\tup{x}.\phi}{M}$.
        \item
        A morphism $(\tup{x}.\phi,\tup{a})\to(\tup{y}.\psi,\tup{b})$ in $\C_M$ is a $\Sigma$-homomorphism $\repn{\tup{\tau}}_\bT\colon \repn{\tup{x}.\phi}_\bT\to\repn{\tup{y}.\psi}_\bT$ such that $\intpn{\tup{y}.\tau_i}{M}(\tup{b})=a_i~(\forall i)$.
    \end{itemize}
    We can prove that the small category $\C_M$ is $\lambda$-filtered as follows:
    Let $(\tup{x}_k.\phi_k,\tup{a}_k)_{k<\gamma}$ be objects in $\C_M$ with $\gamma<\lambda$.
    For simplicity, we assume that no pair of contexts $(\tup{x}_k)_k$ has a common variable.
    Then, $(\tup{x}.\bigwedge_{k<\gamma}\phi_k, (\tup{a}_k)_{k<\gamma})$ becomes the vertex of some cocone over $(\tup{x}_k.\phi_k,\tup{a}_k)_{k<\gamma}$, where $\tup{x}\coloneq\cup_{k<\gamma}\tup{x}_k$.
    Let $(\repn{\tup{\tau}_k}_\bT)_{k<\gamma} \colon (\tup{x}.\phi,\tup{a})\to (\tup{y}.\psi,\tup{b})$ be parallel morphisms in $\C_M$ with $\gamma<\lambda$.
    Then, $(\tup{y}.\psi\wedge\bigwedge_{i,k,k'}\tau_{k,i}=\tau_{k',i},\tup{b})$ becomes the vertex of some cocone over the parallel morphisms $(\repn{\tup{\tau}_k}_\bT)_{k<\gamma}$.
    Thus, $\C_M$ is a $\lambda$-filtered category.
    
    An assignment $(\tup{x}.\phi,\tup{a})\mapsto \repn{\tup{x}.\phi}_\bT$ yields a functor $D_M\colon\C_M\to\PMod\bT$.
    Moreover, by \cref{prop:repn_obj_represents_intpn}, we have a canonical cocone $\xi$ over $D_M$
    \[
        \xi_{\tup{x}.\phi,\tup{a}}\colon \repn{\tup{x}.\phi}_\bT\longarr[\tup{a}] M,
        ~~
        \text{where }(\tup{x}.\phi,\tup{a})\in\C_M.
    \]
    We now show that the cocone $\xi$ forms a colimit.
    Take an arbitrary cocone $\zeta$ over $D_M$
    \[
        \repn{\tup{x}.\phi}_\bT\longarr[\zeta_{\tup{x}.\phi,\tup{a}}] N,
        ~~
        \text{where }(\tup{x}.\phi,\tup{a})\in\C_M.
    \]
    We have to construct a canonical morphism $M\arr[h]N$ and prove its uniqueness.
    For every element $a\in M_s$ of sort $s$, the following diagram should commute:
    \begin{equation}\label{eq:definition_of_h}
        \begin{tikzcd}[huge, every label/.append style = {font = \small}]
            \repn{x\ofsort s.\top}_\bT\arrow[r,"a"]\arrow[rd,"\zeta_{x\ofsort s.\top,a}"'] & M\arrow[d,"h"] \\
            & N
        \end{tikzcd}
    \end{equation}
    This determines $h$ uniquely.
    Moreover, by the naturality of the cocone $\zeta$, \cref{eq:definition_of_h} defines the desired $\Sigma$-homomorphism $M\arr[h]N$.

    Finally, we have proved that the category $\PMod\bT$ is $\lambda$-accessible.
    Since all small limits in $\PMod\bT$ can be constructed sort-wisely in the same way as in $\Set$, by \cite[2.47 Corollary]{adamek1994locally}, we conclude that $\PMod\bT$ is locally $\lambda$-presentable.
\end{proof}

\begin{theorem}\label{thm:repn_enumerates_presn}
    Let $\bT$ be a $\lambda$-ary partial Horn theory over an $S$-sorted $\lambda$-ary signature $\Sigma$.
    Then, for any $\bT$-model $M$, the following are equivalent:
    \begin{enumerate}
        \item
        $M$ is a $\lambda$-presentable object in $\PMod\bT$.
        \item\label{thm:repn_enumerates_presn-2}
        There exists a $\lambda$-ary Horn formula $\tup{x}.\phi$ such that $M\cong\repn{\tup{x}.\phi}_\bT$ in $\PMod\bT$.
    \end{enumerate}
\end{theorem}
\begin{proof}
    By the proof of \cref{thm:PMod_is_loc_presn}, it suffices to show that the class of all representing models $\repn{\tup{x}.\phi}_\bT$ is, up to isomorphism, closed under retracts.
    Here, we prove more:
    That class is closed under coequalizers.
    Let
    \[
        \begin{tikzcd}[large]
            \repn{\tup{x}.\phi}_\bT \arrow[r,shift left=2,"\repn{\tup{\tau}}_\bT"]\arrow[r,shift right=2,"\repn{\tup{\sigma}}_\bT"'] & \repn{\tup{y}.\psi}_\bT
        \end{tikzcd}\incat{\PMod\bT}
    \]
    be parallel morphisms between representing objects with $\tup{x}=(x_i)_{i<\alpha}$.
    Considering the Horn formula $\tup{y}.\chi\coloneq\tup{y}.\left( \psi\wedge\bigwedge_{i<\alpha}\tau_i=\sigma_i \right)$, we have
    \begin{align*}
        \PMod\bT(\repn{\tup{y}.\chi}_\bT, M) &\cong \intpn{\tup{y}.\chi}{M} \\
        &= \{ \tup{a}\in\intpn{\tup{y}.\psi}{M} \mid \forall i,~ \intpn{\tup{y}.\tau_i}{M}(\tup{a})=\intpn{\tup{y}.\sigma_i}{M}(\tup{a}) \} \\
        &= \{ h\in \PMod\bT(\repn{\tup{y}.\psi}_\bT, M) \mid h\circ\repn{\tup{\tau}}_\bT=h\circ\repn{\tup{\sigma}}_\bT \}
    \end{align*}
    for every $\bT$-model $M$ by \cref{prop:repn_obj_represents_intpn}.
    This implies that $\repn{\tup{y}.\chi}_\bT$ is a coequalizer of $\repn{\tup{\tau}}_\bT$ and $\repn{\tup{\sigma}}_\bT$.
\end{proof}

\begin{theorem}\label{thm:PHT_equiv_loc_presn}
    For a category $\A$, the following are equivalent:
    \begin{enumerate}
        \item\label{thm:PHT_equiv_loc_presn-1}
        $\A$ is locally $\lambda$-presentable.
        \item\label{thm:PHT_equiv_loc_presn-2}
        There exists a $\lambda$-ary partial Horn theory such that $\A\simeq\PMod\bT$.
    \end{enumerate}
\end{theorem}
\begin{proof}
    By \cref{thm:PMod_is_loc_presn}, it suffices to prove the direction \cref{thm:PHT_equiv_loc_presn-1}$\implies$\cref{thm:PHT_equiv_loc_presn-2}.
    Let $\A$ be a locally $\lambda$-presentable category.
    By \cite[1.52 Corollary]{adamek1994locally}, we can take a $\lambda$-limit (small) sketch $\mathcal{S}$ such that $\A\simeq\Mod\mathcal{S}$.
    Without loss of generality, we can assume that every selected cone in $\mathcal{S}$ has the form of either $\lambda$-ary products or pullbacks.
    We now define a $\lambda$-ary partial Horn theory $\bT$ as follows:

    The set $S$ of sorts is defined to be the set of all objects in $\mathcal{S}$.
    The $S$-sorted $\lambda$-ary signature $\Sigma$ consists of:
    \begin{itemize}
        \item
        for each morphism $s\arr[f]s'$ in $\mathcal{S}$, a function symbol $f\colon s\to s'$;
        \item
        for each selected cone $(s\arr[p_i]s_i)_{i<\alpha}$ in $\mathcal{S}$ in the form of a product,
        a function symbol $p\colon \sqcap_{i<\alpha}s_i\to s$;
        \item
        for each selected cone
        \begin{equation}\label{eq:selected_cone_pullback}
            \begin{tikzcd}
                s\arrow[d,"q_0"']\arrow[r,"q_1"] & s_1\arrow[d,"r_1"] \\
                s_0\arrow[r,"r_0"'] & t
            \end{tikzcd}\incat{\mathcal{S}}
        \end{equation}
        in the form of a pullback, a function symbol $q\colon \sqcap_{i<2}s_i\to s$.
    \end{itemize}
    The $\lambda$-ary partial Horn theory $\bT$ over $\Sigma$ consists of the following sequents:
    \begin{enumerate}[label*=\alph{enumi})]
        \item\label{eq:sketch_presheaf1}
        for each morphism $f$ in $\mathcal{S}$, the sequent $\top\seq{x}f(x)\defined$;
        \item\label{eq:sketch_presheaf2}
        for each composable pair $s\arr[f]s'\arr[f']s''$ in $\mathcal{S}$, the sequent $\top\seq{x}f'(f(x))=(f'\circ f)(x)$;
        \item\label{eq:sketch_presheaf3}
        for each object $s\in\mathcal{S}$, the sequent $\top\seq{x}x=\id_s(x)$;
        \item\label{eq:sketch_product}
        for each selected cone $(s\arr[p_i]s_i)_{i<\alpha}$ in $\mathcal{S}$ in the form of a product, the sequents
        \begin{gather*}
            \top\seq{x}p(p_i(x))_{i<\alpha}=x,\quad
            \top\seq{\tup{x}}\bigwedge_{i<\alpha}p_i(p(\tup{x}))=x_i;
        \end{gather*}
        \item\label{eq:sketch_pullback}
        for each selected cone \cref{eq:selected_cone_pullback}, the sequents
        \begin{gather*}
            \top\seq{x}q(q_0(x),q_1(x))=x,\quad
            r_0(x_0)=r_1(x_1)\seq{x_0,x_1}\bigwedge_{i<2}q_i(q(x_0,x_1))=x_i.
        \end{gather*}
    \end{enumerate}

    Then, it follows that a partial $\bT$-model is precisely a model for the limit sketch $\mathcal{S}$.
    Indeed, the Horn sequents \cref{eq:sketch_presheaf1,eq:sketch_presheaf2,eq:sketch_presheaf3} make a $\bT$-model to be a functor $\mathcal{S}\to\Set$,
    and the Horn sequents \cref{eq:sketch_product,eq:sketch_pullback} make a $\bT$-model to send each selected cone to a limit cone.
    Now, we have $\A\simeq\Mod\mathcal{S}\simeq\PMod\bT$, which completes the proof.
\end{proof}

\subsection{Validity via orthogonality}
We now associate the validity of Horn sequents with the orthogonality, which becomes a category-theoretic treatment of validity.
A special version has already appeared in \cite{bidlingmaier2018categories}.
Moreover, all the contents in this subsection are well-known for limit theories \cite[{}5.28]{adamek1994locally}.

\begin{definition}
    Let $\C$ be a category.
    \begin{enumerate}
        \item
        An object $C\in\C$ is \emph{orthogonal} to a morphism $f\colon X\to Y$ in $\C$ if for every morphism $g\colon X\to C$, there exists a unique $\hat{g}\colon Y\to C$ satisfying $\hat{g}\circ f=g$.
        \begin{equation*}
            \begin{tikzcd}[large]
                X\arrow[d,"f"']\arrow[r,"g"] & C \\
                Y\arrow[ru,"\exists !\hat{g}"',dashed] &
            \end{tikzcd}
        \end{equation*}
        \item
        Given a class $\Lambda\subseteq\mor\C$ of morphisms, denote by $\orth{\Lambda}\subseteq\C$ the full subcategory consisting of all objects orthogonal to all morphisms in $\Lambda$.
        The full subcategory $\orth{\Lambda}$ is called an \emph{orthogonality class} of $\C$,
        and called a \emph{small-orthogonality class} if $\Lambda$ is (essentially) small.\qedhere
    \end{enumerate}
\end{definition}

\begin{proposition}\label{prop:validity_for_PHL}
    Let $\bT$ be a $\lambda$-ary partial Horn theory over an $S$-sorted $\lambda$-ary signature $\Sigma$.
    Let $\phi\seq{\tup{x}}\psi$ be a $\lambda$-ary Horn sequent.
    Then the morphism
    \[
        \repn{\tup{x}.\phi}_\bT\longarr[\repn{\tup{x}}_\bT]\repn{\tup{x}.\phi\wedge\psi}_\bT \incat{\PMod\bT}
    \]
    is an epimorphism, and for any $\bT$-model $M$, the following are equivalent:
    \begin{enumerate}
        \item
        $M\vDash (\phi\seq{\tup{x}}\psi)$.
        \item\label{prop:validity_for_PHL-2}
        $M$ is orthogonal to $\repn{\tup{x}}_\bT\colon \repn{\tup{x}.\phi}_\bT\to\repn{\tup{x}.\phi\wedge\psi}_\bT$.
    \end{enumerate}
    \begin{equation*}
        \begin{tikzcd}
            \repn{\tup{x}.\phi}_\bT\arrow[r,"\forall h"]\arrow[d,"\repn{\tup{x}}_\bT"'] & M \\
            \repn{\tup{x}.\phi\wedge\psi}_\bT\arrow[ru,"\exists! \hat{h}"',dashed] &
        \end{tikzcd}
    \end{equation*}
\end{proposition}
\begin{proof}
    \cref{prop:repn_obj_represents_intpn} ensures that $\repn{\tup{x}}_\bT$ is an epimorphism; hence the uniqueness of $\hat{h}$ in \cref{prop:validity_for_PHL-2} always holds.
    The existence of $\hat{h}$ for every $h$ is clearly equivalent to $\intpn{\tup{x}.\phi}{M}\subseteq\intpn{\tup{x}.\phi\wedge\psi}{M}$.
\end{proof}

By \cref{prop:validity_for_PHL}, we get the following corollary:

\begin{corollary}
    Let $\bT_0\subseteq\bT_1$ be $\lambda$-ary partial Horn theories over an $S$-sorted $\lambda$-ary signature $\Sigma$.
    Then, the full subcategory $\PMod\bT_1\subseteq\PMod\bT_0$ is a small-orthogonality class.
\end{corollary}

\subsection{Theory morphisms}
In this subsection, we discuss morphisms between partial Horn theories, which contain morphisms of relative algebraic theories (\cref{def:morphism_rat}) as their special case.
We will denote by $(S,\Sigma,\bT)$ a $\lambda$-ary partial Horn theory $\bT$ over an $S$-sorted $\lambda$-ary signature $\Sigma$.

\begin{definition}\label{def:theory_morphism_pht}
    Consider two $\lambda$-ary partial Horn theories $(S,\Sigma,\bT)$ and $(S',\Sigma',\bT')$.
    A \emph{($\lambda$-ary) theory morphism}
    \[
        \rho\colon (S,\Sigma,\bT)\to(S',\Sigma',\bT')
    \]
    consists of:
    \begin{itemize}
        \item
        a map $S\ni s\mapsto s^\rho\in S'$;
        \item
        an assignment to each function symbol $f\colon \sqcap_{i<\alpha}s_i\to s$ in $\Sigma$,
        a $\lambda$-ary term $\tup{x}^\rho.f^\rho$ of sort $s^\rho$ over $\Sigma'$, where $\tup{x}^\rho=(x_i^\rho\ofsort s_i^\rho)_{i<\alpha}$;
        \item
        an assignment to each relation symbol $R\colon \sqcap_{i<\alpha}s_i$ in $\Sigma$,
        a $\lambda$-ary Horn formula $\tup{x}^\rho.R^\rho$ in $\Sigma'$
    \end{itemize}
    such that for every sequent $\phi\seq{\tup{x}}\psi$ in $\bT$, the \emph{$\rho$-translation} $\phi^\rho\seq{\tup{x}^\rho}\psi^\rho$ is a $\PHL_\lambda$-theorem of $\bT'$.
    The \emph{$\rho$-translation} $\phi^\rho$ and $\psi^\rho$ are constructed by replacing all symbols that $\phi$ and $\psi$ include by $\rho$.
\end{definition}

\begin{definition}
    Let $\rho\colon (S,\Sigma,\bT)\to(S',\Sigma',\bT')$ be a theory morphism between $\lambda$-ary partial Horn theories.
    For every $\bT'$-model $M$, we get a $\bT$-model $U^\rho M$ as follows:   
    \begin{itemize}
        \item
        For each sort $s\in S$, define $(U^\rho M)_s\coloneq M_{s^\rho}$;
        \item
        For each function symbol $f$ in $\Sigma$, define $\intpn{f}{U^\rho M}\coloneq\intpn{\tup{x}^\rho.f^\rho}{M}$;
        \item
        For each relation symbol $R$ in $\Sigma$, define $\intpn{R}{U^\rho M}\coloneq\intpn{\tup{x}^\rho.R^\rho}{M}$.
    \end{itemize}
    The assignment $M\mapsto U^\rho M$ yields a functor $U^\rho\colon\PMod\bT'\to\PMod\bT$, which is called the \emph{$\rho$-translation functor}.
\end{definition}

\begin{theorem}\label{thm:adjunction_induced_by_theory_mor}
    For a theory morphism $\rho\colon (S,\Sigma,\bT)\to(S',\Sigma',\bT')$, the translation functor $U^\rho$ has a left adjoint $F^\rho$.
    \begin{equation*}
        \begin{tikzcd}[large]
            \PMod\bT\arrow[r,"F^\rho",shift left=7pt]\arrow[r,"\perp"pos=0.5,phantom] &[10pt]\PMod\bT'\arrow[l,"U^\rho",shift left=7pt]
        \end{tikzcd}
    \end{equation*}
\end{theorem}
\begin{proof}
    By \cref{thm:PHT_equiv_loc_presn,thm:repn_enumerates_presn}, from the viewpoint of relative adjunctions, it suffices to determine where the left adjoint $F^\rho$ sends the representing models.
    Take an arbitrary Horn formula $\tup{x}.\phi$.
    For every $\bT'$-model $M$, we have the following natural isomorphism:
    \begin{equation*}
        \PMod\bT(\repn{\tup{x}.\phi}_\bT, U^\rho M)
        \cong \intpn{\tup{x}.\phi}{U^\rho M}
        = \intpn{\tup{x}^\rho.\phi^\rho}{M}
        \cong \PMod\bT'(\repn{\tup{x}^\rho.\phi^\rho}_{\bT'}, M)
    \end{equation*}
    This leads us to define $F^\rho \repn{\tup{x}.\phi}_\bT\coloneq\repn{\tup{x}^\rho.\phi^\rho}_{\bT'}$.
\end{proof}

\printbibliography

\end{document}

%% file: packages.tex
\usepackage{amsthm,amsfonts}
\usepackage{amssymb}
\usepackage{stmaryrd}
\usepackage{mathtools}
\usepackage{mathrsfs} 
\usepackage[bbgreekl]{mathbbol} 
\usepackage{tikz-cd}
\usepackage{thmtools} 
\usepackage{enumitem}
\usepackage{proof}

\DeclareSymbolFontAlphabet{\mathbb}{AMSb}
\DeclareSymbolFontAlphabet{\mathbbl}{bbold}

\renewcommand{\epsilon}{\varepsilon}
\renewcommand{\phi}{\varphi}

\setlist{nosep}
\setlist*[enumerate]{label*=(\roman*)}

\makeatletter
\newcommand{\myitem}[1]{%
\item[\text{#1}]\protected@edef\@currentlabel{\text{#1}}%
}
\makeatother

\usepackage{comment}
\usepackage[backend=biber,
    style=alphabetic,
    sorting=nyt,
    giveninits=true,
    backref=true,
]{biblatex}
\definecolor{darkblue}{rgb}{0.2,0,0.6}
\definecolor{darkgreen}{rgb}{0.2,0.5,0.2}
\usepackage[%
    colorlinks=true,%
    linkcolor=darkblue,%
    citecolor=green,%
    urlcolor=darkgreen,%
    hypertexnames=false,%
    naturalnames=true,
]{hyperref}
\usepackage[
    nameinlink,%
]{cleveref}

\crefname{enumi}{}{}
\crefname{equation}{}{}
\crefname{section}{Section}{Sections}
\crefname{appendix}{Appendix}{Appendices}

\declaretheorem[%
    style=definition,%
    numberwithin=section,%
    qed=$\lrcorner$,%
]{definition}
\declaretheorem[%
    style=definition,%
    numberwithin=section,%
    sibling=definition,%
    qed=$\lrcorner$,%
]{remark,example,notation}
\declaretheorem[%
    style=definition,%
    numberwithin=section,%
    sibling=definition,%
]{theorem,lemma,corollary,proposition}

\crefname{theorem}{Theorem}{Theorems}
\crefname{definition}{Definition}{Definitions}
\crefname{lemma}{Lemma}{Lemmas}
\crefname{corollary}{Corollary}{Corollaries}
\crefname{proposition}{Proposition}{Propositions}
\crefname{remark}{Remark}{Remarks}
\crefname{example}{Example}{Examples}
\crefname{notation}{Notation}{Notations}

\tikzstyle{huge} = [row sep={3.6em}, column sep={3.6em}]
\tikzstyle{large} = [row sep={2.7em}, column sep={2.7em}]
\tikzstyle{normal} = [row sep={1.8em}, column sep={1.8em}]
\tikzstyle{scriptsize} = [row sep={1.35em}, column sep={1.35em}]
\tikzstyle{small} = [row sep={0.9em}, column sep={0.9em}]
\tikzstyle{tiny} = [row sep={0.45em}, column sep={0.45em}]

\usepackage{xpatch}
\makeatletter   
\xpatchcmd{\@tocline}
{\hfil\hbox to\@pnumwidth{\@tocpagenum{#7}}\par}
{\ifnum#1<2\hfill\else\dotfill\fi\hbox to\@pnumwidth{\@tocpagenum{#7}}\par}
{}{}
\def\l@subsection{\@tocline{2}{0pt}{2pc}{6pc}{}}
\makeatother    

%% file: newcommands.tex
\renewcommand{\-}{\mathchar`-}

\newcommand{\A}{\mathscr{A}}

\newcommand{\C}{\mathscr{C}}
\newcommand{\D}{\mathscr{D}}
\newcommand{\E}{\mathscr{E}}

\newcommand{\M}{\mathscr{M}}

\newcommand{\X}{\mathscr{X}}

\newcommand{\bA}{\mathbb{A}}
\newcommand{\bB}{\mathbb{B}}
\newcommand{\bC}{\mathbb{C}}

\newcommand{\bF}{\mathbb{F}}

\newcommand{\bI}{\mathbb{I}}

\newcommand{\bN}{\mathbb{N}}

\newcommand{\bR}{\mathbb{R}}
\newcommand{\bS}{\mathbb{S}}
\newcommand{\bT}{\mathbb{T}}

\newcommand{\bZ}{\mathbb{Z}}

\newcommand{\bfE}{\mathbf{E}}
\newcommand{\bfM}{\mathbf{M}}

\newcommand{\Set}{\mathbf{Set}}
\newcommand{\Cat}{\mathbf{Cat}}
\newcommand{\CAT}{\mathbf{CAT}}
\newcommand{\Mon}{\mathbf{Mon}}
\newcommand{\Grp}{\mathbf{Grp}}

\newcommand{\Pos}{\mathbf{Pos}}
\newcommand{\POS}{\mathbf{POS}}

\newcommand{\Ban}{\mathbf{Ban}}
\newcommand{\Fuz}{\mathbf{Fuz}_{[0,\infty)}}
\newcommand{\pMet}{{\mathbf{Met}_{*}}}
\newcommand{\gMet}{{\mathbf{Met}_{\infty}}}

\newcommand{\2}{\mathbbl{2}}

\newcommand{\op}{\mathrm{op}}
\newcommand{\ob}{\mathrm{ob}}
\newcommand{\mor}{\mathrm{mor}}

\newcommand{\dom}{\mathop{\mathrm{dom}}}
\newcommand{\cod}{\mathop{\mathrm{cod}}}

\newcommand{\pto}{\rightharpoonup} 

\newcommand{\Id}{\mathrm{Id}}
\newcommand{\id}{\mathrm{id}}

\newcommand{\Colim}[1]{\mathop{\underset{#1}{\mathrm{Colim}}}}

\newcommand{\incat}[1]{~~\text{in }{#1}}

\newcommand{\tup}[1]{\vec{#1}}

\newcommand{\copower}{\cdot}

\newcommand{\arr}[1][]{\mathrel{
\tikz\draw[->] (0,0) -- node[above=1.8pt,inner sep=0pt] {\small$#1$} (0.5,0);
}}
\newcommand{\longarr}[1][]{\mathrel{
\tikz\draw[->] (0,0) -- node[above=1.8pt,inner sep=0pt] {\small$#1$} (1,0);
}}

\newcommand{\Mod}{\mathop{\mathbf{Mod}}}

\newcommand{\presn}[2]{{#1}_{#2\mathrm{p}}}
\newcommand{\lpresn}[1]{\presn{#1}{\lambda}}

\newcommand{\intpn}[2]{\mathord{\left\llbracket #1 \right\rrbracket_{#2}}} 
\newcommand{\defined}{\mathord{\downarrow}}
\newcommand{\seq}[1]{\mathrel{
\tikz\draw[|-] (0,0) -- node[above=1.8pt,inner sep=0pt] {\small$#1$} (1,0);
}}
\newcommand{\biseq}[1]{\mathrel{
\tikz\draw[|-|] (0,0) -- node[above=1.8pt,inner sep=0pt] {\small$#1$} (1,0);
}}
\newcommand{\longseq}[1]{\mathrel{
\tikz\draw[|-] (0,0) -- node[above=1.8pt,inner sep=0pt] {\small$#1$} (1.5,0);
}}

\newcommand{\PStr}{\mathop{\mathbf{PStr}}}
\newcommand{\PMod}{\mathop{\mathbf{PMod}}}
\newcommand{\Term}{\mathrm{Term}}
\newcommand{\repn}[1]{\mathord{\left\langle #1 \right\rangle}} 
\newcommand{\ofsort}{\mathord{:}}
\newcommand{\PHL}{\mathrm{PHL}}

\newcommand{\mon}{\mathrm{mon}}
\newcommand{\pos}{\mathrm{pos}}
\newcommand{\quiv}{\mathrm{quiv}}
\newcommand{\cat}{\mathrm{cat}}
\newcommand{\rsrel}{\mathrm{rsrel}}
\newcommand{\mgset}{\mathrm{mgset}}

\newcommand{\met}{\mathrm{met}}

\newcommand{\fuz}{\mathrm{fuz}}
\newcommand{\pmet}{{\mathrm{met}{*}}}

\newcommand{\sort}{\mathord{*}}

\newcommand{\ar}{\mathrm{ar}}
\newcommand{\outsort}{\mathrm{sort}}
\newcommand{\Alg}{\mathop{\mathbf{Alg}}}
\newcommand{\funcAlg}{\mathop{\mathrm{Alg}}}
\newcommand{\pht}[2]{{\bT_{#1,#2}}} 
\newcommand{\mnd}[2]{{T_{#1,#2}}} 
\newcommand{\Th}{\mathbf{Th}}
\newcommand{\Thl}{\mathbf{Th}_{\lambda}}
\newcommand{\Mndf}{\mathop{\mathbf{Mnd}_{\mathrm{f}}}}
\newcommand{\Mndl}{\mathop{\mathbf{Mnd}_{\lambda}}}
\newcommand{\EM}{\mathop{\mathbf{EM}}}
\newcommand{\funcEM}{\mathop{\mathrm{EM}}}

\newcommand{\cloprod}{\mathbf{P}}
\newcommand{\clocsub}[1]{\mathbf{S_{#1}}}
\newcommand{\clolret}[2]{\mathbf{H_{#1}^{#2}}}

\newcommand{\orth}[1]{{#1}^\bot} 
\newcommand{\rorth}{\mathbf{r}}
\newcommand{\lorth}{\mathbf{l}}
